\documentclass[11pt]{amsart}
\usepackage{amsmath}
\usepackage{amsfonts}
\usepackage{amssymb}
\usepackage{graphicx}
\usepackage{hyperref}
\usepackage{mathrsfs}
\usepackage{pb-diagram}
\usepackage{epstopdf}
\usepackage{amsmath,amsfonts,amsthm,enumerate,amscd,latexsym,curves}
\usepackage{bbm}
\usepackage[mathscr]{eucal}
\usepackage{epsfig,epsf,epic}
\usepackage{caption}
\usepackage{subcaption}
\usepackage{multirow}
\usepackage{color}
\usepackage[all]{xy}
\usepackage{fourier}
\usepackage{tikz-cd}
\usepackage{calc}
\usepackage{float}
\usepackage{comment}

\graphicspath{ {images/} }

\newtheorem{thm}{Theorem}[section]
\newtheorem{lemma}[thm]{Lemma}

\newtheorem{cor}[thm]{Corollary}
\newtheorem{prop}[thm]{Proposition}
\newtheorem{defn}[thm]{Definition}
\newtheorem{example}[thm]{Example}
\newtheorem{remark}[thm]{Remark}
\newtheorem{conjecture}[thm]{Conjecture}

\numberwithin{equation}{section}

\newcommand{\AI}{A_\infty}

\newcommand{\Hom}{{\rm Hom}}

\newcommand{\Z}{\mathbb{Z}}
\newcommand{\R}{\mathbb{R}}
\newcommand{\C}{\mathbb{C}}

\newcommand{\WF}{\mathcal{WF}}

\begin{document}
\title{Floer theory for the variation operator of an isolated singularity}
\author[Bae]{Hanwool Bae}
\address{Center for Quantum Structures in Modules and Spaces, Seoul National University, Seoul 08826, Republic of Korea}
\email{hanwoolb@gmail.com}
\author[Cho]{Cheol-Hyun Cho}
\address{Department of Mathematical Sciences, Research Institute in Mathematics\\ Seoul National University\\ Seoul \\ Republic of Korea}
\email{chocheol@snu.ac.kr}
\author[Choa]{Dongwook Choa}
\address{Korea Institute for Advanced Studies \\ Dongdaemoon-gu\\Seoul 02455\\Republic of Korea }
\address{Current address: Chungbuk National University \\ Seowon-gu\\Cheongju 28644\\Republic of Korea }
\email{dwchoa@chungbuk.ac.kr}
\author[Jeong]{Wonbo Jeong}
\address{Department of Mathematical Sciences, Research Institute in Mathematics\\ Seoul National University\\ Gwanak-gu\\Seoul \\ Republic of Korea}
\email{wonbo.jeong@gmail.com}

\begin{abstract}

The variation operator in singularity theory maps relative homology cycles to compact cycles in the Milnor fiber using the monodromy.
We construct its symplectic analogue for an isolated singularity.  
We define the monodromy Lagrangian Floer cohomology, which provides categorifications of the standard theorems on the variation operator and the Seifert form.
The key ingredients are a special class $\Gamma$ in the symplectic cohomology of the inverse of the monodromy and its closed-open images.
For isolated plane curve singularities whose A'Campo divide has depth zero, we find an exceptional collection consisting of non-compact Lagrangians in the Milnor fiber corresponding to a distinguished collection of vanishing cycles under the variation operator.
\end{abstract}

\maketitle
\tableofcontents

\section{Introduction}

The purpose of this paper is to construct a Floer-theoretic operator $\mathcal {V}$ and a cohomology $HF^*_\rho$ that are analogues of the variation operator and the Seifert pairing associated with an isolated singularity $f$. 

Let $f:(\C^n,0) \to (\C,0)$ be a germ of an isolated hypersurface singularity.
The Milnor fiber $M$ of $f$ is a smooth manifold, given by the intersection $f^{-1}(\delta) \cap \mathbb B_\epsilon(0)$ for sufficiently small real numbers $\delta$ and $\epsilon$.  The geometric monodromy $\rho$ of the singularity $f$ is a diffeomorphism $\rho:(M,\partial M) \to (M,\partial M)$ that is the identity on $\partial M$ and is obtained by moving counter-clockwise along the fiber over the circle of radius $\delta$ in the base $\C$.  

The variation operator
\begin{equation}\label{eq:var}
\mathrm{var}:H_{n-1}(M,\partial M) \to H_{n-1}(M)
\end{equation}
records the effect of monodromy on the middle-dimensional homology group.
Namely, it takes a cycle relative to the boundary and sends it to the difference between its monodromy image and itself:
$$ \mathrm{var} (l )  :=   \rho_*(l) - l  \in H_{n-1}(M).$$
The variation operator is known to be an isomorphism, and it is a fundamental tool of singularity theory.
For example, the Seifert pairing
$$S: H_{n-1}(M;\Z) \otimes H_{n-1}(M;\Z) \to \Z$$ 
which is a non-degenerate bilinear form defined as the linking number of cycles, may be understood in terms of the Poincare-Lefschetz duality of $H_{n-1}(M,\partial M)$ and  $H_{n-1}(M)$ via the variation operator.
 We refer readers to the book \cite{AGV2} for more details.

We first construct a symplectic analogue $\mathcal V$, called {\em variation}, as an operator on Lagrangians. 
We also construct a cohomology theory denoted by $HF^*_\rho(L_1, L_2)$, called {\em monodromy Lagrangian Floer cohomology}, which categorifies classical theorems involving the Seifert form, the intersection pairing, and the variation operator.

Before we explain our results, let us briefly give some historical background on symplectic study of singularities.
Arnold \cite{Ar} observed that for a complex Morse singularity, its monodromy is a generalized Dehn twist, which can be defined as a symplectomorphism.
Seidel systematically developed a symplectic version of Picard Lefschetz theory of a singularity \cite{SVM}, \cite{Sei01}, \cite{Seilong}, \cite{S08}, and it has been one of the most influential constructions in symplectic topology, mirror symmetry, and related fields.
 
Given a singularity $f$, one may perturb it to a complex Morse function $\tilde{f}$. 
Then by choosing a generic regular value and a collection of embedded paths from the chosen generic value to critical values, symplectic parallel transports define a distinguished collection of Lagrangian vanishing cycles.
The directed Fukaya category of such vanishing cycles, or rather its derived version, is called {\em Fukaya-Seidel} category of $f$. 
It is independent of Morsifications, choices of regular value and embedded paths, and various choices made in the construction. 
This category turned out to be a mirror of the derived category of coherent sheaves on a Fano manifold or a graded matrix factorization category of Landau-Ginzburg model.
See Futaki-Ueda\cite{FU09}, Ebeling-Takahashi \cite{ET}, Lekili-Ueda\cite{LU20}, Harbermann-Smith\cite{HS19}, Polishchuk-Varolgunes \cite{PoUm}, and the references therein for computation of the Fukaya-Seidel category and associated mirror correspondences.

There are also recent constructions of Fukaya-Seidel category of Lefschetz fibrations by Abouzaid-Smith \cite{ASm}, Seidel \cite{Se18}, Sylvan \cite{Sy19I}, Ganatra-Pardon-Shende \cite{GPS20}. In these constructions, non-compact Lagrangians in
the total space of the Lefschetz fibration (which are compact when restricted to each fiber) are taken as objects of the partially wrapped version of Fukaya category. For example, instead of a Lagrangian vanishing cycle of the fiber, its Lefschetz thimble
as a Lagrangian in the total space is considered.

Let us go back to the variation operator \eqref{eq:var} and look at its domain. 
It is reasonable to consider Lagrangian submanifolds corresponding to homology classes in $H_{n-1}(M,\partial M)$.
Therefore, a natural starting point is the {\em wrapped Fukaya category} $\mathcal{WF}(M)$, which has been developed by Abouzaid and Seidel \cite{AS} to deal with such a situation. 
Recall that the Milnor fiber $M$ has the canonical structure of a Liouville manifold with a contact boundary.
Roughly, objects of $\mathcal{WF}(M)$ are Lagrangian submanifolds $L$ with possibly a Legendrian boundary $\partial L \subset \partial M$.
We take a completion of $M$ by attaching an infinite cone $[1,\infty) \times \partial M$, extend $L$ by attaching $[1,\infty) \times \partial L$ and then regard the wrapped Floer cohomology $HW^*(L_1, L_2)$ as morphism spaces between objects. 
It can be viewed as the Lagrangian Floer cohomology whose Hamiltonian perturbation term is regulated to be quadratic on $[1,\infty)$ coordinate so that its flow is proportional to the Reeb flow on $\partial M$.

The first task is to extend the variation homomorphism into the context of Floer theory. 
The naive difference $\rho(L) - L$ ( = $\rho(L) \oplus L[1]$ ) is not a suitable candidate because its wrapped Floer cohomology is usually infinite-dimensional.
In contrast, the image of the variation operator is a cycle in $H_{n-1}(M)$. 
Instead of the naive difference, we would like to find a morphism from $L$ to $\rho(L)$ whose mapping cone represents a proper object. 
This leads us to the study of the Floer cohomology of $\rho$ itself. 

The monodromy $\rho:M \to M$ of an isolated singularity $f$ can be taken as a symplectic automorphism.
One approach is to use the Milnor fibration $f/|f|: S^{2n-1}(\delta)\setminus f^{-1}(0) \to S^1$.
Caubel, N{\'e}methi and Popescu-Pampu constructed the Milnor open book which gives rise to a pair
of a Liouville manifold and a compactly supported exact symplectomorphism, called abstract contact open book (see \cite{giroux03g}, \cite{CNPP}, \cite{Mc19}). 

On the other hand,  A'Campo \cite{AC75} computed the Lefschetz numbers of the iterates of the topological monodromy map.
In fact, A'Campo showed that monodromy has very special dynamics via the study about embedded resolutions of singularities. For example,
the Lefschetz number of a monodromy map of an isolated singularity is always zero, and there exists a topological representative of the monodromy, whose fixed point set is  (a collar of)  the boundary of the Milnor fiber. 

Recently, Fernandez de Bobabilla and  Pe{\l}ka \cite{BP} have constructed the so-called {\em symplectic monodromy at radius zero}, which gives
a Liouville manifold $M$ together with an exact symplectomorphism $\rho: M \to M$ which has the same dynamics as the A'Campo's topological monodromy map. Moreover, using a small Hamiltonian flow near the boundary, one can choose such an exact symplectomorphism $\rho$ so that its fixed points are exactly the boundary of the Milnor fiber. 

We use this representative $\rho$ and investigate its Hamiltonian Floer theory.
In Section \ref{sec:symplectic cohomology}, we define the {\em symplectic cohomology of $\phi$}, denoted by $SH^*(\phi)$, for an exact symplectomorphism $\phi$ following Uljarevic \cite{Ulj}.
Roughly, it is a Hamiltonian Floer cohomology whose generators are twisted Hamiltonian orbits.
This is an extension of fixed point Floer cohomology $HF^*(\phi,\pm)$ defined by Seidel \cite{Sei01}, where the $\pm$ sign refers to a small Hamiltonian perturbation at the boundary $\partial M$ using the Reeb flow along positive/negative direction.
Our construction uses a quadratic Hamiltonian at infinity. It can be shown that our construction is canonically quasi-isomorphic to that of Uljarevic \cite{Ulj}, which uses linear Hamiltonians and takes the direct limit as the slopes of these Hamiltonians go to infinity.

In Section \ref{sec:Gamma}, we define the monodromy class $\Gamma$.
\begin{defn}[see Definition \ref{defn:gamma}]
There exists a cohomology class $\Gamma \in SH^0(\rho^{-1})$, called the {\em monodromy class},
representing the fundamental class of a $\partial M$-family of fixed points of  $\rho^{-1}$.
\end{defn}
We observe that the set of fixed points of $\rho^{-1}$ is given by $\partial M$, after a perturbation in the positive Reeb direction near the cylindrical end. 
We remark that the presence of $\Gamma$ is closely related to the vanishing of $HF^*(\rho,+)$ and  $HF^*(\rho^{-1},-) $. 
See \cite{SeDehn} for the $A_1$-case, \cite{Gau} for plane curve singularities, and \cite{Mc19}, \cite{BP} for general isolated singularities.  

After introducing $\Gamma$, we use structural maps in Floer theory to analyze its action on Lagrangian Floer theory.
In Section \ref{sec:symplectic var}, we consider the image of $\Gamma$ under the twisted version of the closed-open map,
\begin{equation*}
\mathcal{CO}_L^\rho (\Gamma) \in  CW^*\left(L, \rho(L)\right)
\end{equation*}
and define an analogue of the classical variation operator \eqref{eq:var} in Lagrangian Floer theory.

\begin{defn}(See Definition \ref{defn: variation} and \ref{defn: variation functor})
\label{defn:v}
We define the {\em variation} of $L$ to be the twisted complex
\begin{equation*}
\mathcal V(L):= \left(L[1] \xrightarrow{\mathcal{CO}_L^\rho(\Gamma)} \rho(L)\right) \in \mathcal{WF}(M).
\end{equation*}
Moreover, the {\em variation operator} extends to an endofunctor on $\mathcal{WF}(M)$ defined by 
	\[\mathcal{V}  = \mathrm{Cone}(\mathrm{Id} \xrightarrow {\mathcal {CO}^\rho(\Gamma)}\rho_\sharp): \WF(M) \to \WF(M).\]
\end{defn}
When $L$ has cylindrical ends, we also define an analogous cochain $\Gamma_L \in CW^*(L, \rho(L))$, which turns out to be the lowest action value term of $\mathcal{CO}_L^\rho(\Gamma)$. 
Moreover, we prove the following property of $\mathcal V(L)$:

\begin{figure}[h]
\begin{subfigure}[t]{0.45\textwidth}
\includegraphics[scale=1]{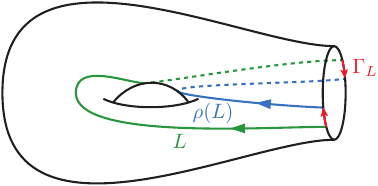}
\centering
\caption{$\Gamma_{L}$ in a curve case}
\end{subfigure}
\begin{subfigure}[t]{0.45\textwidth}
\includegraphics[scale=1]{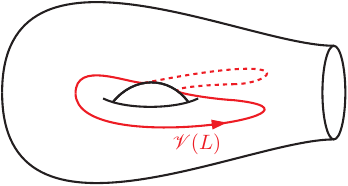}
\centering
\caption{Representative $V_L$ of $\mathcal V(L)$}
\end{subfigure}
\centering
\caption{ }
\label{fig:A2ex}
\end{figure}

\begin{thm}(See Corollary \ref{cor:VLisproper})
\label{thm:1}
For every object $L$ of $\mathcal{WF}(M)$, $\mathcal V(L)$ is a proper object, i.e., for any $L'$, both $HW^*\left(\mathcal V(L), L'\right)$ and $HW^*\left(L', \mathcal V(L)\right)$ are finite-dimensional. 
\end{thm}
It is natural to ask whether there is a compact closed Lagrangian representative for $\mathcal V(L)$.
One can construct a Lagrangian $L\sharp_{\partial L}\rho(L)$ by gluing $L$ and $\rho(L)$ along $\partial L$ via flow surgery \cite{MW18}. 
We consider its compact component $V_L$. See Figure \ref{fig:A2ex} for an illustration of the $n=2$ case.
\begin{thm}(See Proposition \ref{prop: variation dim 2} and Theorem \ref{thm:compactrepresentative})
	Assume that 
	\begin{enumerate} 
		\item $\partial L$ is connected for $n\neq 2$, 
		\item $L\cap \rho(L)$ is empty away from the boundary, and
		\item $HW^*(L, L)$ or $HW^*\left(L, \rho(L)\right)$ is infinite dimensional. 
	\end{enumerate}
Then $V_L$ is quasi-isomorphic to $\mathcal V(L)$ inside $\mathcal{WF}(M)$.
\end{thm}

In Section \ref{sec:HFrho}, we define the following.

\begin{defn}
The  {\em monodromy Lagrangian Floer cochain complex} $CF^*_\rho(L_1, L_2)$  is  defined  as a cone of 
\begin{equation}\label{eq:qvar1}
	\cap \Gamma : CW^*(\rho(L_1),L_2) \to CW^*(L_1,L_2),
\end{equation}
where $\cap$ denotes the twisted quantum cap action. 
We denote its cohomology group by $HF_\rho^*(L_1,L_2)$.
\end{defn}
The twisted quantum cap action is defined by considering pseudo-holomorphic discs with two boundary marked points and
one interior marked point, where the interior point lies on the geodesic connecting the two boundary points. We consider an insertion of $\Gamma$ at the interior marked point.

We prove that monodromy Lagrangian Floer cohomology categorifies the Seifert form. 

\begin{thm}[See Theorem \ref{thm: HF is finite}]
\label{thm:2}
Let $L_1,L_2$ be objects in wrapped Fukaya category of $M$.
Then we have
\begin{equation*}
\chi \big(HF^*_\rho(L_1,L_2) \big) = (-1)^nS(v_2,v_1)
\end{equation*}
where $v_i$ are the homology class of $\mathcal{V}(L_i)$ in $H_{n-1}(M)$, and $S(-,-)$ is the Seifert pairing on $H_{n-1}(M)$.
\end{thm}

Moreover, we show that the following classical theorems regarding relations between variation operator, Seifert form, and intersection pairing are categorified.
\begin{thm} [Arnold, Gussein-Zade, Varchenko \cite{AGV2}]
Let $(-\bullet -)$ and $S(-,-)$ denote the intersection pairing and the Seifert pairing on $H_{n-1}(M)$, respectively. For $v_1, v_2 \in H_{n-1}(M)$, the following hold:
\begin{enumerate}
\item  $S(v_1, v_2) = \mathrm{var}^{-1}(v_1) \bullet v_2$.
\item  $ \mathrm{var}^{-1}(v_1)  \bullet v_2 + v_1 \bullet \rho ( \mathrm{var}^{-1}(v_2))=0$.
\item  $ v_1 \bullet v_2 = - S(v_1, v_2) +(-1)^{n}S(v_2,v_1)$.
\end{enumerate}
\end{thm}

\begin{thm}[See Theorem \ref{prop:HF and V} and \ref{prop:Serre dual}]
\label{thm:var}
For every pair of objects $L_i$, $i=1,2$, we have the following.
\begin{enumerate}
\item $HF_\rho^*(L_1,L_2)  \cong HW^*(\mathcal{V}(L_1)[-1], L_2) \cong HW^{*}(\rho(L_1),\mathcal{V}(L_2))$. 
\item There is a canonical duality isomorphism 
	\[HF_\rho^*(L_1,L_2) \cong HF_\rho^{n-2-*}(L_2, \rho(L_1))^\vee.\]
\item There is a long exact sequence of cohomologies
	\[
	\begin{tikzcd}[column sep=small]
	\cdots \arrow[r] & HF_\rho^{n-1-*}(L_2,L_1)^\vee \arrow[r] & 
	HW^*\left(\mathcal V(L_1), \mathcal V(L_2)\right) \arrow[r] & HF_\rho^*(L_1,L_2) \arrow[r] &
	HF_\rho^{n-2-*}(L_2,L_1)^\vee \arrow[r] & \cdots.
	\end{tikzcd}
	\]
\end{enumerate}
\end{thm}

Theorem \ref{thm:var} follows from the fact that $\mathcal{V}$ is defined via the closed-open map of $\Gamma$.
The essence of the proof is the standard cobordism of moving interior $\Gamma$-insertion toward the either boundary of the Floer strip.   
Observe that the statement (2) is an analogue of Seidel's observation that "winding once around the base" is the Serre functor for the Fukaya-Seidel category up to shift. 

\begin{remark}
We remark that from the construction of Fukaya-Seidel category, the directed Floer cohomology groups
between distinguished collection of Lagrangian vanishing cycles naturally provide a categorification of the Seifert form on this distinguished basis, which naturally extends to the Grothendieck group of the Fukaya-Seidel category 
(cf. \cite{Va23}).
\end{remark}

In Section \ref{sec:distinguished}, we investigate another property of $HF^*_\rho$  to effectively compare it with the Fukaya-Seidel category. 
First note that we do not use any Morsification of the singularity $\rho$ to define $HF^*_\rho$ (and thus we do not
choose vanishing paths). Moreover, our objects are not ordered, nor is the definition of our cohomology directed.
Nevertheless, we still expect that our cohomology is related to the derived Fukaya-Seidel category; in this relation, a choice of Morsification and a set of vanishing paths would correspond to the choice of an exceptional collection in terms of $HF^*_\rho$.
\begin{defn}\label{defn:excep}
An ordered collection of Lagrangians $(L_1,\cdots,L_m)$ which satisfies the following properties is called an {\em exceptional} collection:
\begin{enumerate}
\item $HF^*_\rho(L_i,L_i) = \mathbb K \cdot e_i [0]$ and 
\item $HF^*_\rho(L_i,L_j) = 0$ if $i>j$. 
\end{enumerate}
 \end{defn}
 
Given an exceptional collection, we can relate their monodromy Lagrangian Floer cohomology to the standard Floer cohomology of variation images using Theorem \ref{thm:var} (3).
\begin{thm}[See Corollary \ref{cor:HF and FS}]
\label{thm:distinguished}
For an exceptional collection $(L_1,\cdots,L_m)$, we have
	\[
	HF_\rho^*(L_i, L_j) \cong \left\{
	\begin{array}{ll}
	0 & (i>j),\\
	\mathbb K \cdot e_i[0] & (i=j),\\
	HW^*(\mathcal{V}(L_i), \mathcal{V}(L_j))& (i<j).\\
	\end{array}
	\right.
	 \]
\end{thm}
This formula is interesting for two reasons. 
First, the formula shows that $HF^*_\rho$ is not symmetric in its arguments, as seen in Theorem \ref{thm:var}-(1). This is different from the usual Floer cohomology $HF^*$, which is symmetric due to the Poincar\'e duality isomorphism
\[HF^*(V_i, V_j) \cong HF^{n-1-*}(V_j, V_i)^\vee.\] 
Moreover, this formula reminds us of the very definition of the Fukaya-Seidel category.
The morphisms between distinguished collection of vanishing cycles are defined to be zero in an opposite ordering, which reflects the fact that $\hom^*_{FS}(V_i, V_j)$ is a Floer theory between thimbles of $V_i$ rather than $V_i$ themselves. 
Even though the definition of $HF^*_\rho$ is not directed, for a special choice of objects, $HF^*_{\rho}$ on this collection can be directed and behaves like Fukaya-Seidel category.
If a vanishing result is known in one direction, then the monodromy Floer cohomology in the opposite direction is determined by the usual Floer cohomology via the variation $\mathcal V$.

The final question we ask is therefore the following inverse problem. 
Given a singularity $f$, a choice of Morsification $\tilde{f}$ and vanishing paths determines a distinguished collection of vanishing cycles denoted as $\{V_i \}_{i=1}^\mu$. 
Is there an exceptional collection of non-compact Lagrangians $\{K_i\}_{i=1}^\mu$ with respect to $HF^*_\rho$ such that $\mathcal{V}(K_i) \cong V_i$?

We begin with an easier question of topological representability.
We are looking for Lagrangian representatives $\{K_i\}$ such that each homology class $[K_i]$ goes to $[V_i]$ under 
the classical variation operator \eqref{eq:var}.
We first find a convenient notion of an adapted family, which solves the topological question in terms of intersection numbers.

\begin{defn}
An ordered family $(K_1,\cdots,K_\mu)$ is said to be {\em adapted} to $(V_1,\cdots, V_\mu)$ if it satisfies the following (signed) intersection conditions. 

\begin{enumerate}
\item For any $j > i$, we require that $K_j \bullet V_i = - ( -1)^\frac{n(n-1)}{2}V_j \bullet V_i$.
\item For any $j < i$, we require that $K_j \cap V_i = \emptyset$.
\item We require that $K_j$ intersects $V_j$ at a single point positively for each $j$.
\end{enumerate}
Note that each $K_i$ may have multiple connected components
\end{defn}

Using the Picard-Lefschetz formula, we then prove the following.
\begin{prop}[See Proposition \ref{lem:adpatedvar}]
If $(K_1,\cdots, K_\mu)$ is adapted to $(V_1,\cdots, V_\mu)$, then  we have the topological relation:
$$\mathrm{var} ([K_i]) = ( -1)^\frac{n(n-1)}{2}[V_i] \in H_{n-1}(M), \;\;\; \forall i.$$ 
\end{prop}

In Section \ref{sec:exceptionalcollection}, we will explain how our constructions work for plane curve singularities.
We show that for a plane curve singularity whose A'Campo divide has depth 0, we can find an exceptional collection consisting of non-compact connected Lagrangians as follows. 

 {\emph{A'Campo divide}} \cite{Areal}, \cite{Agene}  is a novel construction that enables one to visualize 
 a distinguished collection of vanishing cycles as well as their intersections in the Milnor fiber of (a perturbation of) a real plane curve singularity.  It produces a planar graph, called the $A\Gamma$ {\em diagram}, whose vertices correspond to vanishing cycles and edges correspond to intersections between them. 
 
 We define the notion of a depth of an $A\Gamma$ diagram (and thus a divide), by counting the minimal number of edges to reach vertices from outside (see Definition \ref{defn:depth} for the precise definition). Hence, any vertex of a depth 0 $A\Gamma$ diagram always meets the non-compact region of the complement of an $A\Gamma$ diagram in the plane. For example, ADE curve singularities have divides of depth 0.
 
 \begin{figure}[h]
\begin{subfigure}[t]{0.43\textwidth}
\includegraphics[scale=1]{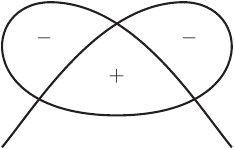}
\centering
\caption{divide of $E_{6}$ singularity}
\label{fig:E6divide}
\end{subfigure}
\begin{subfigure}[t]{0.43\textwidth}
\raisebox{3ex}{\includegraphics[scale=1]{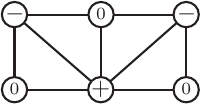}}
\centering
\caption{$A\Gamma$ diagram of $(A)$}
\label{fig:AGE6}
\end{subfigure}
\centering
\caption{Examples of divide and $A\Gamma$ diagram}
\label{fig:E6D}
\end{figure}

 \begin{thm}[See Theorem \ref{thm:ecadapted}]
 \label{thm:ec}
Let $f$ be an isolated plane curve singularity having an A'Campo divide of depth 0.
Denote by $(V_1, \ldots, V_\mu)$ the distinguished collection of Lagrangian vanishing cycles in $M$ from the given divide.
Then there exists an exceptional collection (with respect to $HF^*_\rho$) of non-compact connected Lagrangians $(K_1,\cdots, K_\mu)$  adapted to $(V_1,\cdots, V_\mu)$ satisfying
$$\mathcal{V}(K_i) \simeq V_i[1] \in \mathcal{WF}(M).$$ 
\end{thm}

Let us illustrate this in the case of $E_6$-singularity $f(x,y) = x^3+y^4$.
\begin{example}
For the case of $E_6$ singularity, there is a well-known divide given in Figure \ref{fig:E6D} (see \cite{AGV2} Chapter 4).
The Milnor fiber and its vanishing cycles are drawn in Figure \ref{fig:E6} (A).
The proof of Theorem \ref{thm:ec} is constructive and provides the ordered family of non-compact Lagrangians 
$$(K_1^-, K_2^-, K_1^0, K_2^0,K_3^0,K_1^+)$$ which is drawn in Figure \ref{fig:E6}(B). 

One can explicitly check that this ordered set is adapted and hence gives an exceptional collection.
Note from Theorem \ref{thm:distinguished} that a diagram of $HF_\rho^*$ for this family of non-compact Lagrangians is exactly the same as that of $\hom_{FS}^{*}$ of the corresponding vanishing cycles, which is an $A\Gamma$ diagram with each edge replaced by arrows toward increasing type (see diagram \ref{diagram:E6}). 
\begin{figure}[h]
\begin{subfigure}[t]{0.43\textwidth}
\includegraphics[scale=0.8]{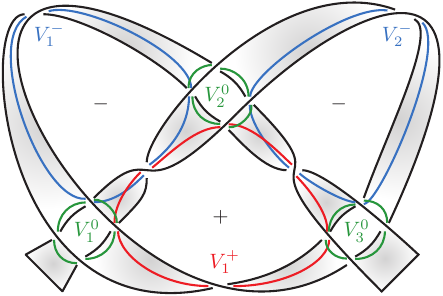}
\centering
\caption{ }
\end{subfigure}
\begin{subfigure}[t]{0.43\textwidth}
\includegraphics[scale=0.8]{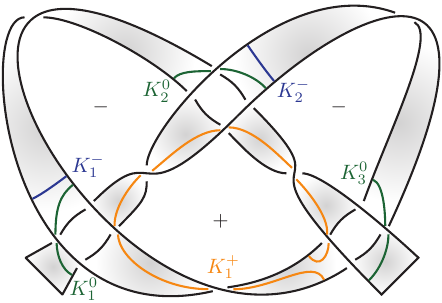}
\centering
\caption{ }
\label{fig:E6b}
\end{subfigure}
\centering
\caption{Vanishing cycles and adapted family of $E_{6}$}
\label{fig:E6}
\end{figure}
\end{example}

To sketch what happens for divides of higher depth, we would need the following conjectural  $\AI$-category. 
\begin{conjecture} \label{conj:AI}
 There exists a $\Z$-graded $\AI$-category $\mathcal{C}_\rho(M)$ whose 
 \begin{enumerate}
 \item objects are the same as that of the wrapped Fukaya category of $M$ and
 \item morphism space between two Lagrangians $L_1, L_2$ is defined by
\begin{equation}\label{conj:eq}
\hom^*_{\mathcal C_\rho}(L_1,L_2): = CF^*_\rho(L_1, L_2) = CW^{*-1}(\rho(K_1), K_2)\epsilon \oplus CW^{*}(K_1,K_2).
\end{equation}
Here, $\epsilon$ is a formal variable of degree $-1$.
\end{enumerate}
Its $\AI$-operation $\{M_k\}_{k=1}^\infty$ is defined by counting rigid pseudo-holomorphic popsicles with $\Gamma$-insertions at each sprinkles.
\end{conjecture}
Namely, $HF^*_\rho$ is expected to be the cohomology of this $\AI$-category. 
The conjectural $\AI$-structure is defined by counting rigid popsicles with $\Gamma$-insertions following  \cite{CCJ}. However, there are obstructions to the $\AI$-structure, which we explain in more detail below.

A popsicle, introduced by Abouzaid-Seidel \cite{AS} and Seidel \cite{Se18}, is a decorated disc with a single negative boundary marking and several positive boundary markings. 
Also, it can have several positive interior markings on geodesics connecting the boundary output and a boundary input. 
Note that the quantum cap action is the counting of the simplest non-trivial pseudo-holomorphic popsicles with a single $\Gamma$-insertion.

The last three authors \cite{CCJ} applied this idea in the case when $\rho$ is the identity: e.g., a quotient orbifold of the Milnor fiber of weighted homogeneous polynomial $f$. 
In general, interior insertions on popsicles collide with each other and create sphere bubbles with an induced popsicle structure. 
(This phenomenon does not happen in the original constructions \cite{AS, Se18}, since their sprinkles can overlap each other.)
Such popsicle spheres give obstructions to having an $\AI$-structure.
 It has been shown that the obstructions vanish for degree reason when $\partial M$ carries a Reeb flow $S^1$-action of a negative index. This holds for weighted homogeneous polynomials of log Fano/Calabi-Yau type,  hence the corresponding $\AI$-category is defined (see \cite{CCJ}). 
A similar obstruction exists for Conjecture $\ref{conj:AI}$.
They consist of similar popsicle spheres but with additional branch cuts along popsicle lines given by $\rho$.
We also conjecture that such an obstruction always vanishes when $\rho$ comes from the monodromy of an isolated singularity.

It would be interesting to study mutations of exceptional collections (Definition \ref{defn:excep}) in this conjectural $\AI$-category, which
we leave to the future. 

Let us go back to the case of higher depth divides.
We conjecture that the most natural way of expressing variation inverse images is via iterated cones of $\epsilon$-morphisms; recall from \eqref{eq:qvar1} that the monodromy Floer cohomology of two Lagrangians $K_1,K_2$ comes from either Floer cochain of $K_1$ and $K_2$ or that of $\rho(K_1)$ and $K_2$. 
The latter corresponds to the component with $\epsilon$ variable in \eqref{conj:eq} and we call them $\epsilon$-morphisms for simplicity.
In the higher depth case, we expect the vanishing cycles of A'Campo divide to be represented by iterated cones (or in general twisted complexes) of
$\epsilon$-morphisms. In Subsection \ref{subsec:depth 1}, we illustrate the $\epsilon$-cone object $L$ in an example of singularity whose A'Campo divide has depth 1.
We will leave a systematic study to future work. 

Last but not least, we remark two interesting results regarding the generation of Fukaya categories.
The first one due to Seidel \cite{S08} is that for a weighted homogenous polynomial $W$ with the weight condition $\frac{1}{w_1} + \cdots + \frac{1}{w_n} \neq 1$, its vanishing cycles generate Fukaya category of its Milnor fiber.
	On the other hand, Keating \cite{Keating} exhibited a Lagrangian torus inside the Milnor fiber of $T_{p,q,r}$ singularities which are not generated by vanishing cycles.
	It is interesting to ask whether Keating's tori is of the form $\mathcal V(L)$ for some exact Lagrangian cylinder $L$. 

	Note that for a Lagrangian plane $L$, if we attach $L$ and $\rho(L)$ via the Lagrangian cobordism in Section \ref{subsec:surgery}, then we obtain a Lagrangian sphere. 
	 But if we apply the construction to a Lagrangian $L$ which is topologically a cylinder, then we obtain a Lagrangian torus (obtained by gluing the two cylinders $L$ and $\rho(L)$).
	 Thus the existence of such $L$ would imply that the essential image of $\mathcal V$ is strictly larger than the subcategory generated by vanishing cycles.  

\subsection{Relation to pertinent works}
Experts have known that the relation between the Fukaya--Seidel category (or, the wrapped Fukaya category with stops $\Lambda_f$ determined by $f$ \cite{GPS20}) and the Fukaya category of the fiber can be described via an adjoint pair $(\bigcup, \bigcap)$ called the {\em cup} and {\em cap} functors:
$$
\begin{tikzcd}
\bigcup : \WF(M) \ar[r, shift left] & \WF(\C^n, \Lambda_f) : \bigcap. \ar[l, shift left]
\end{tikzcd}
$$
(We use a larger symbol $\bigcap$ to distinguish the cap functor from the quantum cap $\cap$.) See \cite{ASm} for the construction of certain Lefschetz fibrations; \cite{AA21} for a fiberwise wrapped construction and its applications to toric hypersurfaces; \cite{Sy19II} for a description using the Viterbo functor in the setting of partially wrapped Fukaya categories; and \cite{Jeff22} for functorial properties, including its application to singular hypersurfaces.

Probably the most relevant property of $(\bigcup, \bigcap)$ to this work is the exact triangle(\cite[Chapter~2]{Jeff22}):
\[
\begin{tikzcd}
\mu \ar[r, "s"] & \mathrm{Id} \ar[r] & \bigcap \bigcup \ar[r, "+1"] & \mu[1],
\end{tikzcd}
\quad \mu = \textrm{clockwise monodromy} = \rho^{-1}.
\]
Here, $s$ is a natural transformation defined by counting pseudo-holomorphic section invariants; see \cite{S08}. We have $\rho \circ (\bigcap \bigcup L) = \mathrm{Cone}\left( \rho(s_L) : L \to \rho(L)\right)$, which is very close to the definition of $\mathcal V(L)$. Thus, establishing the relation between $\rho\circ(\bigcap\bigcup) (L)$ and $\mathcal{V}(L)$ is equivalent to proving the equality 
\begin{equation}
\label{eq: s is Gamma}
\rho(s_L)=\mathcal {CO}^\rho_L(\Gamma).
\end{equation}
While an explicit computation of $s_L$ is generally difficult, we conjecture that equality \eqref{eq: s is Gamma} holds for isolated hypersurface singularities. The invariant $s_L$ is also expressed as the closed-open image of a certain symplectic cochain $s \in SC^*(\rho)$. The specific dynamics of the monodromy $\rho$ of the Milnor fiber allows for an explicit computation of $s$. The detailed proof, which is beyond the scope of the present paper, will be presented in a subsequent work. 

Assuming \eqref{eq: s is Gamma}, the group $HF^*_\rho(L_1, L_2)$ should conicide with $\hom_{\WF(\C^n, \Lambda_f)}\left(\bigcup L_1, \bigcup L_2\right)$. We can interpret Theorem \ref{thm:var} using cup and cap functors. For example, (1) is the adjoint property of $(\bigcup, \bigcap)$:
\[
HF^*_\rho(L_1, L_2) = \Hom_{\WF(\C^n, \Lambda_f)}\left(\bigcup L_1, \bigcup L_2\right) \simeq \Hom_{\WF(M)}\left(\rho(L_1), \bigcap\bigcup \rho (L_2) \right) = HF^*\left(L_1, \mathcal V(L_2)\right).
\]

The idea of using the monodromy orbit $\Gamma$ first appeared in the previous work \cite{CCJ} of the last three authors of this paper, in the case of weighted homogeneous polynomial $W$. In this case, $\Gamma$ is an element of the symplectic cohomology (of the identity automorphism) of the Milnor fiber quotient orbifold $[M_W/G_W]$ with respect to the maximal diagonal symmetry group $G_W$. Namely, the monodromy (which does not fix the boundary) is given by the diagonal weight action, and hence the monodromy becomes trivial in the quotient $[M_W/G_W]$. This allowed us to define the Fukaya category of the Landau--Ginzburg orbifold $(W, G_W)$.

\subsection{Acknowledgements}
We would like to thank Javier Fernandez de Bobabilla, Tomasz Pe{\l}ka, Pablo Portilla Cuadrado, Nero Budur, Kaoru Ono, and Otto van Koert for very helpful discussions and their interest in our work. We also thank the anonymous reviewer for comments and suggestions that have helped improve the paper. Bae and Cho are supported by the National Research Foundation of Korea(NRF) grant funded by the Korea government(MSIT) (No.2020R1A5A1016126). Choa is supported by the KIAS individual grant (MG079401). Cho and Jeong are supported by Samsung Science and Technology Foundation under project number SSTF-BA1402-52. Jeong is supported by the National Research Foundation of Korea (NRF) grants funded by the Korean government (MIST, No. RS-2025-24803252) and through the G-LAMP program (MOE, RS-2024-00441954).

\section{Symplectic cohomology of an exact symplectomorphism} 
\label{sec:symplectic cohomology}

Let $(W,\omega=d\lambda)$ be a Liouville domain of dimension $2n$. Let $Z$ be the associated Liouville vector field, i.e., $\iota_{Z} d\lambda =\lambda$ and let $\phi_Z^t$ be the time $t$-flow of $Z$. Then there is a collar neighborhood $(0, 1] \times \partial W$ of $\partial W$, where the radial coordinate $r \in (0, 1]$ is obtained by integrating $Z$ in the sense that $r(y) = 1$ for all $y\in \partial W$ and $r( \phi_Z^t(y)) = e^t r(y) = e^t$ for all $y \in \partial W, t\in (-\infty,0]$. Therefore $(W,d\lambda)$ can be naturally extended to a non-compact symplectic manifold $(\widehat{W},\omega=d\lambda)$ by gluing the cylindrical end $[1,\infty) \times \partial W$ with $W$ along $\partial W$ as follows:
$$\widehat{W} = W \cup_{\partial W} [1,\infty) \times \partial W
\text{ and } \lambda = r \lambda|_{\partial W} \text{ on } [1,\infty) \times \partial W,$$ 
which is called the completion of $(W,d\lambda)$.
The Liouville vector field $Z$ extends to the cylindrical end by $Z = r \partial_r$. Furthermore, the restriction $\theta:= \lambda|_{\partial W}$ is a contact form. Let $R$ be the associated Reeb vector field and  $\phi_R^t$ be its time-$t$ flow. We may assume that the Liouville form $\lambda$ is generic enough so that Reeb orbits are nondegenerate as discussed in \cite[Section 2]{A10}.

We are interested in exact symplectomorphisms $\phi:W \to W$ that are given by a Reeb flow near $\partial W$ in the sense that there exist $a \in \R$ and small $0<\eta<1$ such that 
\[ \phi(r,y ) = (\mathrm{Id} \times \phi_R^a)(r,y)=(r, \phi_R^a(y)), \forall r\in [1-\eta, 1], y \in \partial W.\]
Such $\phi$ may be extended to the completion $\widehat{W}$ by $\mathrm{Id} \times \phi_R^a$ on $[1,\infty) \times \partial W$ and will be still denoted by $\phi$ for convenience. 
We will assume that $\phi$ is such an exact symplectomorphism throughout this section.

We introduce fixed point Floer cohomology $HF^*(\phi;H,J)$ for certain time-dependent Hamiltonians $H$ and time-dependent almost complex structures $J$. In particular, when $H$ is quadratic at infinity, the corresponding fixed point Floer cohomology will be called symplectic cohomology of $\phi$ and will be denoted by $SH^*(\phi)$, see Subsection \ref{subsubsection:symplectic cohomology}.

Especially when $a=0$, $\phi$ equals the identity map near $\partial W$. In \cite{Ulj}, the author introduced the fixed point Floer cohomology $HF^*(\phi, b)$ for such an exact symplectomorphism $\phi$ and $b \in \mathbb{R}$ that is not a period of any Reeb orbits, which is the fixed point Floer cohomology $HF^*(\phi;H_b,J)$ for a linear Hamiltonian $H_b$ with slope $b$. Furthermore, they defined $HF^*(\phi, \infty)$ by the direct limit of $HF^*(\phi,b)$ with respect to the continuation maps as $b$ goes to infinity. We will see that these definitions can be extended to any exact symplectomorphism $\phi$ that is given by a Reeb flow near $\partial W$ and then in Subsection \ref{subsubsection:symplectic cohomology}, it will be verified that the symplectic cohomology $SH^*(\phi)$ is canonically isomorphic to $HF^*(\phi,\infty)$.\\ 

\subsection{Fixed point Floer cohomology}
\label{subsection:fixed point Floer cohomology}
Let us begin the construction of fixed point Floer cohomology. Let $\phi$ be a symplectomorphism on $W$ given as above. Our assumption on $\phi$ implies that there exists a compactly supported function $F_{\phi} : \widehat{W} \to \R$ such that
\begin{equation}\label{eq:F}
	\phi^*\lambda- \lambda = dF_{\phi}.
\end{equation}
This property holds since $(\mathrm{Id} \times \phi_R^a)^*( r \theta) = r \theta$ on $[1,\infty) \times \partial W \subset \widehat{W}$.

We consider the mapping torus of $\phi$:
$$T_{\phi_{}} = \R \times \widehat{W}/ (t, \phi_{} (x)) \sim (t+1,x).$$
Then sections of the projection $ \pi : T_{\phi_{}} \to \R/\Z=S^1, [(t,x)] \mapsto t$ are $\phi$-twisted loops, i.e., maps $\gamma : \R \to \widehat{W}$ satisfying
$$ \phi_{}(\gamma(t+1)) = \gamma(t).$$
Let us denote the space of twisted loops by
$$ \Omega_{\phi_{}} = \big\{ \gamma : \R \to \widehat{W} | \phi_{}(\gamma(t+1)) = \gamma(t) \big\}.$$

	Now, we consider time-dependent Hamiltonians $H : \R \times \widehat{W} \to \R$ satisfying the following:
	\begin{enumerate}
	\item $H$ is $\phi$-periodic in the sense that
	\begin{equation}
		\label{eq:phi-periodic}
		H(t+1,x) = H(t,\phi_{}(x)), \forall (t,x) \in \R \times \widehat{W}.
	\end{equation}

	\item $H$ satisfies
		\begin{equation}
		\label{eq:admissible hamiltonian}
			H(t, r,y) = h(r) \text{ in some open neighborhood of } R_i \text{'s}.
		\end{equation}
	for some increasing sequence $(R_i)_{i \in \mathbb{N}} \subset \R_{\geq R}$ and some smooth function $h : [R,\infty) \to \R$ such that
		\begin{itemize}
		\item $\lim_{i \to \infty} R_i =\infty$,
		\item $h''(r) \geq 0$ for all $r \geq R$.
		\end{itemize}	
	\end{enumerate}
	We call such a Hamiltonian $H$ $\phi$-admissible and  denote by $\mathcal{H}_{\phi} = \mathcal{H}_{\phi}(\widehat{W})$ the set of all $\phi$-admissible Hamiltonians.
	
	Let $H \in \mathcal{H}_{\phi}(\widehat{W})$ be given. The conditions \eqref{eq:F} and \eqref{eq:phi-periodic} allow us to consider the action functional $\mathcal{A}_{\phi, H} : \Omega_{\phi_{~}} \to \R$ defined by
	\begin{equation}\label{eq:action}
		\mathcal{A}_{\phi,H} (\gamma) = - \int_0^1 \gamma^* \lambda + \int_{0}^1 H(t,\gamma(t))dt  - F_{\phi_{~}}( \gamma(1)).
	\end{equation}

Critical points of $\mathcal{A}_{\phi,H}$ are twisted Hamiltonian orbits, that is, elements of the set
\begin{equation*}\label{eq:twistedorbits}
	\mathcal{P}({\phi,H}) = \big\{ \gamma\in \Omega_{\phi_{~}} | \dot{\gamma}(t) = X_{H_t} ( \gamma(t)) \big\},
\end{equation*}
where $X_{H_t}$ is the Hamiltonian vector field associated with $H_t (\cdot)= H(t,\cdot)$ on $\widehat{W}$. 
A twisted Hamiltonian orbit $\gamma \in \mathcal{P}(\phi,H)$ is called {\em non-degenerate} if
\begin{equation*}\label{eq:nondegenerate}
	\det ( d\phi \circ d\phi_{H}^1 -\text{Id}) \neq 0
\end{equation*}
at $T_{\gamma(1)} W$, where $\phi_H^t$ is the time-$t$ flow of the time-dependent Hamiltonian vector field $\{X_{H_t}\}_{t\in \R}$. We call $H$ {\em non-degenerate} if all the twisted Hamiltonian orbits $\gamma \in \mathcal{P}(\phi,H)$ are non-degenerate.

Next, let $\mathcal{J}_{\phi} = \mathcal{J}_{\phi}(\widehat{W})$ denote the set of all 1-parameter families $J = \{J_t\}_{t\in \R}$ of $\omega$-compatible almost complex structures $J_t$ of contact type on $\widehat{W}$ that are $\phi$-periodic in the sense that
\begin{equation}
	\label{eq:acs}
	J_{t+1} = \phi_{}^* J_t, \forall t\in \R.
\end{equation}
By a {\em Floer datum for $\phi$}, we mean a pair $(H,J)$ of a Hamiltonian $H \in \mathcal{H}_{\phi}(\widehat{W})$ and an almost complex structure $J \in \mathcal{J}_{\phi}(\widehat{W})$.

To endow our Floer chain complex with a $\mathbb{Z}$-grading, we first grade elements of $\mathcal{P}({\phi,H})$ for  $H \in \mathcal{H}_{\phi}(\widehat{W})$. For that purpose, the vertical tangent bundle $T^{\mathrm{ver}} T_{\phi}$ on $T_{\phi}$ given by
$$ T^\mathrm{ver} T_{\phi} : = \ker d\pi = \R \times T\widehat{W} /(t, d\phi(v)) \sim (t+1, v)$$
is required to satisfy $c_1(T^\mathrm{ver} T_{\phi}) =0$. This assumption implies that there is a non-vanishing complex volume form $\Omega$ on $T^\mathrm{ver} T_{\phi}$ (with respect to $J$), or equivalently a 1-parameter family of non-vanishing complex volume form $\{\Omega_{t}\}_{t \in \R}$ on $W$ (with respect to $\{J_t\}_{t\in \R}$) such that
\begin{equation}\label{eq:fiberwisecomplexvolumeform}
	\phi^*\Omega_t = \Omega_{t+1},\forall t\in \R.
\end{equation}
Since $\{\Omega_t\}_{t\in \R}$ is determined by $\{\Omega_t\}_{t\in [0,1]}$ due to  \eqref{eq:fiberwisecomplexvolumeform} and so is $\Omega$, we will just write $\Omega = \{ \Omega_t\}_{t \in [0,1]}$ to denote such a non-vanishing complex volume form on $T^\mathrm{ver} T_{\phi}$.

For a given non-degenerate twisted Hamiltonian orbit $\gamma \in \mathcal{P}({\phi,H})$, we assign its cohomological degree as follows. First we choose a symplectic trivialization $\Psi : [0,1]\times \R^{2n} \to \gamma^* T \widehat{W}$ such that
\begin{equation}\label{eq:compatibility1}
	d\phi \circ \Psi_1 = \Psi_0,
\end{equation}
\begin{equation}\label{eq:compatibility2}
	\Psi_t^* \Omega_t  = \wedge_{j=1}^n (dx_j +\sqrt{-1} dy_j),
\end{equation}
where $\Psi_t$ is the restriction $\Psi |_{\{t\} \times \R^{2n}}$ and $(x_1,y_1,\dots,x_n,y_n)$ is the coordinate of $\R^{2n}$.
This gives rise to a path of symplectic matrices:
\begin{equation}\label{eq:pathofsymplecticmatrices}
	\Lambda(t) = \Psi_{t}^{-1} \circ d\phi_{H}^{t} \circ \Psi_0, t\in [0,1].
\end{equation}
Then we define the cohomological degree of $\gamma$ by
$$\deg \gamma = n - \mu_{CZ}(\Lambda),$$
where the Conley-Zenhder index $\mu_{CZ}$ is defined as in \cite{Gut14}.

Let $(H,J)$ be a Floer datum for some non-degenerate Hamiltonian $H\in \mathcal{H}_{\phi}(\widehat{W})$ and let $\Omega$ be a non-vanishing complex volume form $T^\mathrm{ver}T_{\phi}$. We define the Floer complex $CF^*(\phi;H,J)$ by the graded vector space generated by $\mathcal{P}({\phi,H})$ over a field $\mathbb{K}$, where the degree of $\gamma \in \mathcal{P}({\phi,H})$ is given by $\deg \gamma$, i.e.,
$$ CF^k (\phi;H,J) = \bigoplus_{\deg \gamma = k} \mathbb{K} \cdot \gamma .$$

The differential of the Floer complex  $CF^*(\phi;H,J)$ is defined by counting Floer strips between two twisted Hamiltonian orbits of index 1 with proper signs. Here, a Floer strip means a solution $u : \mathbb{R} \times \R \to \widehat{W}$ to the following Floer equation
\begin{equation}\label{eq:floerequation}
	\begin{split}
		&\partial_s u + J_t ( \partial_t u - X_{H} (u)) =0,\\
		&\phi_{}(u(s, t+1)) = u(s,t), \forall s \in \R,t \in \R.
	\end{split}
\end{equation}

To be more precise, for given $\gamma_{-},\gamma_+ \in \mathcal{P}({\phi,H})$, we consider
\begin{equation}\label{eq:floerstrip}
	\mathcal{M} (\gamma_{-},\gamma_+;H, J) = \big\{ u : \R \times \R \to \widehat{W} : \eqref{eq:floerequation}, \lim_{s \to \pm \infty} u(s,t) = \gamma_{\pm} (t)\big\}.
\end{equation}
Considering that the radial coordinate $r$ is preserved by $\phi$ on the cylindrical end $[1,\infty)\times \partial W \subset \widehat{W}$, it is not difficult to see that a modification of the arguments in \cite{Se06} or in \cite[Section 4]{Ulj} shows that there is a sufficiently large $R'\geq 1$ such that every element $u$ of $\mathcal{M} (\gamma_{-},\gamma_+;H, J)$ do not enter $[R',\infty) \times \partial W \subset \widehat{W}$.

An element $u \in \mathcal{M}(\gamma_{-},\gamma_+;H, J)$ is called {\em regular} if the associated Fredholm operator is surjective so that the component $\mathcal{M}^u (\gamma_{-},\gamma_+;H, J)$ of $\mathcal{M} (\gamma_{-},\gamma_+;H, J)$ containing $u$ is transversely cut out and hence has dimension $\deg \gamma_- - \deg \gamma_+ $. We say that a Floer datum $(H,J)$ is regular if $H$ is non-degenerate and every element of $\mathcal{M}(\gamma_{-},\gamma_+;H, J)$ is regular. A standard argument shows that, for a given Hamiltonian $H \in \mathcal{H}_{\phi}(\widehat{W})$, there exists a residual subset $\mathcal{J}^{reg}_{\phi}(\widehat{W})$ of $\mathcal{J}_{\phi}(\widehat{W})$ such that the Floer datum $(H,J)$ is regular for all $J \in \mathcal{J}^{reg}_\phi(\widehat{W})$.

For a regular Floer datum $(H,J)$, the differential $\delta: CF^* (\phi_{};H,J) \to CF^{*}(\phi_{};H,J)$ is defined by
\begin{equation}\label{eq:differential}
	\delta(\gamma_+) = \sum_{\deg \gamma_- = \deg \gamma_+ +1} \# [\mathcal{M}( \gamma_-,\gamma_+; H,J)/\mathbb{R}] \cdot \gamma_{-},
\end{equation}
where the quotient $/\mathbb{R}$ is taken with respect to the  translation along the $s$-direction. 
The signed count $\# [\mathcal{M}( \gamma_-,\gamma_+; H,J)/\R]$ is defined as in \cite[Section 4]{Mc19}. Then the differential operator $\delta$ squares to zero as explained in \cite[Theorem 2.18]{Ulj}, \cite[Section 4]{Mc19}. We define the Floer cohomology for the Floer datum $(H,J)$ by the cohomology of the Floer complex $(CF^*(\phi_{};H,J),\delta)$:
$$HF^*(\phi_{};H,J) := H^*(CF^*(\phi_{};H,J),\delta).$$

The fixed point Floer cohomology is invariant under compact Hamiltonian isotopy in the following sense.
\begin{prop}
	\label{prop:invariance1}
	Let $\phi^1$ and $\phi^2$ be exact symplectomorphisms on a Liouville domain $(W,d\lambda)$ that are equal to a Reeb flow $\phi_R^a$ for some $a\in \R$ near $\partial W$. Furthermore, assume that $\phi^2 = \phi^1 \circ \phi_H^1$ for some Hamiltonian $H : \widehat{W} \to \R$ with a compact support $\mathrm{Supp}\ H$.
	
	Let $H^1 \in \mathcal{H}_{\phi_1}$ and $H^2 \in \mathcal{H}_{\phi_2}$ be given in such a way that
	$H^1 =H^2$ outside $\R \times K$ for some compact subset $V$ of $\widehat{W}$ containing $\mathrm{Supp}\ H$. 
	Then, for any regular Floer datum $(H^1,J^1)$ for $\phi^1$ and regular Floer datum $(H^2,J^2)$ for $\phi^2$, there is a canonical isomorphism
	\[HF^*(\phi^1 ;H^1,J^1) \cong HF^*(\phi^2 ;H^2,J^2).\]
\end{prop}

This can be shown by the standard argument using a continuation map. Indeed, there exists a continuation map $c_{1,2} : CF^*(\phi^1;H^1,J^1) \to CF^*(\phi^2;H^2,J^2)$ defined by using a 1-parameter family of triples $(\phi^s, H^s,J^s)$, $s\in \R$ of an exact symplectomorphism $\phi^s$ on $(W,d\lambda)$, $H^s \in \mathcal{H}_{\phi^s}$ and $J^s \in \mathcal{J}_{\phi^s}$ such that
\begin{itemize}
	\item $(\phi^s, H^s,J^s) = \begin{cases}(\phi^2, H^2,J^2) & (s<<0), \\(\phi^1, H^1,J^1)  &(s>>0), \end{cases}$
	\item $\phi^s$ is equal to the Reeb flow $\phi_R^a$ for all $s\in \R$, and
	\item $H^s$ is equal to $H^1$ outside $\R \times V$ for the compact subset $V$ above for all $s\in \R$. 
\end{itemize}
It can be seen that such a continuation map is a chain map and induces an isomorphism on the homology as it admits an inverse on the homology defined similarly.

\begin{remark}\label{remark:abstractcontactopenbook}
	An abstract contact open book is a triple $(W,\lambda, \phi)$ consisting of a Liouville domain $(W, \lambda)$ (denoted by $(W,d\lambda)$ in our paper) and an exact symplectomorphism $\phi$ on $W$ that is the identity map near $\partial W$. Following Seidel \cite{SeDehn}, McLean introduced the fixed point Floer cohomology $HF^*(\phi, \pm)$
	for a graded abstract contact open book $(W, \lambda,\phi)$,
	where the sign $\pm$ denotes the direction of the Reeb flow which is used to perturb $\phi$ near $\partial W$. 
	
	Here, a grading on 
	$(W,\lambda,\phi)$ consists of a grading on $W$ and a grading on $\phi$ compatible with the grading on $W$, see \cite{S08}, \cite{Mc19}, \cite[Section 5.4.3]{BP}. The former is determined by a non-vanishing complex volume form $\Omega_0$ on $W$ (up to homotopy) and the latter determined by a homotopy between $\Omega_0$ and $\phi^*\Omega_0$ (up to homotopy), i.e., a $1$-parameter family $\{\Omega_t\}_{t\in [0,1]}$ of complex volume form on $W$ such that
	\[ \Omega_1 = \phi^* \Omega_0.\]
	Therefore, this grading data is equivalent to our grading data, a non-vanishing complex volume form on $T^\mathrm{ver} T_{\phi}$
	as shown in \cite[Corollary 5.10]{BP}.
	
	Two abstract contact open books $(W_0,\lambda_0,\phi_0)$ and $(W_1, \lambda_1, \phi_1)$ are isotopic if $W_0$ and $W_1$ are diffeomorphic and there is a smooth $1$-parameter family of abstract contact open book $\{(W_t,\lambda_t,\phi_t)\}_{t \in [0,1]}$ connecting those two. An isotopy $\{(W_t,\lambda_t,\phi_t)\}_{t \in [0,1]}$ between two graded abstract contact open book is said to be graded if the $1$-parameter family of graded symplectomorphisms $\phi_t$ vary smoothly on $t \in [0,1]$. In \cite{Mc19}, it was proved that the fixed point Floer cohomology $HF^*(\phi, \pm)$ is invariant under a graded isotopy of graded abstract contact open book. For a graded abstract open book, 
	$HF^*(\phi,\pm)$ of \cite{Mc19} can be understood as $HF^*(\phi;H_{\pm \epsilon},J_{\pm \epsilon})$, where  $\epsilon>0$ is chosen to be smaller than the minimal period of positive/negative Reeb orbits and $(H_{\pm \epsilon}, J_{\pm \epsilon})$ is a Floer datum for $\phi$ such that $H_{\pm \epsilon}$ is a linear Hamiltonian with slope $\pm \epsilon$ as will be explained the following subsection. (see \cite[Section 6]{BP}).
\end{remark}

In the following subsections, we will consider the fixed point Floer cohomology for linear and quadratic 
$\phi$-admissible Hamiltonians, separately.
\subsubsection{Fixed point Floer cohomologies for linear Hamiltonians}
Let us first consider the fixed point Floer cohomology for linear Hamiltonians.

	Let $\phi$ be an exact symplectomorphism on a Liouville domain $(W,d\lambda)$ that equals the Reeb flow $\phi_R^a$ for some $a\in \R$ near $\partial W$. For any real number $b$ not equal to periods of any Reeb orbits, let $H_b \in \mathcal{H}_{\phi}(\widehat{W})$ be a Hamiltonian satisfying 
	$$H_b(t,r,x) = br + c , \forall (t,r,x)\in \R \times [R,\infty] \times \partial W$$
	for some $R>0$ and for some $c \in \R$.  We call such a Hamiltonian $H_b$ linear at infinity with slope $b$ or just linear with slope $b$. Note that a linear Hamiltonian $H_b \in \mathcal{H}_{\phi} (\widehat{W})$ can be chosen to be non-degenerate since $H_b$ is allowed to be time-dependent and generic enough on $\R \times \widehat{W} \setminus [R,\infty] \times \partial W$, where all the twisted Hamiltonian orbits of $H_b$ occur.
	
	For any real numbers $a_1 \leq a_2$ that are not periods of any Reeb orbits, we choose linear Hamiltonians $H_{a_1}, H_{a_2} \in \mathcal{H}_{\phi}(\widehat{W})$ with slopes $a_1$, $a_2$ such that $H_{a_1} \leq H_{a_2}$ outside a compact subset of $\widehat{W}$ and almost complex structures $J_{a_1},J_{a_2} \in \mathcal{J}_{\phi}(\widehat{W})$ for which $(H_{a_j}, J_{a_j})$ is regular for $j=1,2$. Then, as shown in \cite{Ulj}, there is a continuation map $c_{a_1,a_2} : CF^*(\phi; H_{a_1},J_{a_1}) \to CF^*(\phi; H_{a_2},J_{a_2})$ defined using a 1-parameter family of admissible Hamiltonians that are increasing from $H_{a_1}$ to $H_{a_2}$ outside a compact subset of $\widehat{W}$. We denote its induced map on homology by $c_{a_1,a_2}:HF^*(\phi; H_{a_1},J_{a_1}) \to HF^*(\phi; H_{a_2},J_{a_2})$ by abuse of notation. It was further shown that $c_{a_1,a_2}$ is independent of the choice of an increasing homotopy from $H_{a_1}$ to $H_{a_2}$ used to define it. Using these continuation maps, we define $HF^*(\phi,\infty)$ as follows.

\begin{defn}\label{defn:floercohomologyforaslope}

Let $(H_{a_i})_{i\in \mathbb{Z}_{\geq 1}}$ be a sequence of linear Hamiltonians in $\mathcal{H}_{\phi}(\widehat{W})$
for some increasing sequence $(a_i)_{i \in \mathbb{Z}_{\geq 1}}$ of real numbers $a_i \in [1 ,\infty)$ such that 
\begin{itemize}
	\item  $a_i$ is not a period of any Reeb orbit for each $i \in \mathbb{Z}_{\geq 1}$ and 
	\item $\lim_{i \to \infty} a_i = \infty$, 
\end{itemize}
and $(J^i)_{i \in \mathbb{Z}_{\geq 1}}$ is a sequence of almost complex structures in $\mathcal{J}_{\phi}(\widehat{W})$ such that the Floer datum $(H_{a_i}, J^i)$ is regular for all $i \in \mathbb{Z}_{\geq 1}$.
	Then $HF^*(\phi;\infty)$ is defined by
		\[HF^*(\phi, \infty)  =\lim_{i \to \infty} HF^*(\phi;H_{a_i},J_{a_i}).\]
		where the direct limit is taken with respect to the continuation maps as $i \to \infty$.
\end{defn}

For given two increasing sequence of Hamiltonians $(H_{a_i})_{i \in \mathbb{Z}_{\geq 1}}$ and $(H'_{a'_i})_{i \in \mathbb{Z}_{\geq 1}}$, after finding their subsequences, we may assume that
$$a_i \leq a'_i \leq a_{i+1},\forall i \in \mathbb{Z}_{\geq 1}.$$
It follows that the continuation map $c_{a_i,a_{i+1}} : HF^*(\phi;H_{a_i}, J_{a_i}) \to HF^*(\phi;H_{a_{i+1}},J_{a_{i+1}})$ factors through
$$c_{a_i,a'_i} :HF^*(\phi;H_{a_i}, J_{a_i}) \to HF^*(\phi;H'_{a'_i}, J'_{a'_i})  \text{ and } c_{a'_{i},a_{i+1}} : HF^*(\phi;H'_{a'_i}, J'_{a'_i}) \to HF^*(\phi;H_{a_{i+1}},J_{a_{i+1}} ) $$
and a similar statement holds for $c_{a'_i,a'_{i+1}}$. 
This implies that there exist canonically well-defined maps $\lim_{i\to \infty} HF^*(\phi;H_{a_i},J_{a_i}) \to \lim_{i \to \infty}  HF^*(\phi;H'_{a'_i},J'_{a'_i})$ and $\lim_{i\to \infty} HF^*(\phi;H'_{a'_i},J'_{a'_i}) \to \lim_{i \to \infty}  HF^*(\phi;H_{a_i},J_{a_i})$ inverse to each other. This shows that $HF^*(\phi, \infty)$ is well-defined regardless of the choice of Floer data $\{(H_{a_i}, J_{a_i})\}_{i\in \mathbb{Z}_{\geq 1}}$.

\subsubsection{Perturbation with Reeb flows}
\label{subsubsection:perturbation}
For a real number $c$, we say that $\widetilde{\phi}$ is a {\em good perturbation} of $\phi$ with a time-$c$ Reeb flow if
\[ \widetilde{\phi} = \phi \circ \phi^1_{H_{c}} \] 
for the time-$1$ Hamiltonian flow $\phi_{H_{c}}^1$ of a linear Hamiltonian $H_{c} \in \mathcal{H}_{\phi}(\widehat{W})$ satisfying
$$ H_c(t,r,y) = h_c(r),\forall (t,r,y) \in \R \times [1-\eta,\infty) \times \partial W]$$ for some smooth function $h_{c} : [1-\eta,\infty) \to \R$ for some $0<\eta < 1$ such that $h_{c} (r) = cr$ for $r\geq 1$. As a result, $\widetilde{\phi}$ is an exact symplectomorphism that is equal to the Reeb flow $\phi_{R}^{a+c}$ near $\partial W$. 

\begin{prop}  
	\label{prop:perturbationwithreebflow}
	Let $\phi$ be an exact symplectomorphism on a Liouville domain $(W,d\lambda)$ that is equal to a Reeb flow near $\partial W$.
	Let $\widetilde{\phi}$ be a good perturbation of $\phi$ with a time-$c$ Reeb flow for a real number $c$.
	\begin{eqnarray*}
		HF^*(\phi, \infty)  &\cong& HF^*(\widetilde{\phi}, \infty)
	\end{eqnarray*}
\end{prop}
\begin{proof}
	
	Let $b\in \R$ and $H_b \in \mathcal{H}_\phi$ be a linear Hamiltonian with slope $b$. Let us define $\widetilde{H}_{b-c}: \R \times \widehat{W} \to \R$ by
	\[ \widetilde{H}_{b-c}(t,x) = H_b(t, \phi_{H_c}^t(x)) - H_c(t, \phi_{H_c}^t(x)).\]
	Then, the assumption that $H_b, H_c \in \mathcal{H}_{\phi}(\widehat{W})$ implies that $\widetilde{H}_{b-c} \in \mathcal{H}_{\widetilde{\phi}}(\widehat{W})$. Furthermore, $\widetilde{H}_{b-c}$ is linear with slope $b-c$ at infinity.

	A simple computation shows that the time $t$-flow of the Hamiltonian vector field associated with $\widetilde{H}_{b-c}$ is given by
	\[\phi_{\widetilde{H}_{b-c}}^t(x) = (\phi_{H_c}^t)^{-1} \circ \phi_{H_b}^t(x).\]
	As a consequence, for a given $\gamma \in  \mathcal{P}({\phi,H_b})$, the path $\widetilde{\gamma}: [0,1] \to \widehat{W}$ defined by
	\[ \widetilde{\gamma}(t) = (\phi_{H_c}^t)^{-1}(\gamma(t)) \]
	is an element of $\mathcal{P}({\widetilde{\phi},\widetilde{H}_{b-c}})$. This correspondence gives the bijection
	\begin{equation}\label{eq:bijection}
		\mathcal{P}({\phi,H_b}) \leftrightarrow{} \mathcal{P}({\widetilde{\phi},\widetilde{H}_{b-c}}), \gamma \mapsto \widetilde{\gamma}.
	\end{equation}
	
	Assume that we have used the complex volume form $\{\Omega_t\}_{t\in [0,1]}$ on the vertical tangent bundle $T^\mathrm{ver} T_{\phi}$ to grade elements of $\mathcal{P}({\phi, H_b})$. If we grade elements of $\mathcal{P}({\widetilde{\phi},\widetilde{H}_{b-c}})$ with the complex volume form $\{\widetilde{\Omega}_t\}_{t \in [0,1]}$ on the vertical tangent bundle $T^\mathrm{ver} T_{\widetilde{\phi}}$ given by
	\[ \widetilde{\Omega}_t = (\phi_{H_c}^t)^* \Omega_t, t\in [0,1],\]
	then the bijection \eqref{eq:bijection} respects the gradings on $\mathcal{P}({\phi,H_b})$ and $\mathcal{P}({\widetilde{\phi},\widetilde{H}_{b-c}})$.

	Furthermore, for a given $J =\{J_t\} \in \mathcal{J}_{\phi}(\widehat{W})$, we define $\widetilde{J} =\{\widetilde{J}_t\} \in \mathcal{J}_{\widetilde{\phi}}(W)$ by
	\[ \widetilde{J}_t = (d\phi_{H_c}^t)^{-1} \circ J_t \circ d\phi_{H_c}^t.\]
	Then the bijection \eqref{eq:bijection} further extends to an identification between $\mathcal{M} (\gamma_{-},\gamma_+;H_b, J) \to  \mathcal{M} (\widetilde{\gamma}_{-},\widetilde{\gamma}_+;\widetilde{H}_{b-c}, \widetilde{J})$ by mapping $u \in \mathcal{M} (\gamma_{-},\gamma_+;H_b, J)$ to $\widetilde{u} \in  \mathcal{M} (\widetilde{\gamma}_{-},\widetilde{\gamma}_+;\widetilde{H}_{b-c}, \widetilde{J})$ defined by
	\begin{equation}\label{eq:bijectionbetweenmoduli}
		\widetilde{u}(s,t) = (\phi_{H_c}^t)^{-1} \circ u(s,t).
	\end{equation}
	Furthermore, if a Floer datum $(H_b,J)$ for $\phi$ is regular, then $(\widetilde{H}_{b-c}, \widetilde{J})$ is a regular Floer datum for $\widetilde{\phi}$. Summarizing what we have observed, the bijection \eqref{eq:bijection} induces a chain isomorphism
	\begin{equation}\label{eq:chainisomorphism}
		CF^*(\phi;H_b,J) \to CF^*(\widetilde{\phi};\widetilde{H}_{b-c},\widetilde{J}).
	\end{equation} 
	The correspondence \eqref{eq:bijection} further extends to a bijection between moduli spaces used to define continuations maps. Therefore, since $b-c$ goes to infinity as $b$ goes to infinity for a fixed real number $c$, the desired isomorphism follows.
\end{proof}

\subsubsection{Symplectic cohomology of $\phi$}
\label{subsubsection:symplectic cohomology}
Let us consider the fixed point Floer cohomology for quadratic $\phi$-periodic Hamiltonians.

	Let $H^0 \in \mathcal{H}_{\phi}(\widehat{W})$ be a Hamiltonian satisfying 
	$$ H^0(t,r,x) = r^2, \forall (t,r,x) \in \R \times [R,\infty] \times \partial W$$
	for some $R \geq 1$.  
	
	Since $H^0$ is time-independent on $[R,\infty] \times \partial W$ and $\phi$ is assumed to be the identity map on $\widehat{W}\setminus W$, we further need to perturb the Hamiltonian $H^0$ with a time-dependent function to get rid of the $S^1$-symmetry of Hamiltonian orbits. Indeed, as done in \cite{A10}, we choose a smooth function $F : \R \times \widehat{W} \to \R$ such that 
	\begin{itemize}
		\item $\mathrm{Supp}\ F \subset [R,\infty] \times \partial W$,
		\item $F(t+1,x) = F(t,x)$ for all $(t,x) \in \R\times \widehat{W}$,
		\item $F$ and $\lambda(X_F)$ are uniformly bounded in absolute value, and 
		\item $F$ vanishes outside small neighborhoods of periodic Hamiltonian orbits of $X_{H^0}$.
	\end{itemize}
	Then we define $H \in \mathcal{H}_{\phi}(\widehat{W})$ by
	\begin{equation}
	\label{eq:quadratic}
	H(t,x) = H^0(t,x) + F(t,x), \forall (t,x) \in \mathbb{R} \times \widehat{W},
	\end{equation}
	which makes sense since both $H^0$ and $F$ are $\phi$-periodic. We call such a Hamiltonian $H$ quadratic at infinity or just quadratic. As proven in \cite[Lemma 5.1]{A10}, for a generic choice of such functions $F$ and $H^0$, every twisted Hamiltonian orbit $\gamma \in \mathcal{P}(\phi,H)$ is non-degenerate. 
	
	For a quadratic Hamiltonian $H \in \mathcal{H}_{\phi}(\widehat{W})$ and $J \in \mathcal{J}_{\phi}(\widehat{W})$ that form a regular Floer datum $(H,J)$, we denote the Floer complex $(CF^*(\phi;H,J), \delta)$ by $SC^*(\phi)$ and denote the cohomology of $SC^*(\phi)$ by $SH^*(\phi)$, which we will call the symplectic cohomology of $\phi$. The name ``symplectic" can be justified since $SH^*(\phi)$ is canonically isomorphic to the symplectic cohomology $SH^*(W)$ when $\phi = \mathrm{Id}$.

A slight modification of the argument given in \cite[Section (3e)]{Se06} adapted to the fixed point Floer cohomology proves the following proposition.
\begin{prop}\label{prop:equivalence} Let $\phi$ be an exact symplectomorphism on a Liouville domain $(W,d\lambda)$ that is equal to a Reeb flow near $\partial W$. Then there is a canonical isomorphism:
	\[HF^*(\phi, \infty) \cong SH^*(\phi).\]
\end{prop}

Combining Proposition \ref{prop:perturbationwithreebflow} and Proposition \ref{prop:equivalence}, we have the following corollary.
\begin{cor}\label{cor:invarianceofsymplecticcohomology}
	Let $\widetilde{\phi}$ be a good perturbation of $\phi$ with a time-$c$ Reeb flow for some real number $c$. Then there is an isomorphism
	\[ SH^*(\phi) \cong SH^*(\widetilde{\phi}).\]
\end{cor}

\begin{remark}
	\label{remark:continuation}
	As mentioned in Remark \ref{remark:abstractcontactopenbook}, $HF^*(\phi,\pm)$ can be understood as the fixed point Floer cohomology $HF^*(\phi; H_{\pm \epsilon},J_{\pm})$ for some admissible Hamiltonian $H_{\pm\epsilon} \in \mathcal{H}_{\phi}(\widehat{W})$ that equals $\pm \epsilon r +c$ for a sufficiently small $\epsilon>0$ and a real number $c$ on the cylindrical end $[1,\infty) \times \partial W$ and some $J_{\pm} \in \mathcal{J}_{\phi}(\widehat{W})$.
	
	For later use in this paper, note that there is a continuation map
	\[ c : HF^*(\phi; H_{-\epsilon},J_-) \to HF^*(\phi; H_{\epsilon},J_+)\]
	defined using an increasing 1-parameter family of admissible Hamiltonians from $H_{-\epsilon}$ to $H_{\epsilon}$. Similarly, since there is also an increasing 1-parameter family of admissible Hamiltonians from $H_{\pm \epsilon}$ to a quadratic Hamiltonian $H$, we may consider continuation maps
	\[c_{\pm} : HF^*(\phi; H_{\pm \epsilon},J_\pm) \to SH^*(\phi).\]
	We will use the continuation map $c_+$ when defining the class $\Gamma$ in Section \ref{subsec:Gamma}
\end{remark}

\section{Monodromy class $\Gamma$ of an isolated singularity}
\label{sec:Gamma}
Let $f:(\C^n,0) \to (\C,0)$ be a germ of an isolated hypersurface singularity for some $n \geq 2$.
The monodromy $\rho$ of the Milnor fiber $M$ of $f$ can be represented by an exact symplectomorphism that is
the identity in a neighborhood of the boundary $\partial M$.  We will define a Floer cohomology class $\Gamma \in SH^*(\rho^{-1})$, called the monodromy class. We will work with a particular monodromy symplectomorphism $\rho$ constructed by Fernandez de Bobabilla-Pełka \cite{BP}, since it has no fixed points except possibly in a neighborhood of $\partial M$. This will be used to show that $\Gamma$ is a cohomology class.
 
 \subsection{Milnor fiber and Monodromy $\rho$ of an isolated singularity}
  Let $\mathbb{B}_\epsilon(0) \subset \mathbb{C}^{n}$ be the Milnor ball of $f$. For
a sufficiently small $0<\delta < \epsilon$, we have a Milnor fibration
\begin{equation}\label{eq:mfib}
f|_{f^{-1}(B_{\delta}^*(0))} : f^{-1} (B_{\delta}^*(0)) \cap \mathbb{B}_\epsilon(0)  \to B_{\delta}^*(0).
\end{equation}
Here $B_{\delta}(0) \subset \C$ is the closed disc of radius $\delta$ centered at the origin and  $B_{\delta}^*(0) := B_{\delta}(0) \setminus\{0\}$.
The pre-image of a point of $f$ is diffeomorphic to the Milnor fiber $M$, and the parallel transport along the circle $S_{\delta} : = \partial B_{\delta}(0)$ on the base
defines a monodromy of $M$. It is an important question to understand the properties of this monodromy.

A'Campo considered an embedded resolution of the singularity $f$ until the preimage of 0 becomes normal crossing divisors and used their local structure to define so-called the A'Campo space of $f$. This was used to compute the zeta function of the monodromy of $f$. In particular, the monodromy of $f$ can be chosen so that it does not have any fixed points except in a neighborhood of the boundary of the Milnor fiber. Thus one may ask if there exists a symplectic monodromy having such a nice topological behavior studied by A'Campo. 

Fernandez de Bobabilla-Pe{\l}ka \cite{BP} answered this positively by constructing a symplectic version of A'Campo space. Let us recall their settings in two steps. The first step is a symplectic refinement of the fibration \eqref{eq:mfib} and the second step is to do the same thing for A'Campo space.

For the first step, they constructed a Liouville fibration $f : N_{\delta} \to  B_{\delta}^*(0)$ as a modification of the Milnor fibration \eqref{eq:mfib} where $N_{\delta} \subset  f^{-1} (B_{\delta}^*(0))$ is a carefully
chosen submanifold with corners via symplectic parallel transports so that monodromy
around any loop in $B_{\delta}^*(0)$ is given by an exact symplectomorphism.

Namely, the fibration $f:N_{\delta} \to  B_{\delta}^*(0)$ is equipped with a $1$-form $\lambda_f \in \Omega^1(N_{\delta})$ that restricts to a Liouville $1$-form on each fiber.
Let ${M}_{\delta} = N_{\delta} \cap f^{-1} (\delta)$ be the fiber at the boundary point $\delta \in S_{\delta}$ and let ${\lambda}_{\delta} \in \Omega^1({M}_{\delta})$ be the restriction of $\lambda_f$ to the fiber ${M}_{\delta}$. Then, consider the restriction of this fibration to the boundary circle $S_{\delta} \subset B_{\delta}^*(0)$.
As argued in \cite[Section 5]{BP}, this admits a symplectic monodromy trivialization (\cite[Definition 5.5]{BP}) $\Phi :{M}_{\delta} \times [0,1] \to N_{\delta} \cap f^{-1}(S_{\delta})$ such that
\begin{equation}\label{rhodel}
f|_{f^{-1}(S_{\delta})} (\Phi(x,t)) = \delta \cdot e^{2 \pi i t} \in S_{\delta},\forall (x,t) \in M_{\delta} \times [0,1].
\end{equation}
In particular, ${\rho}_{\delta} = \Phi( - ,1): M_{\delta} \to M_{\delta}$ is the symplectic monodromy relative to $\Phi$, which is an exact symplectomorphism that is the identity map near $\partial M_{\delta}$.

For the second step, denote by  $\tau: B_{\delta, \log}(0) \to B_{\delta}(0)$ the real oriented blow up at the origin.
Namely,  the origin is blown-up to the circle $\partial B_{\delta,\log}(0) := \tau^{-1}(0)$, called the circle at radius 0. Note that $B_{\delta,\log}(0) \setminus \partial B_{\delta,\log}(0)$ is identified with $B_{\delta}^*(0)$. 
 Fernandez de Bobabilla-Pe{\l}ka \cite{BP} constructed a symplectic version of A'Campo space $A$, which has a Liouville fibration
 $$ f_{A}: A \to B_{\delta,\log}(0)$$ extending $f|_{f^{-1}(B_{\delta}^*(0))}$ to the circle $\partial B_{\delta,\log}(0)$ at radius zero.

Furthermore, in \cite[Section 5.5]{BP}, they constructed a subset $N$ of $A$ and a Liouville fibration $f_{A}|_{N} : N \to B_{\delta,\log}(0)$ together with a $1$-form $\lambda_N\in \Omega^1(N)$ in such a way that
\begin{enumerate}
\item \label{item:1} it coincides with the usual Milnor fibration $f: f^{-1} (B_{\delta}^*(0)) \to B_{\delta}^*(0)$ outside $f^{-1} (B_{\delta'}(0))$ for a certain $0 < \delta' <\delta$ and
\item the monodromy along the circle $\partial B_{\delta,\log}(0)$, called the symplectic monodromy at radius zero, has the same dynamics as that considered by A' Campo.
\end{enumerate}
See \cite[Section 5.5, Corollary 7.9]{BP} for more details. Using this construction, we define the following. 
\begin{defn}
	\label{defn:monodromy}
	Let $(M,d\lambda)$ be the Liouville domain obtained as a fiber of $f_{A}|_N$ at some point in $\partial B_{\delta,\log}(0)$ equipped with the Liouville $1$-form $\lambda := \lambda_N|_M$ and let $\rho :M \to M$ be the symplectic monodromy at zero.
\end{defn}
In particular, one important dynamical property of the symplectic monodromy $\rho$ used in this paper is that its fixed locus is given by $\partial M$. By gluing a collar neighborhood of $\partial M$ to $M$ and extending $\rho$ trivially over it, we assume that the fixed point set of $\rho$ is a neighborhood of $\partial M$.

\subsection{Floer complex for $\rho^{-1}$ and its Hamiltonian perturbation $\check{\rho}^{-1}$}
For the exact symplectic monodromy $\rho$ (at radius 0), we consider the symplectic cohomology $SH^*(\rho^{-1})$  of its inverse $\rho^{-1}$ defined in the previous subsection and will define a certain cohomology class $\Gamma \in SH^0(\rho^{-1})$. 

Let us first comment on the relation with $\rho$ and $\rho_\delta$ defined from \eqref{rhodel}.
The first item \eqref{item:1} for the fibration $f_{A}: N \to B_{\delta,\log}(0)$ shows that $\rho$ and $\rho_\delta$ are isotopic as exact symplectomorphisms \cite{BP}. It follows that the abstract contact open book $(M,\lambda,\rho)$ is isotopic to $(M_{\delta}, {\lambda}_{\delta}, {\rho}_{\delta})$. As a consequence the Floer cohomologies $HF^*(\rho^{\pm 1},\pm)$ and $HF^*({\rho}^{\pm 1}_{\delta},\pm)$ are canonically isomorphic as proven in \cite[Appendix B]{Mc19}. A slight modification of their method further shows that $SH^*(\rho^{\pm 1})$ and $SH^*({\rho}_{\delta}^{\pm 1})$ are canonically isomorphic.

The radius 0 monodromy $\rho$ has the following property as in the case of the topological monodromy studied by A'Campo: The fixed point set $\mathrm{Fix}\ \rho$ is given by a tubular neighborhood of $\partial M$.  More precisely, $\mathrm{Fix}\ \rho$ is a codimension zero submanifold $B_0$ of $M$, which is homeomorphic to the cylinder $[1-\tilde{\eta},1] \times f^{-1}(0) \cap \partial \mathbb{B}_{\epsilon}(0)$ of the link $f^{-1}(0) \cap \partial \mathbb{B}_{\epsilon}(0)$ for some small $0<\tilde{\eta}<1$.

We can also perturb the fixed points away using a Hamiltonian. Namely, it was further argued in \cite{Mc19}, \cite{BP} that there exist an open neighborhood $V$ of $B_0$ and a time-independent Hamiltonian $H_{0} : V \to (-\infty, 0]$ such that
\begin{itemize}
\item $B_0 = H_0^{-1} (0) = \mathrm{Fix}\ \rho$ and
\item $\rho|_{V} = {\phi_{H_0}^1}|_{V} :V \to V$.
\end{itemize}
Then $\rho$ can be further perturbed into $\check{\rho}$ with a Hamiltonian isotopy that equals a positive Reeb flow near $\partial M$ so that the fixed locus $B_0$ vanishes as follows.
\begin{prop}[\cite{Mc19}, Section 5]\label{prop:perturbation}
There exist a Hamiltonian $b : V \to \R$ and small $\nu>0$ satisfying
\begin{itemize}
\item $b = F \circ H_0$ for some smooth function $F : \R \to \R$ in $H_0^{-1} \left[ -\nu ,-\frac{\nu}{3}\right]$ such that $F \circ H_0 = H_0$ near $H_0^{-1} (-\nu)$,
\item $b$ is $C^2$-small in $H_0^{-1} \left(  -\frac{\nu}{3},0 \right)$,
\item $b = \epsilon r$ near $\partial M$, where $r$ is the radial coordinate of $M$ for $\epsilon>0$ smaller than minimal period of Reeb vector field, and
\item $b$ has no critical points.
\end{itemize}
Then $\check{\rho} : M \to M$ defined by
$$\check{\rho} =\begin{cases}\rho& {\mathrm{outside }\ V} \\ \phi_b^1 &{\mathrm{inside }\ V}\end{cases}$$
has no fixed points. (See Figure \ref{fig:rhocheck}.)
\end{prop}

\begin{figure}[h]
\begin{subfigure}[t]{0.45\textwidth}
\includegraphics[scale=0.9]{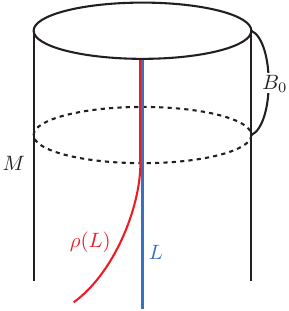}
\centering
\caption{$\rho = \mathrm{Id}$ on $B_0$}
\end{subfigure}
\begin{subfigure}[t]{0.45\textwidth}
\includegraphics[scale=0.9]{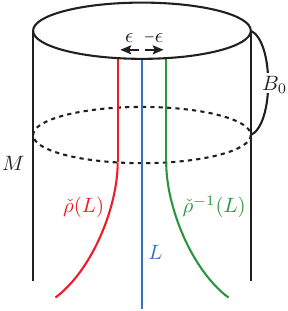}
\centering
\caption{$\check\rho$ and $\check\rho^{-1}$ has no fixed point}
\end{subfigure}
\centering
\caption{Illustration of  the image of $\rho$ and $\check\rho$ for the curve $L$ in the case $n=2$.}
\label{fig:rhocheck}
\end{figure}

Now we will argue that $SH^*(\check{\rho}^{-1})$ and $SH^*(\rho^{-1})$ are isomorphic. Indeed, the above proposition says more specifically that 
\begin{equation}\label{eq:rhoK}
 \check{\rho} = \phi_K^1 \circ \rho,
 \end{equation}
where $\phi_K^1$ is the time-$1$ flow of a time independent Hamiltonian $K$ on $M$ such that $K = \epsilon r$ on a collar neighborhood $[1-\eta,1] \times \partial M$ of $\partial M$ in $M$ for some $\eta \in (0,1)$ such that $[1-\eta,1] \times \partial M \subset \mathrm{Fix}\ \rho$. 

On the other hand, there exists a Hamiltonian $H_\epsilon \in \mathcal{H}_{\rho}$ satisfying 
$$ H_{\epsilon}(t,r,y) = h_{\epsilon}(r), \forall (t,r,y) \in \R \times [1-\eta, \infty) \times \partial M,$$
 where $h_{\epsilon} : [1-\eta,\infty) \to \R$ satisfies $h_{\epsilon}(r) = \epsilon r -\epsilon$ for $r\geq 1$. This means that $\widetilde{\rho}^{-1} := \rho^{-1} \circ \phi_{-H_\epsilon}^1$ is a good perturbation of $\rho^{-1}$ with a time-$(-\epsilon)$ Reeb flow. It follows from Corollary \ref{cor:invarianceofsymplecticcohomology} that there exists an isomorphism
\begin{equation}\label{eq:isomorphism1}
SH^*(\rho^{-1}) \cong SH^*(\widetilde{\rho}^{-1}).
\end{equation}

Furthermore, by construction of $\check{\rho}$ and $\widetilde{\rho}$, it can be deduced that there is a Hamiltonian isotopy between $\check{\rho}^{-1}$ and $\widetilde{\rho}^{-1}$ generated by some time-independent Hamiltonian whose support lies in the interior of $M$. Thus, Proposition \ref{prop:invariance1} implies that there exists a canonical isomorphism
\begin{equation}\label{eq:isomorphism2}
SH^*(\widetilde{\rho}^{-1}) \cong SH^*(\check{\rho}^{-1}).
\end{equation}
 
 Combining \eqref{eq:isomorphism1} and \eqref{eq:isomorphism2}, we obtain the desired isomorphism
\begin{equation}\label{eq:isomorphism3}
SH^*(\rho^{-1}) \cong SH^*(\check{\rho}^{-1}).
\end{equation}

\subsection{Monodromy class $\Gamma$} \label{subsec:Gamma}
Let us begin the construction of the cohomology class $\Gamma \in SH^*(\check{\rho}^{-1})$. 
We will first define it as a cohomology class in $HF^*(\rho^{-1},+)$ and consider its image under the continuation map $c_+ : HF^*(\rho^{-1},+) \to SH^*(\rho^{-1})$; see Remark \ref{remark:abstractcontactopenbook} for the definition of $HF^*(\rho^{-1},+)$ and see Remark \ref{remark:continuation} for the continuation map $c_+$.
Since we will work with closed-open maps from $SH^*({\rho}^{-1})$ to $HW^*(L,{\rho}(L))$ for some Lagrangian $L$ of $M$ later in this paper, it is more natural to regard $\Gamma$ as an element of $SH^*({\rho}^{-1})$.

Let $H_{2\epsilon} \in \mathcal{H}_{\check{\rho}^{-1}}$ be a linear Hamiltonian with slope $2\epsilon$ for $\epsilon>0$ chosen above such that
\begin{itemize}
\item it is constantly zero on $M$ and
\item it satisfies $H_{2\epsilon} (t,r,y) = h_{2\epsilon}(r), \forall (t,r,y) \in \R \times [1,\infty)\times \partial M$ for some function $h_{2\epsilon} : [1,\infty) \to \R$ that is linear with slope $2\epsilon$ at infinity and satisfies $h_{2\epsilon}''(r) \geq 0$.
\end{itemize}
Then, for an almost complex structure $J_{2\epsilon} \in \mathcal{J}_{\check{\rho}^{-1}}$ that forms a regular Floer datum $(H_{2\epsilon},J_{2\epsilon})$, $HF^*(\rho^{-1},+)$ is defined by $HF^*(\check{\rho}^{-1};H_{2\epsilon},J_{2\epsilon})$. This can be justified since there is an isomorphism $HF^*(\check{\rho}^{-1}, H_{2\epsilon},J_{2\epsilon}) \cong HF^*(\rho^{-1};\widetilde{H}_{\epsilon},\widetilde{J}_{\epsilon})$ for some linear Hamiltonian $\widetilde{H}_{\epsilon}$ of slope $\epsilon$ as seen in the proof of Proposition \ref{prop:perturbationwithreebflow}.

Here, it makes sense to assume that $H_{2\epsilon}$ is zero on $M$ and non-degenerate simultaneously since $\check{\rho}^{-1}$ has no fixed point in $M$ (Proposition \ref{prop:perturbation}). Furthermore, there is a unique $r_0 \in (1,\infty)$ such that $h'_{2\epsilon}(r_0) = \epsilon$ by our choice of $h_{2 \epsilon}$. These observations and Proposition \ref{prop:perturbation} make it easier to compute $HF^*(\rho^{-1}, +)$.

\begin{lemma}
The following hold.
\begin{enumerate}
\item There is no twisted Hamiltonian orbit $\gamma \in \mathcal{P}({\check{\rho}^{-1},H_{2\epsilon}})$ whose image interesects $M \cup [1,r_0) \times \partial M$.
\item $\Sigma:= \{r_0\} \times \partial M$ forms a Morse-Bott component of twisted Hamiltonian orbits in $\mathcal{P}({\check{\rho}^{-1},H_{2\epsilon}})$.
\end{enumerate}
\end{lemma}
\begin{proof}
 Note that the set $\mathcal{P}({\check{\rho}^{-1}, H_{2\epsilon}})$ is in one-to-one correspondence with the fixed point set $\mathrm{Fix}\left(\check{\rho}^{-1} \circ \phi_{H_{2\epsilon}}^1\right)$.
But, for $x \in M$, $\check{\rho}^{-1} \circ \phi_{H_{2\epsilon}}^{1}(x) = \check{\rho}^{-1} (x) \neq x$ since $H_{2\epsilon} = 0$ on $M$ and $\check{\rho}^{-1}$ is assumed to have no fixed point as in Proposition \ref{prop:perturbation}. Moreover, for $x =(r,y)\in [1,r_0) \times \partial W$,
$$\check{\rho}^{-1} \circ \phi_{H_{2\epsilon}}^{1}(x) = \check{\rho}^{-1}\left( r, \phi_R^{h_{2\epsilon}'(r)}(y)\right) = \left(r, \phi_R^{h_{2\epsilon}'(r)-\epsilon}(y) \right)\neq (r,y)=x$$
 due to our choice of $\epsilon$ and $r_0$.

For the second statement, for any $y \in \partial M$, $\gamma_y : [0,1] \to \widehat{M}$ defined by
\[\gamma_y(t) = \left(r_0, \phi_R^{h'(r_0 )t }(y)\right) = \left(r_0, \phi_R^{\epsilon t } (y)\right)\]
is an element of $\mathcal{P}({\check{\rho}^{-1},H_{2\epsilon}})$ since $\check{\rho}^{-1}$ is extended to $[1-\eta,\infty) \times \partial M$ by $\mathrm{Id}\times \phi_R^{-\epsilon}$. These form a Morse-Bott component of twisted Hamiltonian orbits, which we denote by $\Sigma =\{ \gamma_y | y \in \partial M\}.$ (See Figure \ref{fig:Sigma}.)
\end{proof}

\begin{figure}[h]
\includegraphics[scale=0.8]{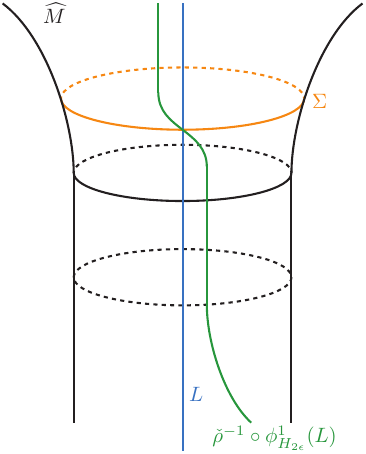}
\centering
\caption{Illustration of the Morse-Bott family $\Sigma$ for $n=2$ case}
\label{fig:Sigma}
\end{figure}

Our class $\Gamma$ will correspond to the fundamental chain of the Morse-Bott component $\Sigma$.  More precisely, we use the perturbation introduced in \cite{KvK16} to define it. Indeed, for a small real number $\delta>0$, Morse function $g_{\partial M}: \partial M \cong \{r_0\} \times \partial M \to \R$, and a smooth concave function $q : [1-\eta,\infty) \to [0,1]$ supported near $r_0$ with a unique maximum at $r_0$, we define a perturbation $H_{2\epsilon,\delta}$ of $H_{2\epsilon}$ with $g_{\partial M}$ as follows:
\begin{align*}
		\overline g(t,r,y)&:= \left( (g_{\partial M} \cdot q) \circ \phi_K^t(r,y) \right) ,\forall (t,r,y) \in \R \times [1-\eta,\infty) \times \partial M,\\
		H_{2\epsilon,\delta}&:= H_{2\epsilon} +\delta \overline g,
\end{align*}

where $g_{\partial M} \cdot q : \widehat{M} \to \R$ is a smooth function supported in $[1-\eta,\infty) \times \partial M$ such that
$$g_{\partial M} \cdot q(r,y) = g_{\partial M}(y) q(r), \forall (r,y) \in [1-\eta,\infty) \times \partial M.$$ 
Note that $H_{2\epsilon,\delta}$ is still a linear Hamiltonian in $\mathcal{H}_{\check{\rho}^{-1}}$ with slope $2\epsilon$ at infinity since $q$ is supported near $r_0$.
		
 Let $J_{2\epsilon,\delta} \in \mathcal{J}_{\check{\rho}^{-1}}$ be an almost complex structure such that $(H_{2\epsilon,\delta},J_{2\epsilon,\delta})$ forms a regular Floer datum. If $\delta>0$ is sufficiently small, then the Floer complex $CF^*({\rho}^{-1}; H_{2\epsilon,\delta},J_{2\epsilon,\delta})$ is generated by short twisted Hamiltonian orbits $\gamma_p \in \mathcal{P}_{\check{\rho}^{-1}, H_{2\epsilon,\delta}}$, $p \in \mathrm{Crit}\ g$. Here, by a short twisted Hamiltonian orbit, we mean a twisted Hamiltonian orbit $\gamma_p(t) = (\mathrm{Id} \times \phi_{R}^{\epsilon t})(p)$ of length $\epsilon$, which can be chosen to be an arbitrarily small positive number. Furthermore, the differential on $CF^*(\check{\rho}^{-1};H_{2\epsilon,\delta},J_{2\epsilon,\delta})$ is identified with the Morse differential on the Morse complex $CM^*(g_{\partial M})$ of $g_{\partial M}$, which follows from \cite{CFHW96}, \cite{KvK16}.

For each critical point $p \in \mathrm{Crit}\ g$, the degree of $\gamma_p$ is given by
\begin{align*}
	\deg \gamma_p &= - \mu_{CZ}(\Sigma) + \mathrm{ind}_{g_{\partial M}} p.
\end{align*}
Here the Conley-Zehnder index $\mu_{CZ}(\Sigma)$ for the linearized Hamiltonian flow is defined via the symplectic trivialization of a Hamiltonian orbit $\gamma \in \Sigma$ as in Section \ref{sec:symplectic cohomology} (see \cite{Gut14}).   In Proposition \ref{prop:index}, we will show that  $\mu_{CZ}(\Sigma)=0$. Now we define the class $\Gamma_{\epsilon}$ as follows.
\begin{defn}
 We define $\Gamma_{\epsilon} \in CF^*(\check{\rho}^{-1};H_{2\epsilon, \delta},J_{2\epsilon,\delta})$ by
\[ \Gamma_{\epsilon} = \sum_{\substack{p \in \mathrm{Crit}\ g_{\partial M}, \\ \mathrm{ind}_{g_{\partial M}} p = 0}} \gamma_{p}.\]
\end{defn}
According to the above formula, if we assume $\mu_{CZ}(\Sigma)=0$, then the degree of $\Gamma$ is given by
\[ \deg \Gamma_{\epsilon} = -\mu_{CZ} (\Sigma)  +\mathrm{ind}_{g_{\partial M}} p = 0.\]
Furthermore, $\Gamma_{\epsilon}$ is a cycle since $\sum_{p \in \mathrm{Crit}\ g_{\partial M},  \mathrm{ind}_{g_{\partial M}} p = 0} p$ is a cycle of the Morse complex of $g_{\partial M}$ representing the identity class in the cohomology of $\partial M$ or the fundamental class of the homology of $\partial M$ under the canonical isomorphisms:
\[ HM^0(g) \cong H^0(\partial M;\mathbb{K}) \cong H_{2n-1}(\partial M;\mathbb{K}).\]

Now we are ready to define the class $\Gamma$.
\begin{defn}[Monodromy class]
\label{defn:gamma} The monodromy class
$\Gamma \in SH^*(\check{\rho}^{-1}) \cong SH^*(\rho^{-1})$ is defined by 
\[ \Gamma = c_+(\Gamma_{\epsilon}),\]
where $c_+ : HF^*(\check{\rho}^{-1}; H_{2\epsilon},J_{2\epsilon}) \to SH^*(\check{\rho}^{-1})$ is the continuation map introduced in Remark \ref{remark:continuation}.
\end{defn}

To realize the chain level description of the class $\Gamma$, we consider a regular Floer datum $(H,J)$ consisting of an almost complex structure $J \in \mathcal{J}_{\check{\rho}^{-1}}$ and a quadratic Hamiltonian $H \in \mathcal{H}_{\check{\rho}^{-1}}$ such that
\begin{itemize}
\item it is constantly zero on $M$ and 
\item it satisfies \eqref{eq:quadratic} for some quadratic Hamiltonian $H_0$ and some perturbation function $F$ such that 
\begin{equation*}
H_0(t,r,y) = h(r), \forall (t,r,y) \in \R \times [1,\infty) \times\partial M
\end{equation*}
for some function $h : [1,\infty) \to \R$ quadratic at infinity.
\end{itemize}
Then $SH^*(\check{\rho}^{-1})$ can be defined by $HF^*(\check{\rho}^{-1};H,J)$.

Moreover, for convenience, we may choose $h$ in such a way that
\begin{itemize}
\item $h = h_{2\epsilon}$ on $[1, 2r_0]$ for $r_0$ introduced above and
\item $h \geq h_{2\epsilon}$ everywhere.
\end{itemize}

Then, as in the case of $H_{2\epsilon}$, there is a Morse-Bott family of twisted orbits of $\mathcal{P}({\check{\rho}^{-1}, H})$ diffeomorphic to $\partial M$ along $\{r_0\} \times \partial M \subset \widehat{M}$. Once again, we may perturb $H$ with the Morse function $g_{\partial M}$ on $\partial M$ and some sufficiently small $\delta>0$ exactly in the same way we did for $H_{2\epsilon}$. Let us call the resulting perturbed Hamiltonian $H_{\delta}$. Consequently, for each $p \in \mathrm{Crit}\ g_{\partial M}$, there is a corresponding twisted Hamiltonian orbit $\gamma_p \in \mathcal{P}_{\check{\rho}^{-1},H_{\delta}}$ and hence we may define $\Gamma \in CF^*(\check{\rho}^{-1}; H_{\delta},J_{\delta})$ by 
\[ \Gamma = \sum_{\substack{p \in \mathrm{Crit}\ g_{\partial M},\\ \mathrm{ind}_{g_{\partial M}} p =0}} \gamma_p .\] 

We may assume that
\[ H_{2\epsilon, \delta} = H_{\delta} \]
on $M \cup [1,2r_0]\times \partial M \subset \widehat{M}$ due to our choice of $h$. Then considering that the continuation map does not decrease the action value, it is not difficult to see that the continuation map $c_+ : CF^*(\check{\rho}^{-1}; H_{2\epsilon,\delta},J_{2\epsilon,\delta}) \to CF^*(\check{\rho}^{-1}; H_{\delta},J_{\delta})$ is identified with the inclusion $CF^*_{(-2\epsilon,\infty)}(\check{\rho}^{-1};H_{\delta},J_{\delta}) \to CF^*(\rho^{-1};H_{\delta},J_{\delta})$ of the subcomplex generated by twisted Hamiltonian chords of action greater than $-2\epsilon$ and there for it sends $\Gamma_{\epsilon}$ to $\Gamma$ on the chain level.

\subsection{Comparison with the spectral sequence of McLean\cite{Mc19}}\label{remark:spectralsequence}
 The spectral sequence associated with the action filtration is a useful tool for studying Floer cohomology.
McLean \cite{Mc19} proved that the minimal natural number $m$ such that $HF^*(\rho^m,+)$ does not vanish is equal to the multiplicity of the singularity (conjectured by Seidel \cite{SeidelLecture2014}) using the spectral sequence.
The monodromy class $\Gamma$ should be visible in the similar spectral sequence that converges to the Floer cohomology  $HF^*(\rho^{-1}, + )$, and we specify the corresponding element in this subsection.
For that purpose, we first show that a local fixed point cohomology $HF^*_{\mathrm{loc}}(\rho^{-1}, + )$ that appears in the spectral sequence is isomorphic to the singular cohomology $H^*(\partial M;\mathbb{K})$, following McLean \cite{Mc19}, Fernandez de Bobabilla and Pe{\l}ka \cite{BP}.

To see this, we need some preparations.
First, our degree $\deg$ is different from the degree $\deg_{\#}$ used in \cite{Mc19}. Indeed, those two are related by
	\[  \deg = n + \deg_{\#}.\]
	This means that there is an isomorphism
	\begin{equation}
		\label{eq:McLean1}
		HF^*(\rho^{-1}, \pm ) \cong HF^{*-n}_{\#} (\rho^{-1}, \pm),
	\end{equation}
	where $HF^{*}_{\#} (\rho^{-1}, \pm)$ is the fixed point Floer cohomology defined using the degree convention $\deg_{\#}$ in \cite{Mc19}.

Second,  we already have seen that there is a perturbation $\check{\rho}^{-1}$ of $\rho^{-1}$ with a negative Reeb flow near $\partial M$, i.e., there is a time-independent Hamiltonian $K$ linear with slope $\epsilon$ on $[1-\eta, 1] \times \partial M$ such that $\check{\rho}^{-1} = \rho^{-1} \circ \phi_{-K}^{1}$. 

\begin{lemma}
There is a perturbation $\widetilde{\rho^{-1}}$ of $\rho^{-1}$ with a positive Reeb flow near $\partial M$ such that the fixed locus $\widetilde{B}=\mathrm{Fix} \widetilde{\rho^{-1}}$ of $\widetilde{\rho^{-1}}$ is given by a connected submanifold of $M$ of codimension $0$, which is diffeomorphic to the cylinder $[a,b] \times \partial M$ for some real numbers $a<b$. Furthermore, there is an open neighborhood $\widetilde{U}$ of $\widetilde{B}$ in $M$ and a time-independent Hamiltonian $\widetilde{H} : \widetilde{U} \to [0,\infty)$ such that
\begin{itemize}
	\item $\widetilde{H}^{-1}(0) = \widetilde{B}$,
	\item $\widetilde{\rho^{-1}}|_{\widetilde{U}} = \phi_{\widetilde{H}}^1$, and
	\item for every $p\in \partial \widetilde{B}$, there is an open neighborhood $V_p$ of $p$ such that $H|_{V_p \setminus \overline{\widetilde{B}}} >0$.
\end{itemize}
\end{lemma}
\begin{proof}
We may consider a time-independent 
Hamiltonian $G$ on $M$ satisfying $G = g(r)$ on $[1-\eta,1] \times \partial M$ for some $g : [1-\eta, 1] \to \R$ such that
\begin{itemize}
\item $g'(1-\eta) =0$,
\item $g'(1) = 2\epsilon$,
\item $g''(r) \geq 0$, and
\item $g'(r) = \epsilon$ for $r\in \left[1-\frac{2\eta}{3}, 1-\frac{\eta}{3}\right]$.
\end{itemize}
Then $\widetilde{{\rho}^{-1}}$ defined by
\[ \widetilde{{\rho}^{-1}} = \check{\rho}^{-1} \circ \phi_G^1 = \rho^{-1} \circ \phi_{-K}^{1} \circ \phi_{G}^1\]
satisfies all the properties in the statement. For example, the fixed locus of $\widetilde{{\rho}^{-1}}$ is given by
\[\widetilde{B} := \left[1-\frac{2\eta}{3}, 1-\frac{\eta}{3}\right] \times \partial M\]
and it is contained in an open neighborhood
\[ \widetilde{U} = (1-\eta, 1) \times \partial M,\]
where $\widetilde{\rho^{-1}}$ is given by the time-$1$ flow of the time-independent Hamiltonian
\[ \widetilde{H}  = G - K - C \]
for the constant $C = G(1, y) - K(1,y)$, $y \in \partial M$.
\end{proof}

Now, if we consider the local fixed point Floer cohomology $HF^*_{\mathrm{loc}} (\rho^{-1} , +)$ localized at $\widetilde{U}$ and the degree adjustment \eqref{eq:McLean1} explained above, then the argument in the proof of \cite[Lemma 9.4]{Mc19}, \cite[Proposition 6.3]{BP} shows that there is an isomorphism
\begin{equation}
\label{eq:McLean2}
HF^{*}_{\mathrm{loc}} (\rho^{-1}, +) \cong H_{2n - * - \mu_{CZ}(\widetilde{B})} (\widetilde{B}, \partial \widetilde{B};\mathbb{K}).
\end{equation}

Almost the same argument as in Subsection \ref{subsubsection:grading} can be applied here to show that $\mu_{CZ}(\widetilde{B})=0$. Furthermore, since $\widetilde{B} \cong [a,b] \times \partial M$ and $\partial \widetilde{B} \cong  \{b\} \times \partial M \bigsqcup \{a\} \times \partial M$, the long exact sequence of homology associated with the triple $([a,b]\times \partial M, \{a, b\} \times \partial M, \{a\} \times \partial M)$ shows that
\begin{equation}
\label{eq:les}
H_{2n - *} (\widetilde{B}, \partial \widetilde{B};\mathbb{K}) \cong HF_{2n-1-*}(\partial M ;\mathbb{K})\cong HF^*(\partial M ;\mathbb{K}),
\end{equation}
where the second isomorphism is the Poincar\'{e} duality.

Finally, since the fixed locus of $\widetilde{\rho^{-1}}$ is given by $\widetilde{B} ( \subset \widetilde{U}$), there is no difference between the fixed point Floer cohomology $HF^* (\rho^{-1},+)$ and the local fixed point Floer cohomology $HF^*_{\mathrm{loc}} (\rho^{-1} , +)$. Consequently, combining \eqref{eq:McLean2}, \eqref{eq:les} and $\mu_{CZ}(\widetilde{B}) =0$, we obtain an isomorphism
\[ HF^*(\rho^{-1},+) \cong H^*(\partial M ;\mathbb{K}).\]

\subsection{Grading}
\label{subsubsection:grading}
Let us discuss the grading on $SH^*(\check{\rho}^{-1})$ and show that $\mu_{CZ}(\Sigma) =0$ as asserted above.

First note that $c_1\big(f^{-1} ( B_{\delta}^*(0) \setminus f^{-1}(0)\big) =0$ (for a small $\delta >0$) as there is the standard non-vanishing holomorphic volume form $dz_1 \wedge \dots \wedge dz_n$ on $\C^n \setminus f^{-1}(0)$
(here $(z_1,\dots,z_n)$ is the standard coordinate of $\C^n$).
Furthermore, there exists a fiberwise complex volume form $\overline{\Omega}^{\mathrm{ver}}_{\delta}$ on the vertical tangent bundle $T^\mathrm{ver}f^{-1}(B_{\delta}^*(0)) = \ker df$ over $f^{-1}(B_{\delta}^*(0))$ such that
\begin{equation}\label{eq:residualcomplexvolumeform}
	dz_1\wedge \dots \wedge dz_n = \overline{\Omega}^\mathrm{ver}_{\delta} \wedge df.
\end{equation}
 Also  observe that the symplectic monodromy trivialization $\Phi$ on $f^{-1}(S_{\delta})$ identifies $f^{-1}(S_{\delta})$ with the mapping torus $T_{\rho_{\delta}}$ as a fibration over $S^1$.
 Recall that a non-vanishing complex volume form on the vertical tangent bundle $T^\mathrm{ver} T_{\rho_{\delta}}$ determines a grading on $(M_{\delta},\lambda_{\delta},\rho_{\delta})$ (see Remark \ref{remark:abstractcontactopenbook}).
Therefore, the restriction of $\overline{\Omega}^\mathrm{ver}_{\delta}$ to $f^{-1}(S_{\delta}) \cong T_{\rho_{\delta}}$ induces a grading on $(M_{\delta},\lambda_{\delta}, \rho_{\delta})$ (see \cite[Setting 5.11]{BP} for more details),
and thus  a grading on $(M,\lambda,\rho)$ associated with the Milnor fiber at radius zero
\cite[Corollary 5.16]{BP}.
Observe further that $\overline{\Omega}^\mathrm{ver}_{\delta}$ is well-defined away from the critical point $0 \in f^{-1} (B_{\delta}(0)) \subset \C^n$ by \eqref{eq:residualcomplexvolumeform}.

From Remark \ref{remark:abstractcontactopenbook}, this grading determines a non-vanishing complex volume form $\overline{\Omega}$ on the vertical tangent bundle of $T^\mathrm{ver} T_{\rho}$ uniquely up to homotopy. In other words, there is a 1-parameter family of complex volume forms $\{\overline{\Omega}_t\}_{t\in [0,1]}$ on $\widehat{M}$ such that $\overline{\Omega}_1 = \rho^* \overline{\Omega}_0.$

Next, the fact that $(M,\lambda,\rho)$ is graded implies that  $(M,\lambda,\rho^{-1})$ is also graded. Indeed, we have $\widetilde{\Omega}=\{\widetilde{\Omega}_t\}_{t \in [0,1]}$ on the vertical tangent bundle $T^\mathrm{ver} T_{\rho^{-1}}$ given by
\begin{equation}\label{eq:tildeomega}
\widetilde{\Omega}_t = \overline{\Omega}_{1-t},
\end{equation}
which satisfies $(\rho^{-1})^* \widetilde{\Omega}_0 = \widetilde{\Omega}_1$.  

Finally, to endow the Floer complex $SC^*(\check{\rho}^{-1})$ with a grading, we define a complex volume form ${\Omega} = \{\Omega_t\}_{t \in [0,1]}$ on the vertical tangent bundle $T^\mathrm{ver} T_{\check{\rho}^{-1}}$ by $\Omega_t =  (\phi_{-K}^t)^*\widetilde{\Omega}_t,$ where $K$ is from \eqref{eq:rhoK}.

As the Morse-Bott component $\Sigma$ lies in the collar neighborhood of $\partial M$, we need to
discuss a collar trivialization of our Milnor fibration and the complex volume form thereof.
Recall from \cite[Setting 5.11]{BP} that the authors found an open subset $W$ of $\C^n$ containing $0$ that satisfies
\begin{itemize}
\item its closure $\overline{W}$ is compact,
\item its boundary is a manifold, and
\item $f^{-1}(0) \pitchfork \partial W$.
\end{itemize}
Then modifying the restriction of $f$ to $\overline{W} \cap f^{-1}(B_{\delta}^*(0))$, they constructed a subset $N_{\delta}$ of $W$ such that the restriction of $f$ to $N_{\delta}$ is a Liouville fibration $f: N_{\delta} \to B_{\delta}^*(0)$. They further showed that there is a collar neighborhood $C^*_{\delta}$ of the horizontal boundary $\partial_{hor} N_{\delta}$ with a collar trivialization ${F}_{\delta}^*$:
\[ {F}_{\delta}^* : D_\delta \times B_{\delta}^*(0) \xrightarrow{\sim} C^*_{\delta},\]
for some collar neighborhood $D_{\delta}$ of $\partial M_{\delta}$ in $M_{\delta}$, for which the symplectic monodromy $\Phi : M _{\delta} \times [0,1] \to f^{-1}(S_{\delta})$ satisfies (for all $x\in D_{\delta}$ and $t\in [0,1]$)
\begin{equation}\label{eq:identityoncollar}
	\Phi (x, t) = {F}_{\delta}^* (x, \delta e^{2\pi i t}).
\end{equation}

Moreover, their construction further says that the collar trivialization ${F}_{\delta}$ extends even over $f^{-1}(0)$
since the origin is the unique critical point inside the Milnor ball. Hence there is a collar neighborhood $D_0$ of the boundary of the singular fiber $N_{\delta} \cap f^{-1}(0)$ so that the collar neighborhood $C_{\delta}^*$ of $\partial N_{\delta}$ can be extended over $f^{-1}(0)$ by adding $D_0$:
\[ C_{\delta}= C^*_{\delta} \cup D_0.\]
Therefore, the collar trivialization $F_{\delta}^*$ extends to a collar trivialization ${F}_{\delta}$ for $C_{\delta}$:
\begin{equation}\label{eq:collartrivializationdelta}
	{F}_{\delta} : D_{\delta} \times B_{\delta}(0) \xrightarrow{\sim} C_{\delta}.
\end{equation}

An analogous collar trivialization exists for the Milnor fibration at radius zero. Consider the A'Campo space $A$ and the associated fibration $f_A : A \to B_{\delta,\log}(0)$, whose fiber is denoted by $M$. By \cite[Setting 5.15]{BP},  there is a collar neighborhood ${C}$ of the horizontal boundary of a compact subdomain $N$ of $A$ obtained as the fiber product of the  projection map $B_{\delta,\log}(0) \to B_{\delta}(0)$ and the fibration $f: C_{\delta} \to B_{\delta}(0)$:
\[ 
\begin{tikzcd}
{C} \arrow{r}{} \arrow[swap]{d}{f_A}&  C_{\delta} \arrow{d}{f} \\
B_{\delta,\log}(0) \arrow{r}{} & B_{\delta}(0).
\end{tikzcd}
\]

 Consequently, there is a collar neighborhood $D$ of $\partial M \cap N$ in $M$ identified with the collar neighborhood $D_0$ of the boundary of $N_{\delta} \cap f^{-1}(0)$ by definition and a collar trivialization of ${C}$
\begin{equation}\label{eq:collarneighborhoodofrho}
	{F} : D \times B_{\delta,\log}(0) \xrightarrow{\sim} {C}
\end{equation}
obtained from the collar trivialization ${F}_{\delta}$ \eqref{eq:collartrivializationdelta}
satisfying $\mathrm{pr}_{B_{\delta,\log}(0)} = f_A \circ {F}$.

Identifying $\partial B_{\delta,\log}(0)$ with $S^1=\R/\Z$, then $N\cap f_A^{-1}(\partial B_{\delta,\log}(0)) (\xrightarrow{f_A} \partial B_{\delta,\log}(0) =S^1)$ is identified with the mapping torus $T_{\rho}$ as a fibration over $S^1$.

As we work with $\rho^{-1}$, we use the diffeomorphism $\sigma : T_{\rho} \to T_{\rho^{-1}}, [(t,x)] \mapsto [(1-t,x)]$ that lifts the orientation-reversing diffeomorphism $S^1 \to S^1, t\mapsto 1-t$ and the collar trivialization ${F}$ to obtain the collar neighborhood $E = \sigma({C}) \subset T_{\rho^{-1}}$ of its horizontal boundary and consequently a collar trivialization 
\begin{equation}\label{eq:collartrivializationzero}
\widetilde{F}  = \sigma \circ {F} : D \times S^1 \to E
\end{equation}
of $E$ in $T_{\rho^{-1}}$. Then we observe the following.
\begin{lemma}\label{lemma:timeindependentoncollar}
	The complex volume form $\widetilde{\Omega} = \{\widetilde{\Omega}_t\}_{t \in [0,1]}$ \eqref{eq:tildeomega} on the vertical tangent bundle $T^\mathrm{ver} T_{\rho^{-1}}$ is homotopic to a complex volume form $\widetilde{\Omega}' = \{\widetilde{\Omega}'_t \}_{t \in [0,1]}$ that is independent of $t \in [0,1]$ on the collar neighborhood $E$ in the sense that the complex volume form on $D$ given by
	\[ \widetilde{F}^* \widetilde{\Omega}' |_{D \times \{ t\}}\]
	is independent of $t \in [0,1]$ as a complex volume form on $D$, where each fiber $D \times \{t\}$ is identified with $D$.
\end{lemma}
\begin{proof} By the definition of $\widetilde{\Omega}_t$ in \eqref{eq:tildeomega}, it suffices to show that the complex volume form $\overline{\Omega} = \{ \overline{\Omega}_t \}_{t\in [0,1]}$ on the vertical tangent bundle $T^\mathrm{ver} T_{\rho}$ is homotopic to a complex volume form $\overline{\Omega}' = \{ \overline{\Omega}'_t\}_{t\in [0,1]}$ such that ${F}^* \overline{\Omega}'|_{D \times \{t\}}$ is independent of $t \in [0,1]$ as a complex volume form on $D$. We prove this assertion for a proof.
	
Since the collar trivialization ${F}_{\delta}$ \eqref{eq:collartrivializationdelta} extends over $B_{\delta}(0)$ and the fiberwise complex volume form $\overline{\Omega}^\mathrm{ver}_{\delta}$ is well-defined over $C_{\delta}$, we deduce that $\left\{ {F_{\delta}^*\overline{\Omega}}|_{D_{\delta} \times \{t\}}\right\}_{t\in [0,1]}$ is homotopic to the 1-parameter family of constant complex volume form on $D_\delta$.
	
	Now \cite[Corollary 5.16]{BP} says that the graded abstract open book $(M,\lambda,\rho)$ at radius zero is graded isotopic to the graded abstract open book $(M_{\delta}, \lambda_{\delta}, \rho_{\delta})$, whose grading is induced by the complex volume form $\{ \overline{\Omega}_{\delta,t}\}_{t\in [0,1]}$ on $T^\mathrm{ver} T_{\rho_{\delta}}$ as explained above. Recall that the collar neighborhood ${C}$ of $T_{\rho}$ that appears in \eqref{eq:collarneighborhoodofrho} is obtained as the fiber product of the obvious projection map $B_{\delta,\log}(0) \to B_{\delta}(0)$ and the fibration $f: C_{\delta} \to B_{\delta}(0)$. Furthermore the collar trivialization ${F}$ \eqref{eq:collartrivializationzero} is obtained from the collar trivialization ${F}_{\delta}$ \eqref{eq:collartrivializationdelta}. It follows that the restriction of complex volume form $\{\overline{\Omega}_t\}_{t \in [0,1]}$ on the vertical tangent bundle $T^\mathrm{ver} T_{\rho}$ is homotopic to a complex volume form $\{\overline{\Omega}'_t\}_{t \in [0,1]}$ such that ${F}^* {\overline{\Omega}'_t}|_{D \times \{t\}}$ is independent of $t\in [0,1]$. This completes the proof.
\end{proof}
Finally, using the above lemma, we prove that $\mu_{CZ}(\Sigma)=0$ as desired.
\begin{prop}\label{prop:index}
Let $\Sigma$ be the Morse-Bott component of twisted Hamiltonian orbits $\mathcal{P}({\check{\rho}^{-1}, H})$ for a quadratic Hamiltonian $H \in \mathcal{H}_{\check{\rho}^{-1}}$ discussed in Subsection \ref{subsec:Gamma}. Then,
\[ \mu_{CZ} (\Sigma) = 0.\]
\end{prop}
\begin{proof}
First of all, we may shrink ${C}$ so that the corresponding collar neighborhood $D$ of $\partial M$ is included in the region $[1-\eta, 1] \times \partial M$ where $K$ is linear with slope $\epsilon$. It follows that, on $D \subset [1-\eta, 1] \times \partial M$, $\phi_{-K}^t$ is given by $\mathrm{Id} \times \phi_{R}^{-\epsilon t}$ for all $t\in [0,1]$. As a consequence of Lemma \ref{lemma:timeindependentoncollar} and the above observation, the complex volume form $\Omega= \{ \Omega_t\}_{t \in [0,1]}$ on the vertical tangent bundle $T^\mathrm{ver} T_{\check{\rho}^{-1}}$ can be homotoped so that
\begin{equation}\label{eq:complexvolumeform}
F^* \Omega|_{D \times \{ t\}} = (\mathrm{Id} \times \phi_{R}^{-\epsilon t})^* F^* \Omega|_{D \times \{0\}},
\end{equation}
where both $D \times \{t\}$ and $D\times \{0\}$ are identified with $D$.
Since the grading depends only on the homotopy type of complex volume form on the vertical tangent bundle $T^\mathrm{ver} T_{\check{\rho}^{-1}}$, we may assume that $\Omega = \{ \Omega_t\}_{t\in [0,1]}$ already satisfies \eqref{eq:complexvolumeform}. 

For a quadratic Hamiltonian $H \in \mathcal{H}_{\widetilde{\rho}^{-1}}$ as chosen in Subsection \ref{subsec:Gamma}, let $\gamma \in \mathcal{P}({\check{\rho}^{-1};H})$ be a twisted Hamiltonian orbit belonging to the Morse-Bott component $\Sigma$. Note that the image of $\gamma$ lies in $D \cup [1,\infty) \times \partial M$. Then, we may choose a trivialization $\Psi: [0,1] \times \R^{2n} \to \gamma^* M$ in such a way that
\[ \Psi_t = d (\mathrm{Id} \times \phi_R^{\epsilon t})_* \Psi_0\]
in $T_{\check{\rho}^{-1}}$. Therefore \eqref{eq:compatibility1} is satisfied for $\Psi$.
Moreover, it satisfies \eqref{eq:compatibility2} due to \eqref{eq:complexvolumeform}. Now it follows that the path of symplectic matrices $\Lambda$ given by
\[ \Lambda(t) = \Psi_t^{-1} \circ d (\phi_H^{t}) \circ \Psi_0 = \Psi_t^{-1} \circ d (\mathrm{Id} \times \phi_R^{\epsilon t}) \circ \Psi_0 \]
is constantly the identity matrix, which implies that the index  $\mu_{CZ}(\Sigma)$ is zero.
\end{proof}

\section{Variation operator $\mathcal {V}$ via the closed-open map}
\label{sec:symplectic var}

We have defined the cohomology class $\Gamma \in SH^0(\rho^{-1})$ in Definition \ref{defn:gamma}.
The goal of this section is to construct the variation operator $\mathcal{V}$ via the action of $\Gamma$ on Lagrangian Floer theory through the closed-open map
\[\mathcal{CO}^\rho_L : SH^* \left(\rho^{-1}\right) \to HW^* \left(L, \rho(L)\right). \]
In analogy with the definition of $\Gamma$, we define $\Gamma_L \in CW^*(L, \rho(L))$ for each Lagrangian $L$. This will be useful throughout the paper.
It turns out to be equal to the lowest action value term of $\mathcal{CO}^\rho_L(\Gamma)$.

\subsection{Wrapped Floer cohomology}
\label{subsec: WF}
	From now on, a Lagrangian $L$ we consider is always assumed to have the following properties: 
	\begin{itemize}
		\item it is orientable and spin, 
		\item it satisfies $2c_1(M, L)=0$, and
		\item it is either closed or it has cylindrical ends at $\partial M$. 
	\end{itemize}
	
	We briefly recall the definition of wrapped Floer complex (see \cite{A10}).
	We take $H$ a time-independent Hamiltonian that is quadratic at infinity. 
	First, we choose a regular Floer datum $(H_{L_1, L_2}, J_{L_1, L_2})$ for a pair of Lagrangians $(L_1, L_2)$.  
	$H_{L_1, L_2}$ is a small, possibly time dependent perturbation of $H$ so that Hamiltonian chords from $L_0$ to $L_1$ are non-degenerate.
	We denote the set of non-degenerate Hamiltonian chords from $L_0$ to $L_1$ by
		\[\mathcal P(L_1, L_2 ; H_{L_1, L_2}).\]  
	$J_{L_1, L_2}$ denotes a time-dependent almost complex structure of contact type at infinity. 
	For a generic choice of $J_{L_1, L_2}$, every element of the following moduli space of pseudo-holomorphic strips is regular; 
		\begin{align*}
		\mathcal M_{d_{CW}} \left(\xi_+, \xi_-; H_{L_1, L_2}, J_{L_1, L_2}\right):=
		\left\{u: \R\times [0,1] \to M  \;\middle|\;			\begin{array}{l}
			\partial_su + J_{L_1, L_2}\left(\partial_tu-X_{H_{L_1, L_2}}(u)\right) =0, \\
			\lim_{s \to -\infty} u(s,t) = \xi_-(t), \hspace{5pt} \lim_{s\to \infty} u(s,t) = \xi_+(t), \\
			u(s,0) \in L_1, \hspace{5pt} u(s,1) \in L_2
			\end{array}
		\right\}.
		\end{align*}
	
	The wrapped Floer complex of $L_1$ and $L_2$ is defined as 
		\[CW^*(L_1, L_2; H_{L_1, L_2}, J_{L_1, L_2}):= \bigoplus_{\xi \in \mathcal P(L_0, L_1;H_{L_1, L_2})} \mathbb{K} \cdot \xi,\]
	where the grading is given by the Maslov index. 
	The Floer differential $d_{CW}$ is defined as a signed counting of rigid pseudo-holomorphic strips modulo translational symmetry:
	\[d_{CW}(\xi_+) := \sum_{\deg(\xi_-) = \deg(\xi_+)+1} \#\left[\mathcal M_{d_{CW}} \left(\xi_+, \xi_-;  H_{L_1, L_2}, J_{L_1, L_2}\right)/ \R\right]\cdot \xi_-.\]
	
\subsubsection{The induced grading}
\label{subsubsec: induced grading}
	Recall that we have equipped the Milnor fiber $M$ with a 1-parameter family of complex volume form $\{\Omega_t \}$ in Section \ref{subsubsection:grading}.
	Suppose that $L\subset M = M_0$ is graded with respect to $\Omega_{0}$, which means that $L$ carries a particular choice of a lift $\Lambda_0$ of a squared phase function 
	$\overline\Lambda_0: L \to S^1 $:
		\[
		\begin{tikzcd}
		L \arrow[r, dotted, "\Lambda_0"]\arrow[dr, "\overline \Lambda_0"'] & \R \arrow{d}\\
		& S^1
		\end{tikzcd},
			\hspace{5pt}
		\overline\Lambda_0(p) = \left( \frac{\Omega_0(\Lambda^n T_pL)}{\vert\Omega_0(\Lambda^n T_pL)\vert} \right)^2.
		\]
	We have a relative family of squared phase function $\overline\Lambda_t: L \times [0, 1]\to S^1$ by
	replacing $\Omega_0$ with $\Omega_t$.
	This can be viewed as a homotopy between $\overline \Lambda_{0}$ and other $\overline \Lambda_{t}$. 
	We conclude that we have a one-parameter family $\Lambda_{t}: L\times [0,1] \to \R$ of lifts.
	Since $\rho^*\Omega_{t_0} = \Omega_{t_0+1}$, one obtains a canonical grading $(\rho^{-1})^*\Lambda_{1}$ on $\rho(L)$.	
	 
	Let $\rho_t: M_0 \to M_t$ be a symplectic parallel transport \cite[Chapter(8k)]{S08} around a circle in the base.
	Then one obtains a moving boundary condition
		\[\mathcal L_t=\left(\rho_t (L), (\rho^{-1}_{t})^* \Lambda_t\right)\] 
	from $L_0=L$ to $L_1= \rho_1(L)$.
The grading structure is transported accordingly because our $\{\Omega_t\}$ is defined as a family of complex volume forms induced by the fibration. 
	This boundary condition can be equivalently expressed as a `branch cut' version. By abuse of notation, we write
	\begin{align}
	\label{defn:moving bd}
	\mathcal L_t=\left\{ 
	\begin{array}{ll} 
	(L, \Lambda_t) & t\in (0, t_0), \\
	\left(\rho(L), (\rho^{-1}_{t})^* \Lambda_t\right) & t\in (t_0, 1)
	\end{array}
	\right.
	\end{align}
	with a jump at $t_0 \in (0,1)$.
	These correspond to two equivalent descriptions of pseudo-holomorphic curves twisted by $\rho$.  
	When a boundary condition includes a branch cut, we may place the branch cut on the domain of the curve and twist the map along it by $\rho$, causing the boundary condition to jump at the branch cut.
	Alternatively, we may interpret the same map as a pseudo-holomorphic section of a symplectic mapping torus associated with $\rho$ (with appropriate insertions):
		\[
		\begin{tikzcd}
		E_\rho:= \frac{M \times \R_s\times\R_t}{(s, t+1, x) \sim (s, t, \rho(x))} \arrow[r, "\pi"] & \R \times S^1
		\end{tikzcd}
		,
		\hspace{5pt}
		\pi(x,s,t) = (s, t).
		\]
	{\em We mostly stick to the `branch cut' description in the paper}. From now on, $L$ is understood as a Lagrangian submanifold equipped with a grading and a spin structure, while $\rho(L)$ is always considered as a graded Lagrangian $\left(\rho(L) , (\rho^{-1})^* \Lambda_1\right)$.
	
\subsection{Construction of $\Gamma_L$}
\label{subsec: Construction of GammaL}

	 Using Proposition \ref{prop:perturbation}, we choose a neighborhood $U:=(1-\eta, 1]\times \partial M\subset \mathrm{Fix}(\rho) $ of $\partial M$ on which $\check \rho$ is written as a Hamiltonian flow of $b$. 
	In particular, $\check \rho^{-1}$ have no fixed points in $M^\mathrm{in}:= M \setminus (1-\eta,1] \times \partial M$. 
	We extend this neighborhood, still denoted as $U$, in the symplectization of $M$, so that $U \simeq (1-\eta,1+\eta) \times \partial M $. 
	Suppose that a Lagrangian $L$ has nontrivial cylindrical ends at infinity, then we may write
	\[L\cap U \simeq (1-\eta, 1+\eta) \times \partial L \subset (1-\eta,1+\eta) \times \partial M \simeq U.\]
	We also write $\mathcal P\left(L, \rho(L); H_{L, \rho(L)}, U\right)$ for the subset of  $\mathcal P\left(L, \rho(L); H_{L, \rho(L)} \right)$ consisting of Hamiltonian chords in $U$. 
	The local Floer complex is defined as
		\[CF^*_\mathrm{loc}\left(L, \rho(L); H_{L, \rho(L)}, J_{L, \rho(L)} \right) = \bigoplus_{\xi \in \mathcal P \left(L, \rho(L); H_{L, \rho(L)}, U\right)} \mathbb{K} \cdot \xi,\]
	with its differential defined as a counting of local pseudo-holomorphic strips modulo translations. 
$$d_\mathrm{loc}(\xi_+) := \sum_{\xi_- \in \mathcal P\left(L, \rho(L); H_{L, \rho(L)}, U\right)} \#\left[\mathcal M_{d_\mathrm{loc}}\left(\xi_+, \xi_-; H_{L, \rho(L)}, J_{L, \rho(L)}, U\right)/\R\right] \cdot \xi_-, $$
$$	\mathcal M_{d_\mathrm{loc}}\left(\xi_+, \xi_-; H_{L, \rho(L)}, J_{L, \rho(L)}, U\right):= \left\{ u\in \mathcal M_{d_{CW}} \left(\xi_+, \xi_-; H_{L, \rho(L)}, J_{L, \rho(L)} \right) \;\middle|\;  u(\R\times [0,1]) \subset U \right\}. $$
	
	The next proposition justifies the definition of the local Floer complex. 
	
	\begin{prop}
	\label{prop:GammaL}
		For $\Vert \nabla(H_{L, \rho(L)}-H)\Vert \ll 1$,  $\left( CF^*_\mathrm{loc}\left(L, \rho(L); H_{L, \rho(L)}, J_{L, \rho(L)}\right), d_\mathrm{loc} \right)$ is a well-defined complex.  
		In particular, the cohomology of the complex, denoted as $HF^*_\mathrm{loc} \left(L, \rho(L)\right)$, is independent of the choice of $H_{L, \rho(L)}$ and $J_{L, \rho(L)}$. Moreover, there is an isomorphism
		\[HF^*_\mathrm{loc} \left(L, \rho(L)\right) \simeq H^*(L\cap U).\]
	\end{prop}
	
	\begin{proof}
	See \cite{Pozniak}, \cite{CFHW96} for a detailed proof.
	We sketch the proof of the last claim, which follows from an identification 
	\begin{align}
	\label{eqn: C and CF}
	CF^*_\mathrm{loc}\left(L, \rho(L); H_{L, \rho(L)}, J_{L, \rho(L)}\right) \simeq CM^*\left(g_L \vert _{L\cap U}\right)
	\end{align}
	 of chain complexes for a particular Hamiltonian $H_{L, \rho(L)}$ and a Morse function $g_L$. 
	 
	 The intersection of $L$ and $\rho(L)$ is highly degenerate; $\rho$ is the identity on the set $B_0$. 
	 To get rid of this degeneracy, we first perturb $\rho$ to $\check \rho$ by Hamiltonian function $b$ in Proposition \ref{prop:perturbation}. Then one can check that 
	\begin{align*}
	L\cap \check \rho(L) \cap U = \emptyset, \hspace{5pt} & \phi_H^1(L)\cap  \check \rho(L) \cap U = \{r_0\} \times \partial L
	\end{align*} 
	for a fixed $r_0 \in (1-\eta, 1+\eta)$.
	Next, we choose a $C^2$-small function $g_L$ on $M$ by the following procedure.
	Start with a positive $C^2$-small Morse function $g_{\partial L}$ on $\partial L$ and
	take smooth concave function $q =q(r)$ with the unique maximum at $r_0$. 
	Then, $g_{L\cap U} = g_{\partial L} \cdot q$ on $L\cap U$ has its critical points on $\{r_0\} \times \partial L$ and its gradients pointing inward at the boundary.
	We extend it to a $C^2$-small function $g_L$ on $M$ so that its gradients flow into $L$ along the normal direction of the Weinstein neighborhood of $L \cap U$. 
	
	Next, recall that $\check \rho = \phi^1_K\circ \rho$ for another time-independent Hamiltoanian $K$. Define a time-dependent Hamiltonian
		\begin{align*}
		\overline g_L(t, x)&:= g_L \circ \phi^{t}_K,\\
		H_{L, \rho(L) , \delta}&:= G+\delta \overline g_L + F_L, \hspace{5pt} G = H\circ\phi^t_K -K\circ \phi^t_K, \hspace{5pt} 0<\delta \ll 1.
		\end{align*}

	Here, $F_L$ is a generic $C^2$-small time-dependent perturbation whose support is disjoint from that of $\overline g_L$.
	As seen in \eqref{eq:bijection}, under the time dependent coordinate change $(\phi_{K}^{t})^{-1}: M \to M$, the set $\mathcal P (L, \rho(L); H_{L, \rho(L), \delta}, U)$ is in one-to-one correspondence with the set of Morse critical points of $g_L$ inside $U$ for a small enough $\delta$. 
	Also, Floer strips are in one-to-one correspondence with Morse trajectories between those critical points. 
	The term $F_L$ is chosen as in \ref{subsubsection:symplectic cohomology} so that $H_{L, \rho(L), \delta}$ non-degenerate. 
	Therefore, we get the desired identification (\ref{eqn: C and CF}) up to degree shift given by $\mu_\mathrm{CZ}(\partial L)$. An argument similar to Proposition \ref{prop:index} shows $\mu_\mathrm{CZ}(\partial L)=0$. 
	\end{proof}
	\begin{defn}
	For each Lagrangian $L$ with cylindrical ends, we define 
		\[\Gamma_L \in CF^0_{\mathrm{loc}}\left(L, \rho(L); H_{L, \rho(L), \delta}, J_{L, \rho(L)}\right)\]
	by the cochain corresponding to the fundamental cycle of $CM^*\left(g_L \vert _{L\cap U}\right)$. 
	\end{defn}
	Note that $\Gamma_L$ does not define a cocycle in $CW^*\left(L, \rho(L)\right)$ for general $L$.
	It does define a cocycle in the following situation, just as $\Gamma$ does. 
	\begin{prop}
		Further assume $L\cap \rho(L) = \emptyset$ except where $\rho=\mathrm{Id}$ near $\partial M$.
		Then $[\Gamma_L]$ defines a cocycle of $CW^*\left(L, \rho(L); H_{L, \rho(L)}\right)$.
	\end{prop}
	\begin{proof}
	The inclusion $CF_\mathrm{loc}^*\left(L, \rho(L); H_{L, \rho(L), \delta}\right) \to CW^*\left(L, \rho(L) \right)$ is a cochain map since $L$ and $\rho(L)$ does not have interior intersection points. The result follows. 
	\end{proof}

\subsection{Closed-open map and $\Gamma_L$}
	In this subsection, we describe a relation between $\Gamma$ and $\Gamma_L$ under the twisted closed-open map $\mathcal{CO}_L^\rho$.

	\subsubsection{Twisted closed-open map}
	We recall the definition of the twisted version of the closed-open map \cite{Sei01}:
	\begin{equation}
		\label{eq:co_L}
			\mathcal{CO}_L^\rho: SC^* \left(\rho^{-1}\right) \to CW^*\left(L, \rho(L)\right).
	\end{equation}
	Let $S$ be the standard disc with the following decorations:
		\begin{itemize}
			\item one interior marked point $0 \in D^2$ equipped with a positive cylindrical end $\eta^+$,
			\item one boundary marked point $-1\in \partial D^2$ equipped with a negative strip-like end $\epsilon^-$, and
			\item a branch cut $(0,1] \subset D^2$.
		\end{itemize}
	We use coordinate $z = e^{s+it}$ with $s\in (-\infty, 0], \hspace{3pt} t\in [0, 2\pi]$.
	The branch cut indicates that we distinguish $e^{s+2\pi i}$ from $e^{s+0\cdot i}$. 
	We choose a domain-dependent Hamiltonian function denoted by $H_S$,
	a domain-dependent almost complex structure $J_S$, and a closed one form $\alpha$ on $S$ which satisfy 
	\begin{align*}
	(\eta^+)^* H_S = H&,&(\eta^+)^* J_S = J&,&(\eta^+)^* \alpha = dt, \\
	(\epsilon^-)^* H_S= H_{L, \rho(L)}&,&(\epsilon^-)^*J_S = J_{L, \rho(L)}&,&(\epsilon^-)^*\alpha= dt 
	\end{align*}
	with $\alpha \vert_{\partial D^2} = 0$ and  $d_S(H_S \cdot \alpha)\leq 0$ holds for $\forall s\in S$. 
	
	\begin{remark}
	\label{rmk:choice of subclosed one form}
The condition $d_S(H_S \cdot \alpha)\leq 0$ is required to apply the maximum principle.
	In particular, $H_S$ must decrease along the radial direction of $D^2$.
	To make such a choice possible, we may assume that the perturbation data satisfy $ H \leq H_{L, \rho(L)} $ at every point. This is different from \cite{A10} where the Liouville rescaling trick is used at the price of changing the weight of the closed-open map equation.
	In the case of our interest, $H$ and $H_{L, \rho(L)}$ are obtained by adding a bounded perturbation term $\overline g$ and $\overline g_L$ to a fixed function.  
	We may assume $0 \leq \overline g \leq \delta_1 \leq  \overline g_L  \leq \delta_2 \ll 1 $  by adding constants and scaling. 
	\end{remark}

	Now, consider the following moduli space of pseudo-holomorphic curves 
		\begin{align*}
		\mathcal M_{\mathcal{CO}_L^\rho}\left( \gamma_+, \xi_-; H_S, J_S \right):=
			\left\{u: S \to M  \;\middle|\;
							\begin{array}{l}
				\left(du-X_{H_S}(u) \otimes \alpha \right)^{0,1}_{J_S}=0,\\
				\lim_{s\to -\infty}(\epsilon^-)^* u = \xi_-(t), \hspace{5pt} \lim_{s\to \infty} (\eta^+)^* u = \gamma_+(t), \\
				\rho\circ u(e^{s+2\pi i}) = u(e^{s+0\cdot i}),\\
				u(e^{it}) \in\mathcal L_t
				\end{array} 
			\right\}.
		\end{align*}
	The twisted closed-open map is defined by counting rigid elements of the moduli space: 
	\[\mathcal{CO}_L^\rho(\gamma_+): = \sum_{\xi_- \in \mathcal P(L, \rho(L); H_{L, \rho(L)})} \#\mathcal M_{\mathcal{CO}_L^\rho}\left( \gamma_+, \xi_-; H_S, J_S \right) \cdot \xi_- .\]	
	The following proposition implies that $\mathcal{CO}^\rho_L \left( \Gamma \right)$ is a degree zero cocycle.

	\begin{prop}
		\label{prop:co}
		$\mathcal{CO}^\rho_L$ is a cochain map.
	\end{prop}
	\begin{proof}
	This can be proved by a slight modification of the arguments in \cite[Subsection 5.2]{A10} involving pseudo-holomorphic curves with branch cuts glued via $\rho$.
	\end{proof}

\subsubsection{Local closed-open map}
	We will consider the moduli space of local pseudo-holomorphic curves as before. Recall that $\rho$ is the identity map on $U$. Now, consider the moduli space
		\begin{align*}
		\mathcal M\left(\gamma_+, \xi_-; H_S, J_S, U\right):= \left\{u\in \mathcal M\left(\gamma_+, \xi_-; H_S, J_S\right) \;\middle|\; u(S) \subset U\right\}
		\end{align*}
	for a specific Hamiltonian perturbation we used before and define a local version of the twisted closed-open map by
		\[\mathcal {CO}_{L, \mathrm{loc}}^\rho: CF^*_\mathrm{loc} \left(\rho^{-1}; H_{\delta_1}, J \right) \to CF^*_\mathrm{loc} \left(L, \rho(L); H_{L, \rho(L), \delta_2}, J_{L, \rho(L)}\right)\]
	for sufficiently small $\delta_1, \delta_2 \ll 1$. 
	\begin{lemma}
		For sufficiently small $0 \leq \delta_1, \delta_2\ll 1$, every element $u\in \mathcal M(\gamma_+, \xi_-; H_S, J_S, U)$ has uniformly bounded image away from the boundary of $U$. 
	\end{lemma}
	\begin{proof}
	See Proposition \ref{prop: locality} below. 
	\end{proof}
	In particular, $\mathcal{CO}_{L, \mathrm{loc}}^\rho$ is a well-defined chain map. 
	One of the main results of this section is a partial computation of $\mathcal{CO}_{L, \mathrm{loc}}^\rho(\Gamma)$. 
	To state the result, choose 
	\begin{itemize}
		\item a Morse function $f_M$ on $M$ such that $f_M \vert_U$ is still Morse and its negative gradient flow points inward along $\partial U$, and
		\item a Morse function $f_L$ on $L$ such that $f_L \vert_{L\cap U}$ is still Morse and its negative gradient flow points inward along $\partial(L\cap U)$.
	\end{itemize}
	\begin{thm}
	\label{thm: PSS CO diagram}
	There are cochain maps
		$$\Phi : CM^*\left(f_M\vert_U\right)\to CF^*_\mathrm{loc}\left(\rho^{-1}; H_{\delta_1}, J\right),$$
		$$\Phi_L :  CM^*\left(f_L\vert _{L\cap U}\right)\to CF^*_\mathrm{loc} \left(L, \rho(L); H_{L, \rho(L), \delta_2}, J_{L, \rho(L)}\right),$$
	which make the following diagram commute:
	\[
	\begin{tikzcd}[column sep=large]
		H^*\left(U, \C \right)\ar[r, "i^*"] \ar[d, "\Phi"]& H^*\left(L \cap U, \C \right)  \ar[d, "\Phi_L"]\\
		SH^*_\mathrm{loc}\left(\rho^{-1}\right) \ar[r, "\mathcal{CO}_{L,\mathrm{loc}}^\rho"] & HF^*_\mathrm{loc}\left(L, \rho(L) \right)
	\end{tikzcd}.
	\]
	Moreover, $\Phi$ and $\Phi_L$ are quasi-isomorphisms.
	\end{thm}
	The vertical homomorphisms in Theorem \ref{thm: PSS CO diagram} are versions of PSS homomorphisms, or more precisely, their local inverses. 
	We can ignore the twist by $\rho$ since $U \subset \mathrm{Fix}(\rho)$.
	At the level of cochains, 
	\[\Phi: CM^*\left( f_M\vert_U \right)\to CF^*_\mathrm{loc}\left(\rho^{-1}; H_{\delta_1}, J\right)\] 
	is defined as follows: consider the following data on a cylinder. 
	\begin{itemize}
		\item A regular Floer datum $(H_{\delta_1}, J)$ for $\rho^{-1}$, and
		\item a sub-closed one form $\alpha = a(s)dt$, where $a(s)$ is a smooth function only depending on $s$-coordinate decreasing from $1$ to $0$. 
	\end{itemize}
	Then we define 
		$$\Phi(p_+):= \sum_{\gamma_- \in \mathcal P(\rho^{-1}; H_{\delta_1}(u), U)} \#\mathcal M_\Phi \left(p_+, \gamma_-; f_M, H_{\delta_1}, J, U\right) \cdot \gamma_-,$$
		where
		$$\mathcal M_\Phi \left(p_+, \gamma_-; f_M, H_{\delta_1}, J, U\right) :=\left\{(u, l)  \;\middle|\; 
			\begin{array}{ll}
			u: \R\times S^1 \to M, &\partial_su +J(\partial _t u-X_{H_{\delta_1}}(u)\cdot a(s))=0,\\
			u(-\infty, t) = \gamma_-(t), & u(\infty, t) = l(0), \\
			u(\R \times S^1) \subset U,&\\
			l: [0, \infty) \to M, & l'(s)+\nabla f_M \circ l =0,\\
			l(\infty) = p_+, & l([0, \infty))\subset U
			\end{array}
		\right\}.$$
		Similarly, 
	\[\Phi_L :  CM^*\left(f_L\vert _{L\cap U}\right)\to CF^*_\mathrm{loc} \left(L, \rho(L); H_{L, \rho(L), \delta_2}, J_{L, \rho(L)}\right)\] 
	is defined as follows. Consider the following data on an infinite strip:
	\begin{itemize}
		\item  a Floer datum $(H_{L, \rho(L), \delta_2}, J_{L, \rho(L)})$ for $(L, \rho(L))$, and 
		\item  a sub-closed one form $\alpha = a(s)dt$, where $a(s)$ is a smooth function only depending on $s$-coordinate decreasing from $1$ to $0$. 
	\end{itemize}
	Then $\Phi_{L}$ is defined by
		$$\Phi_L(q_+) := \sum_{\gamma_- \in \mathcal P(L, \rho(L), H_{L, \rho(L), \delta_2}, U)} \#\mathcal M_{\Phi_L} \left(q_+, \gamma_-; f_L, H_{L, \rho(L), \delta_2}, J_{L, \rho(L)}, U\right) \cdot \gamma_-,$$
		where
		$$\mathcal M_{\Phi_L} \left(q_+, \xi_-; f_L, H_{L, \rho(L), \delta_2}, J_{L, \rho(L)}, U\right) :=\left\{(u, l)  \;\middle|\;
			\begin{array}{ll}
			u: \R \times [0,1] \to M, & \partial_su +J_{L, \rho(L)}(\partial _t u-X_{H_{L, \rho(L), \delta_2}}(u)\cdot a(s))=0, \\
			u(-\infty, t) = \xi(t), & u(\infty, t) = l(0), \\
			u(\R \times [0,1]) \subset U, & u(\R \times \{0\} \sqcup \R\times \{1\}) \subset L,\\
			l: [0, \infty) \to L, & l'(s)+\nabla f_L \circ l =0,\\
			l(\infty) = q_+, &  l([0, \infty))\subset L\cap U\\
			\end{array}
		\right\}.$$
	Meanwhile, the pull-back $i^*$ is defined by counting rigid pairs of Morse flows concatenated at a point; 
		$$i^*(q_+) := \sum_{p_- \in \mathrm{crit}(f_L) \cap U}\#\mathcal M_{i^*}\left(p_+,q_-; f_M, f_L, U \right) \cdot p_-,$$
		where
		$$\mathcal M_{i^*}\left(p_+,q_-; f_M, f_L, U \right) := \left\{ (l_1, l_2)  \;\middle|\;
			\begin{array}{lll}
			l_1: (-\infty, 0] \to M,  & l'_1(s)+\nabla f_M \circ l_1 =0,&\\
			l_1(\infty) =q_-, & l_1(0) =l_2(0), & l_1((-\infty, 0]) \subset U, \\
			l_2: [0, \infty) \to L,  & l'_2(s)+\nabla f_L \circ l_2 =0,\\
			l_2(0) = p, & l_2([0, \infty)) \subset L\cap U& 
			\end{array}
			\right\}.$$
	\begin{prop}
		For sufficiently small $0<\delta_i\ll 1, \hspace{5pt} (i=1,2)$, every element of each moduli space has uniformly bounded image away from the boundary of $U$. 
	\end{prop}
	\begin{proof}
		Gradient flow lines of $f$ and $f_L$ are confined in $U$ by our choice of functions. The boundedness of perturbed pseudo-holomorphic discs is covered by Proposition \ref{prop: locality} below. 
	\end{proof}	
	In particular, a usual degeneration argument is valid over $U$. 
	Therefore, $\Phi$, $\Phi_L$, and $i^*$ are well-defined cochain maps so that the diagram of cohomologies in Theorem \ref{thm: PSS CO diagram} makes sense. 
	Before going into the proof, we derive some corollaries of the theorem. 
	
	\begin{cor}
		Suppose that $L$ has cylindrical ends. 
		We have $\mathcal{CO}_{L, \mathrm{loc}}^\rho (\Gamma) = \Gamma_L$.
	\end{cor}
	\begin{proof}
	Theorem \ref{thm: PSS CO diagram} immediately implies $[\mathcal{CO}_\mathrm{L,loc}^\rho (\Gamma)] = [\Gamma_L]$. 
	This should hold at the cochain level because both cocycles are of degree zero and each cochain complex vanishes in negative degrees. 
	\end{proof}
	\begin{cor}
	\label{cor: Gamma and GammaL}
		Suppose $L$ has cylindrical ends. 
		We have 
			\[\mathcal{CO}^\rho_L(\Gamma) = \Gamma_L + \sum_{p\in L\cap \rho(L) \cap M^\mathrm{in}} a_p\cdot p\]
		with possibly non-zero coefficient $a_p$. If $L\cap \rho(L) = \emptyset$ except where $\rho=\mathrm{Id}$ near $\partial M$, then $\mathcal{CO}^\rho_L(\Gamma) = \Gamma_L$. 
	\end{cor}

\subsection{Proof of Theorem \ref{thm: PSS CO diagram}}
	We construct a homotopy between $\mathcal{CO}_\mathrm{loc}\circ \Phi$ and $\Phi_L\circ i^*$ in two steps. 
\subsubsection{First Homotopy}
\label{subsubsec: First homotopy}

	We define the first homotopy denoted by $K_I$.  We equip a positive half cylinder $[0, \infty)\times S^1$ with the following data parametrized by $R \in [1, \infty)$:
		\begin{itemize}
			\item a negative boundary marked point $(0,1) \in [0, \infty)\times S^1$, 
			\item a strip-like end $\epsilon^-: (-\infty, 0]\times [0,1] \to [0, \infty)\times S^1$ at the marked point $(0, 1)$ such that $\mathrm{Im}(\epsilon^-)\subset [0, 1/2]\times S^1$, and
			\item a domain dependent Hamiltonian function $H_{S,R}$ and a closed one form $\beta_R$ on the domain satisfying 
				\begin{enumerate}
					\item $(\epsilon^-)^*X_{H_{S, R}}\otimes\beta_R = X_{H_{L, \rho(L), \delta_2}}\otimes dt$,
					\item $X_{H_{S, R}} \otimes \beta_R \vert_{[1, R) \times S^1} = X_{H_{\delta_1}} \otimes dt$, 
					\item $X_{H_{S, R}} \otimes \beta_R \vert_{[R+1, \infty) \times S^1} =0$, 
					\item $\frac{\partial}{\partial s} H_{S, R} \leq 0$, and
					\item $\beta_R \vert_{\{0\}\times S^1}=0$.
				\end{enumerate}
		\end{itemize}
	The fourth condition (4) is satisfied after we manipulate our Morse perturbation term so that $\overline g \geq \overline g_L >0$.  
	Now, take a smooth family of diffeomorphisms $\Psi_R =\mathrm{exp} (-(\alpha_R(s)+it)) : [0, \infty) \times S^1 \xrightarrow{\sim} D^2\setminus \{0\}$ where $\alpha_R(s): \R \to \R$ a strictly increasing function such that $\alpha_1 = \mathrm{Id}$ and 
	\[\alpha_R'(s) = \left\{ 
	\begin{array}{ll} 
		1 & (0\leq s\leq \frac{1}{2}),\\
		\textrm{decreasing} & (\frac{1}{2} \leq s \leq 1 ), \\
		1/R & (1 \leq s\leq R),\\
		\textrm{increasing} & (R \leq s \leq R+\frac{1}{2}),\\
		1 & (R+\frac{1}{2} \leq s).
	\end{array}
	\right.\]
	In this way, $\Psi_R$ interpolates the standard biholomorphic map $\mathrm{exp}(-(s+it))$ on $\{s\in (0, 1/2) \sqcup (R+1/2, \infty)\}$ and a $1/R$-scale map $\mathrm{exp}(-(s/R+it))$ on $\{ s\in (1, R) \} $ along $\{s \in (1/2, 1)\sqcup (R, R+1/2)\}$. 
	We obtain a family of complex structures $j_R$ on the disc by pushing forward the standard complex structure on $[0, \infty) \times S^1$ along $\Psi_R$. 
	 The family $j_R$ describes a degeneration of the domain. Indeed, it degenerates away from the standard one $j_{R=1}$ as $R\to \infty$, as depicted in Figure \ref{fig:degen1}. 
	\begin{remark}
		The construction of $K_I$ is one half of the one that appeared in \cite[Theorem 1.5]{Alb05}.
		The author extended $\beta_R$ to $R<1$ in a non-monotone manner, which cannot be adapted to our situation. 
	\end{remark}
	
\begin{figure}[h]
\includegraphics[scale=1]{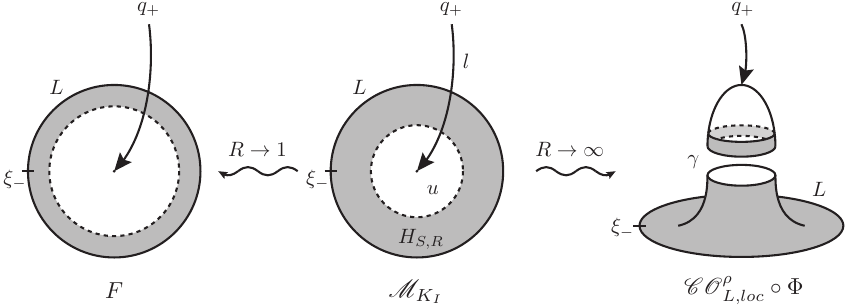}
\centering
\caption{Degeneration of $\mathcal M_{K_{I}}$}
\label{fig:degen1}
\end{figure}

	The moduli space defining $K_I$ consists of triples; 
	\begin{align*}
	\mathcal{M}_{K_I}\left(q_+, \xi_-; H_{S,R}, J_{S,R}, U\right):= \left\{ (R, l, u)  \;\middle|\;
		\begin{array}{ll}
		l: [0, \infty) \to U, & l'(s) + \nabla f \circ l = 0,\\
		l(\infty) = q_+, & l(0) = u(0),\\
		u: D^2 \to U, &( du - X_{H_{S, R}}(u) \otimes \beta_R )^{0,1}_{J_{S, R}}=0,\\
		u(\partial D^2)\subset L,   & \lim_{s\to -\infty}(\epsilon^-)^* u=\xi_-,\hspace{5pt} R \in [1, \infty)
		\end{array}
		\right\}. 
	\end{align*}
	We also define a homomorphism 
	\[F: CM^*(f\vert_U) \to CF^*_\mathrm{loc}\left(L, \rho(L); H_{L, \rho(L), \delta_2}, J_{L, \rho(L)}\right)\] 
	by the rigid counting of the moduli space for $K_I$ but restricted to $R=1$:
	\[\mathcal M_F \left(q_+, \xi_-; H_{S}, J_{S}, U\right):=\mathcal{M}_{K_I}\left(q_+, \xi_-; H_{S, R}, J_{S, R}, U\right) |_{R=1}.\] 
	The usual degeneration argument proves the following. 
	\begin{prop}(See \cite{Alb05} and Figure \ref{fig:degen1})
		$K_1\circ d_\mathrm{Morse}+d_{CF}\circ K_1 = \mathcal{CO}_\mathrm{loc}^\rho \circ \Phi-F$.
	\end{prop}

\subsubsection{Second homotopy}
	Now we want to construct a second homotopy $K_{II}$ between $F$ and $\Phi_L\circ i^*$. 
	We choose decorations on a half-cylinder parametrized by $r \in[0, \infty)$ as follows:
		\begin{itemize}
			\item a negative boundary marked point $(0,1) \in [0, \infty)\times S^1$, 
			\item a strip-like end $\epsilon^-: (-\infty, 0]\times[0,1] \to [0, \infty) \times S^1$ at $(0, 1)$ such that $\mathrm{Im}(\epsilon^-)\subset [0,1/2]\times S^1$,
			\item a neighborhood $W$ such that $\epsilon^-((-\infty, 0]\times[0,1]) \subset W \subset [0, 1] \times \{-\pi/2<\theta<\pi/2\}\subset [0, \infty)\times S^1$,
			\item a family of smooth bump functions $\theta_{S, r}: [0, \infty) \times S^1\to [0,1]$, $r\in [0,\infty)$ such that 
			\[\theta_{S, r}(s) = \left\{
			\begin{array}{ll}
			1 & (s\in \mathrm{Im}(\epsilon^-)),\\
			\zeta(r) & (s \in W^c)
			\end{array}
			\right.
			, \hspace{5pt}
			\zeta(t) = \left\{
			\begin{array}{ll}
			1 & (t < 1),\\
			0 & (2 \leq t)
			\end{array}
			\right.
			\]
			with $\theta_{S, 0}(s)$ constant function. 
			Here, $\zeta: [0, \infty) \to [0,1]$ is a smooth decreasing function from $1$ to $0$,
			\item a family of Hamiltonian perturbations $ H_{S, r}: = H_{S, R=1}\times \theta_{S,r}(s)$, $r\in [0,\infty)$, and
			\item a closed one-form $\beta = \beta_{R=1}$. 
		\end{itemize}
	These data have been chosen so that the following conditions are satisfied:
	\begin{enumerate}
		\item $(\epsilon^-)^*X_{H_{S, r}}\otimes\beta = X_{H_{L, \rho(L), \delta_2}}\otimes dt$ for $\forall r \in [0, \infty)$,
		\item $(H_{S, r}, \beta)\vert_{r=0}$ coincides with $(H_{S, R}, \beta_R)_{R=1}$, 
		\item $X_{H_{S, r}} \otimes \beta = 0$ for $\forall s\in W^c$ when $r \geq 2$, 
		\item $\frac{\partial}{\partial s} H_{S, r} \leq 0$ for $\forall r\in [0, \infty)$, and
		\item $\beta\vert_{\{0\}\times S^1} = 0$. 
	\end{enumerate}
	We import this data to the standard disc under the conformal diffeomorphism $\mathrm{exp}(-(s+it))$.
	The moduli space $\mathcal M_{K_{II}}$ is defined as the union of two different moduli spaces.
	The construction is motivated by pearl complexes, but since there are no $J$-holomorphic discs bounded by an exact Lagrangian submanifold, we only need to consider the following two moduli spaces.

	The first moduli space consists of triples, 
	\begin{align*}
	\mathcal M_{K_{II}, 1}\left(q_+, \xi_-; f_M, H_{S, r}, J_{S, r}\right):= \left\{ (r, l, u) \;\middle|\;
		\begin{array}{ll}
		 r\in [0, \infty), \\
		l: [0, \infty) \to U, & l'(s) + \nabla f_M \circ l = 0, \\
		l(\infty) = q_+, & l(0) = u(1-e^{-r}),\\
		u: D^2 \to U, &( du - X_{H_{S, R}}(u) \otimes \beta )^{0,1}_{J_{S, r}}=0,\\
		\lim_{s\to -\infty}(\epsilon^-)^* u=\xi_-, &  u(\partial D^2)\subset L\\
		\end{array}
		\right\},
	\end{align*}
	while the second moduli space $\mathcal M_{K_{II}, 2}(q, \xi)$ consists of quintuples,
	\small{$$
	\mathcal M_{K_{II}, 2}\left(q_+, \xi_-; f_M, f_L, H_{S, \infty}, J_{S, \infty}\right) := \left\{ (r, l_1, u_1, l_2, u_2)  \;\middle|\;
		\begin{array}{ll}
		l_1: [0, \infty) \to U, &l_1'(s) + \nabla f_M \circ l_1 = 0,\\
		l_1 (\infty) = q_+, & l_1 (0) = u_1 (1),\\
		u_1: D^2 \to U, & \partial_s u_1+ J\circ \partial_t u_1 =0, \\
		u_1(-1) = l_2(r), & r\in [0, \infty), \\
		l_2: [-r, r] \to U, & l_2'(s) + \nabla f_L \circ l_2 = 0,\\
		l_2(-r) = u_2(0),& l_2([-r,r]) \subset L \cap U,\\ 
		u_2 : D^2 \to U, & (du -X_{H_{S, \infty}}(u_2) \otimes \beta )^{0,1}_{J_{S, \infty}}=0,\\
		u_i(\partial D^2)\subset L, & \lim_{s\to -\infty}(\epsilon^-)^* u_2=\xi_-\\ 
		\end{array}
		\right\}.
	$$}\normalsize
	
	\begin{prop}
	\label{prop:second homotopy}
		$K_2\circ d_\mathrm{Morse}+d_{CF}\circ K_2 = F - \Phi_L\circ i^*$.
	\end{prop}
	\begin{proof}
		The boundary component of $\mathcal M_{K_{II}, 1}(q, \xi)$ corresponding to $r \to 0$ contributes to $F$.  
		
\begin{figure}[h]
\includegraphics[scale=1]{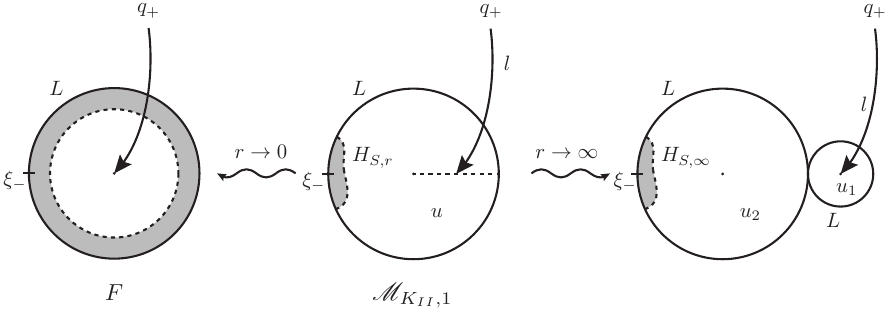}
\centering
\caption{Degeneration of $\mathcal M_{K_{II}, 1}$}
\label{fig:degen2}
\end{figure}

		Observe that $X_{H_{S, r}}\otimes \beta = 0$ around a neighborhood of $1\in D^2$ for all $r\geq 2$ due to the above requirements.
		In particular, Hamiltonian perturbation term around $l(0) = u(1-e^{-r})$ becomes zero. 
		Therefore, the limit of $\mathcal M_{K_{II}, 1}$ as $r \to \infty$ consists of pair $(l, u_1, u_2)$ such that 
		\begin{itemize}
		\item $l : [0, \infty) \to U$ is a Morse flowline of $f$ such that $l(\infty) = q_+$, 
		\item $u_1$ is a $J$-holomorphic disc bounded by $L$ such that $u_1(0) = l(0)$, and
		\item $u_2: D^2 \to U$ is a solution of perturbed Floer equation
			\begin{align*}
			\left( du -X_{H_{S, \infty}}(u_2) \otimes \beta \right)^{0,1}_{J_{S, \infty}}=0, \hspace{5pt}&u_2(\partial D^2) \subset L 
			\end{align*}
			such that $u_2(1) = u_1(-1)$ and $\lim_{s\to -\infty}(\epsilon^-)^* u_2 = \xi_-$.
		\end{itemize}
	Of course, $u_1$ is always constant because $L$ is exact. 
	
	At first, the limit of $\mathcal M_{K_{II}, 2}$ as $r\to 0$ cancels the limit of $\mathcal M_{K_{II}, 1}$ as $r \to \infty$.
	This procedure simply degenerates $l_2$ to a point and the results are the same as we just described.
	
\begin{figure}[h]
\includegraphics[scale=1]{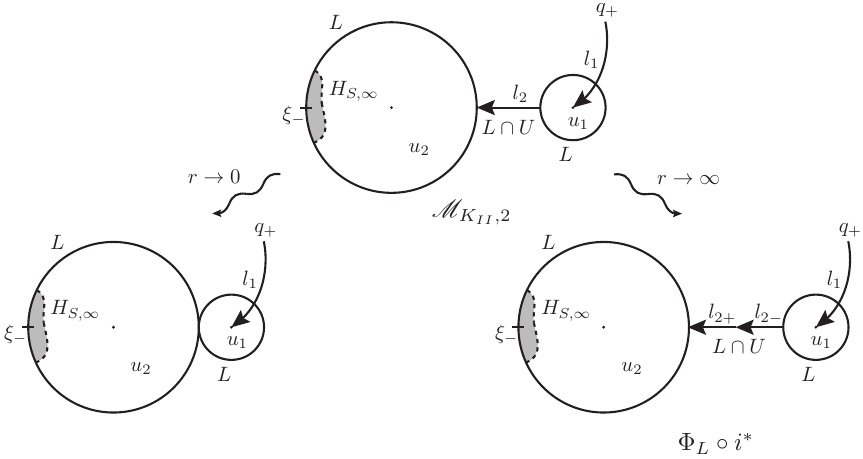}
\centering
\caption{Degeneration of $\mathcal M_{K_{II}, 2}$}
\label{fig:degen3}
\end{figure}

	On the other extreme, the $r \to \infty$ limit of $\mathcal M_{K_{II}, 2}(q, \xi)$ consists of quintuples $(l_1, u_1, l_{2-}, l_{2+}, u_2)$ where
		\begin{itemize}	
		\item $l_1 : [0, \infty) \to U$ is a Morse flowline of $f$ such that $l_1(\infty) = q_+$, 
		\item $u_1$ is a J-holomorphic disc bounded by $L$ such that $u_1(0) = l_1(0)$, 
		\item $l_{2-}: (-\infty, 0] \to L\cap U$ is a Morse flowline of $g_L$ such that $l_{2-}(0) = u_1(-1)$,
		\item $l_{2+}: [0, \infty) \to L\cap U$ is a Morse flowline of $g_L$ such that $l_{2+}(\infty) = l_{2-}(-\infty)$,  and 
		\item $u_2: D^2 \to U$ is a solution of perturbed Floer equation
			\begin{align*}
			\left( du -X_{H_{S, \infty}}(u_2) \otimes \beta \right)^{0,1}_{J_{S, \infty}}=0,\hspace{5pt} &u_2(\partial D^2) \subset L
			\end{align*}
			such that $u_2(1) = l_{2+}(0)$ and $\lim_{s\to -\infty}(\epsilon^-)^* u_2 = \xi_-$.
		\end{itemize}

	Note that $l_{2+}(\infty)=l_{2-}(\infty)$ is a Morse critical point of $f_L$. Again, $u_1$ is always constant because $L$ is an exact Lagrangian submanifold. It is easy to check that the contribution of such quintuples is exactly $\Phi_L\circ i^*$. 
	\end{proof}
	
\subsection{Uniform boundedness of perturbed holomorphic curves inside $U$}

We have repeatedly used the fact that various holomorphic curves in $U$ do not escape $U$. This idea first appeared in \cite{CFHW96} and led the authors to introduce the notion of local Floer cohomology, which has been employed in many works such as \cite{KvK16}, \cite{Mc19}, \cite{McR}.

	Let us recall the setting. 
	We choose a fixed collar neighborhood $U$ of $\partial M$ by the inverse Liouville flow. 
	We use the coordinate system $(r, x) \in (1-\eta,1+\eta) \times \partial M \cong U$. 
	The symplectomorphism $\check \rho$ is equal to a Hamiltonian diffeomorphism of $b$ as in Proposition \ref{prop:perturbation}. 
	We choose a Hamiltonian function $H=h(r)$ so that it only depends on $r$.
	Moreover, 
	\begin{itemize}
		\item $U$ is invariant under the Hamiltonian flow of $H$, which is a multiple of Reeb flow $h'(r) \cdot R$,
		\item if a (twisted) Hamiltonian orbit of $h$ is contained in $U$, then it is contained in $\{r_0\}\times \partial M$, and
		\item if Hamiltonian chords of $H$ between $L$ and $\rho (L)$ is contained in $U$, then it is also contained in $\{r_0\}\times\partial M$.
	\end{itemize}
	Next, we use
	\begin{itemize}
	\item a bounded time-dependent perturbation term $f_M$ to get $H_\delta := H+\delta f_M$,	
	\item a bounded time-dependent perturbation term $f_L$ to get $H_{L, \rho(L), \delta}:= H+\delta f_L$, and
	\item a perturbed pseudo-holomorphic curve $u$ with at most one input and at least one output among $\gamma \in \mathcal P(\rho^{-1}; H_\delta)$ and $\xi \in \mathcal P(L, \rho(L); H_{L, \rho(L), \delta})$. 
	(We abbreviate them as $u\in \mathcal M$.)
	\end{itemize}

	\begin{prop}
	\label{prop: locality}
		For any neighborhood $U^{\prime}$ such that $U \supset U^{\prime} \supset \{ r_{0} \} \times \partial M$, there exist a $\delta_{U^{\prime}}>0$ such that for all $\delta_{U^{\prime}} >\delta>0$, the following are true.
		\begin{enumerate}
		 	\item The image of each element of $\mathcal P (\rho^{-1}; H_{\delta} ,U)$ is already contained in $U^{\prime}$.
			\item The image of each element of $\mathcal P \left (L, \rho(L); H_{L, \rho(L), \delta} , U \right)$ is already contained in $U^{\prime}$.
			\item If the image of an element $u \in \mathcal M$ is contained in $U$, then it is also contained in $U^{\prime}$.  
		\end{enumerate}
	\end{prop}
	\begin{proof} We proceed with our proof in several steps. 
	
	{\em Step 1: prove (1) and (2).}
	We prove (1) by contradiction following \cite{CFHW96}. 
	If it is false, one can find an open set $U^{\prime}$ and a sequence $(\delta_n, \gamma_n)$ such that $\delta_n \to 0$ and $\gamma_n \in \mathcal P(\rho^{-1}; H_\delta, U)$ which are not contained in $U^{\prime}$ for any $n$. 
	By the compact Sobolev embedding, there is a subsequence $\gamma_{i_k}$ that converges to a Hamiltonian orbit $\gamma$ of a time-independent $H$.
	It contradicts to our assumption on $\gamma_n$ because $\gamma$ in completely contained in $\{r_0\} \times \partial M$ and in particular, inside $U^{\prime}$.  The same argument works for $(2)$ as well.
	 
	{\em Step 2: Limiting to an autonomous case.}
	We are going to apply a similar argument to $u\in \mathcal M$ with a little more care. 
	Suppose the claim is not true; then we may choose (after passing to a subsequence) a sequence of quadruples $(\delta_n, u_n, H_{S, n}, J_{S, n})$ such that $\delta_n \to 0$ and each $u_n$ is a regular solution of Floer equation
		\[\left(du_n-X_{H_{S,n}} (u_n)\otimes \alpha\right)^{0,1}_{J_{S, n}}=0\] 
	whose images are contained in $U$ but not in $U^{\prime}$, with $(\epsilon^i)^* H_{S, n} = H_{i, \delta_n}$ for $i=1,2$. 
	Here, $\alpha$ is an appropriate sub-closed one form. 
	Observe that a sequence $H_{S, n}$ is converging to an autonomous Hamiltonian $H$.
	By the usual convergence argument for the solution of an elliptic partial differential equation, we conclude that $u_n$ converges to $u$ which satisfies  
		\[\left(du-X_H (u) \otimes \alpha\right)^{0,1}_{J_S}=0,\]
	whose inputs or outputs, denoted as $\gamma$ and $\xi$, are Hamiltonian orbits or chords of $H$. 
		
	{\em Step3: Unwrapping.} 	
	We cancel out the effect of autonomous Hamiltonian $H$ by the unwrapping procedure. 
	By the assumption on $U$ and $H$, $\gamma$ and $\xi$ are contained in $\{r_0\}\times \partial M$ and share the common period $h'(r_0)$.
	We apply the time-dependent change of coordinate $\phi_{-H}^t: U \to U$. 
	Then $\tilde u = \phi_{-H}^t \circ u$ is a solution of 
	\[\left(d\tilde u-X_{\tilde H} ( \tilde u) \otimes \alpha\right)^{0,1}_{\tilde J_S}=0,\]
	where $\tilde J_S = (\phi_{-H}^t)^* J_S$ and $\tilde H = H\circ \phi_{-H}^t$. 
	Note that $\tilde J_S$ is still of contact type at the boundary. 
	Moreover, possible inputs or outputs, $\phi_{-H}^t \circ \gamma$ or $\phi_{-H}^t \circ \xi$, of $\tilde u$ are points of $\{r_0\}\times \partial M$ since $\phi_{-H}^t$ preserves the level set of $U$. 
	By assumption on $u_n$, the image of $\tilde u$ is not contained in $\phi_{-H}^t(U^{\prime})$. 
	It contradicts the following lemma and hence finishes the proof.
	\end{proof}
	\begin{lemma}
	\label{lem: about energy zero curve}
	Suppose $\tilde u: S \to U$ is an inhomogeneous pseudo-holomorphic curve satisfying 
	\[\left(d\tilde u-X_{\tilde H} ( \tilde u) \otimes \alpha\right)^{0,1}_{\tilde J_S}=0,\]
	where $J_S$ is a domain-dependent almost complex structure contact type at infinity, $H=h(r)$ is an autonomous Hamiltonian function that depends only on the coordinate $r$, and $\alpha$ is a sub-closed one form on $S$.  
	If the action values of possible inputs and outputs are zero, then the image $u$ is contained in an integral curve of the Reeb flow. 
	\end{lemma} 
	\begin{proof}
	For a contact type almost complex structure $J$, $\partial_r, R$ and $\ker \theta$ are mutually orthogonal with respect to $\langle -, - \rangle_J = \omega(-, J-)$. 
	Now write $u = (a, v): S \to (1-\eta, 1+\eta)\times \partial M$. 
	We decompose $du$ into three orthogonal summands as follows:
		\[du-X_{\tilde H}\otimes \alpha = \left( da, \pi_{\mathrm{ker}\theta} \circ dv, \pi_R\circ dv-h'(r) R\otimes \alpha  \right).\]
	Recall that energy inequality tells us  $0\leq E_\mathrm{geo}(u) \leq E_\mathrm{top}(u) = \mathcal A_{\tilde H}(\gamma_\mathrm{out})-\mathcal A_{\tilde H}(\gamma_\mathrm{in}) =0$.
	Therefore, the geometric energy of $u$ is zero and we must have 	
		\begin{align*}
		0=E_\mathrm{geo}(u) & = \int_S \left \Vert du-X_{\tilde H}\otimes \alpha \right\Vert^2\\
		&=\int_S \left \Vert da\right\Vert^2 + \left \Vert \pi_{\mathrm{ker}\theta}\circ dv \right\Vert^2 + \left \Vert \theta(dv)-h'(r)\otimes \alpha \right\Vert^2 .
		\end{align*}
		As $u$ is smooth, this implies that $a$ needs to be constant and $dv$ must be parallel to $R$. 
	\end{proof}	
	
	\begin{remark}
	A proof of (3) for Floer cylinders/strips can be found in \cite{CFHW96} \cite{Pozniak}.  The original proof does not apply because it exploited translational symmetries of a cylinder.
	\end{remark}
	
\subsubsection{Inverses of $\Phi$ and $\Phi_L$}
	Next, we see that $\Phi$ and $\Phi_L$ are isomorphisms and establish Theorem \ref{thm: PSS CO diagram}.
	Hamiltonian PSS morphisms are defined by the counting of the moduli space of pairs:
		\begin{align}
		\label{eqn:PSS}
		\mathcal M_\mathrm{PSS} (\gamma_+, q_-; H, J) :=\left\{(l, u)  \;\middle|\; 
			\begin{array}{ll}
				l: (-\infty, 0] \to M, & l'(s)+\nabla f_M \circ l =0,\hspace{5pt}\\
				u: \R\times S^1 \to M, & \partial_s u + J\left( \partial_t u-X_H (u) \cdot \check a(s)\right)=0,\\
				u(\infty, t) = \gamma_+(t), & l(-\infty) = q_-\\
			\end{array}
		\right\}, 
		\end{align}
	where $\check a(s)$ is an increasing cut-off function from $0$ to $1$. 
	The condition breaks the monotonicity of the Floer equation.
	The compactness of the moduli space is therefore not guaranteed because elements of the moduli space space may escape to infinity.
	We explain how to recover the compactness by showing that the solutions we consider do not escape to infinity, following \cite[Lemma 15.5]{Rit13} in the case of symplectic cohomology.

	Note that elements of the local Floer cohomology have arbitrary small periods.
	We can even use $\rho$ rather than a small perturbation $\check \rho$.
	The isomorphism \eqref{eq:bijection} provides a chain isomorphism
	\[CF_\mathrm{loc}^*(\check \rho^{-1}; H_\delta) \simeq CF^*_\mathrm{loc}(\rho^{-1}; \check H) \cong CF^*(U; \check H),\]
	\[H_\delta = H +\delta \overline g, \hspace{5pt} \check H=g=g_{\partial M} \times q(r).\]
	Here, we can choose $H$ and $g$ to satisfy 
		$$\frac{\partial}{\partial r}\check H =
			\begin{cases}
			-\epsilon & (r< 1-\eta),\\
			\epsilon & (1+\eta <r)
			\end{cases}, \hspace{5pt}
		\Vert \nabla \check H\Vert_\infty \ll 1$$	
so that its negative gradient flow points inward along the boundary of $U$, and Hamiltonian orbits of $\check H$ are precisely the Morse critical points. 

	\begin{prop}
	\label{prop:dichotomy}
		For a fixed open subset $W \supset U$, the image of the solution of the Floer equation 
			$u: \R \times S^1 \to M$,
			$$\partial_s u + J\left( \partial_t u -X_{\check H} (u) \otimes \check a(s)\right)=0,$$
			$$\lim_{s\to \infty} u(s, t) = p, \;\; \lim_{s\to -\infty} u(s, t) = q\in \mathrm{crit}(\check H)$$
		with $\check a(s)$ as in \eqref{eqn:PSS}, is either contained in $U$ or escapes $W$ for $\Vert \nabla \check H \Vert_\infty$ sufficiently small. 
	\end{prop}
	\begin{proof}
		See \cite[Lemma 15.5]{Rit13} for a detailed proof where the author considered $M^\mathrm{in}$ instead of $U$.  
		The key idea is that every solution of the Floer equation becomes $t$-independent, i.e.,  Morse flowlines of $\check H$ whenever $\Vert \check H \Vert_\infty$ is sufficiently small.  
		We can further confine the solution into $U$ because the Morse flowlines of $\check H$ flow into $U$.
	\end{proof}
	With this observation in mind, combined with the fact that $\phi_{H}^{-t}$ preserves level sets, we define $\Phi^{-1}$ by 
	$$\Phi^{-1}(\gamma_+) := \sum_{q_- \in \mathrm{crit}(f) \cap U} \#\mathcal M_{\Phi^{-1}} \left(\gamma_+, q_-; f_M, \check H, J\right) \cdot q_-,$$
	where 
	$$\mathcal M_{\Phi^{-1}} \left(\gamma_+, q_-; f_M, \check H, J\right) :=\left\{(l, u)  \;\middle|\; 
	\begin{array}{ll}
		u: \R\times S^1 \to M,  &\partial_s u + J\left( \partial_t u-X_{\check H}(u) \cdot \check a(s)\right)=0,\\
		u(\infty, t) = \gamma_+(t), & u(\R\times S^1) \subset U,\\
		l: (-\infty, 0] \to M, & l'(s)+\nabla f_M \circ l =0,\\
		l(-\infty) = q_-, & l(0)=u(-\infty, t), \hspace{5pt} \forall t \in S^1\\
	\end{array}
	\right\}.$$
	By Proposition \ref{prop:dichotomy}, we conclude that no elements of $\mathcal M_{\Phi^{-1}}$ are $C^0$-close to a solution of the same Floer equation which escapes $U$. 
	The element of $\mathcal M_{\Phi^{-1}}$ has a uniformly bounded image, and then we recover the compactness of the moduli. 
	\begin{prop}
		$\Phi^{-1}$ is a chain map and  $\Phi^{-1} \circ \Phi$ and $\Phi\circ \Phi^{-1}$ are homotopic to the identities.  	
	\end{prop}
	\begin{proof}
		The proof works the same as in \cite[Theorem 6.6]{Rit13}.
	\end{proof}
	This implies $\Phi: H^* (U) \to SH^*_\mathrm{loc}(\rho^{-1})$ is an isomorphism. 
	Using a similar argument, one can show that $\Phi^{-1}_L$ also exists and $\Phi_L: H^* (L \cap U) \to HW^*_\mathrm{loc}\left(L, \rho(L)\right)$ is an isomorphism. 
	This finishes the proof of Theorem \ref{thm: PSS CO diagram}.

\subsection{Variation operator $\mathcal V$}
\label{subsec:speculation}
	
	For a Milnor fiber $M$ of $f$, there is a homomorphism called {\em variation} defined as 
	\[ \mathrm{var}: H_{n-1}(M, \partial M) \to H_{n-1}(M), \hspace{5pt} l \mapsto \rho_*(l) - l. \]
	This homomorphism is known to be an isomorphism when $f$ defines an isolated singularity. The variation homomorphism $\mathrm{var}$ motivates the following definition.
	\begin{defn}[Definition \ref{defn:v}]
	\label{defn: variation}
	For each object $L$ of $\mathcal{WF}(M)$, we define its {\em variation} to be the twisted complex
	\[\mathcal V(L):= \left( \begin{tikzcd} L[1] \arrow[r,"\mathcal{CO}_L^\rho(\Gamma)"] & \rho(L)\end{tikzcd} \right).\]
	\end{defn}
	
	It is worth pointing out that the variation $\mathcal{V}$ in Definition \ref{defn: variation} extends to an endofunctor $\mathcal V: \WF(M) \to \WF(M)$. To be more precise, $\mathcal V$ is the mapping cone of the image of $\Gamma$ under a more general $\rho-$twisted closed-open map 
	\begin{equation}
	\label{eq: genCO}
	\mathcal{CO}^\rho: SC^*\left(\rho^{-1}\right) \to \hom^*_{\mathrm{End}(\WF(M))}\left(\mathrm{Id} , \rho_\sharp\right),
	\end{equation}
	where $\rho_\sharp$ denotes the endofunctor on $\WF(M)$ induced by $\rho$ and $\mathrm{End}(\WF(M))$ is the $A_{\infty}$-category of endofunctors on $\WF(M)$. Here, the morphism space $\hom^*_{\mathrm{End}(\WF(M))}(\mathrm{Id},\rho_\sharp)$ is given by 
	\[\hom^*_{\mathrm{End}(\WF(M))}(\mathrm{Id},\rho_\sharp) = CC^*(\WF(M),\hom (\mathrm{Id},\rho_\sharp)),\]
	the Hochschild cochain complex with coefficient in the bimodule  
	$$
	\hom(\mathrm{Id},\rho_\sharp):=(\mathrm{Id} \times \rho_\sharp)^*(\Delta_{\WF(M)}) 
	$$ 
	over $\WF(M)$. Then \eqref{eq: genCO} can be shown to be a chain map by modifying the arguments in \cite[Section 5.4]{Ga} as in the proof of Proposition \ref{prop:co}. Since $\Gamma$ is a cocycle, its image $\mathcal{CO}^{\rho}(\Gamma)$ is a natural transformation; see \cite[Section (1d)]{S08} for the definition of natural transformation.
	 In particular, the image of $\mathcal{CO}^{\rho}(\Gamma)$ under the projection $\hom^*_{\mathrm{End}(\WF(M))}(\mathrm{Id},\rho_\sharp) \to CW^*(L,\rho(L))$ is exactly $\mathcal{CO}^\rho_L(\Gamma)$ for the map  \eqref{eq:co_L}.

	The construction of \eqref{eq: genCO} is a slight modification of the construction of the usual closed-open map \cite[Section 5.4]{Ga}, which involves branch cuts glued by $\rho$.
	\begin{defn}
	\label{defn: variation functor}
	The {\em variation functor} is the endofunctor of $\WF(M)$ defined by
	\[
	\mathcal V := \mathrm{Cone}\left(\mathrm{Id} \xrightarrow{\mathcal {CO}^\rho(\Gamma)} \rho_\sharp \right): \WF(M) \to \WF(M).
	\]
	\end{defn}
	Here, we elaborate on the functor $\mathcal{V} = \mathrm{Cone}\left(\mathrm{Id} \xrightarrow{\mathcal {CO}^\rho(\Gamma)} \rho_\sharp \right)$.
	 First, the functor sends an object $L$ to its variation $\mathcal V(L)$. To describe $\mathcal V$ at the morphism level, consider the decomposition
		\begin{equation*}
			\label{eq:decomposition}
			\begin{split}
				CW^*\left(\mathcal{V}(L_0),\mathcal{V}(L_k)\right) & = CW^*\left(L_0[1], L_k[1]\right) \oplus CW^*\left(L_0[1], \rho(L_k)\right)\\ & \oplus CW^*\left(\rho(L_0),L_k[1]\right) \oplus CW^*\left(\rho(L_0), \rho(L_k)\right),
			\end{split}
		\end{equation*}
and write $\eta \in CW^*\left(\mathcal{V}(L_0),\mathcal{V}(L_k)\right)$ as a matrix 
\[\begin{pmatrix} \eta_{11} & \eta_{12} \\ \eta_{21} & \eta_{22} \end{pmatrix}\]
for some $\eta_{11} \in CW^*(L_0[1],L_k[1])$, $\eta_{12} \in CW^*(L_0[1],\rho(L_k))$, $\eta_{21} \in CW^*(\rho(L_0),L_k[1])$, and $\eta_{22}\in CW^*(\rho(L_0),\rho(L_k))$. Then, the components of the functor $\mathcal{V}$ at the morphism level are given by
\begin{align*}
\mathcal{V}^k &: CW^*(L_0, L_1)\otimes \cdots \otimes CW^*(L_{k-1}, X_k) \to CW^{*+1-k}(\mathcal {V}(L_0), \mathcal{V}(L_k)),\\
&\eta_1\otimes \cdots \otimes \eta_k \mapsto  \begin{pmatrix} \mathrm{Id}^k(\eta_1, \ldots, \eta_k) & \mathcal{CO}^{\rho} (\Gamma)^k(\eta_1, \ldots, \eta_k) \\ 0 & \mathrm{\rho_\sharp}^k (\eta_1, \ldots, \eta_k) \end{pmatrix} = 
\left\{
\begin{array}{cc}
\begin{pmatrix} \eta_1 & \mathcal{CO}^{\rho} (\Gamma)^1 (\eta_1) \\ 0 & \rho^1_\sharp (\eta_1) \end{pmatrix} & (k=1),\\
\begin{pmatrix} 0 & \mathcal{CO}^{\rho} (\Gamma)^k (\eta_1, \ldots, \eta_k) \\ 0 & 0 \end{pmatrix} & (k>1).
\end{array}
\right.
\end{align*}

We explain how the functor $\mathcal {V}$ recovers the classical variation \eqref{eq:var} by applying Grothendieck group $K_0$. One must be careful here because the naive assignment $K_0\left(\WF(M)\right) \xrightarrow{[L]_{K_0} \mapsto [L]} H_{n-1}(M, \partial M)$ is not well-defined. Instead, we apply \cite[Theorem 1.4]{Laz22} which states the following;  for Weinstein $X$, each relative homology class $\alpha \in H_{n-1}(X, \partial X)$ is represented by some $\mathcal L_\alpha \in \WF(X)$. Moreover, this assignment induces a well-defined, injective homomorphism denoted by $\mathcal L : H_{n-1}(X, \partial X) \xrightarrow{\alpha \mapsto \mathcal L_\alpha} K_0(\WF(X))$.

\begin{prop}
The following diagram commutes;
$$
\begin{tikzcd}
H_{n-1}(M, \partial M) \ar[r, "\mathrm{var}"] \ar[d, "\mathcal L"] & H_{n-1}(M) \ar[r, "q"] & H_{n-1}(M, \partial M) \ar[d, "\mathcal L"]\\
K_0\left(\WF(M) \right) \ar[rr, "{[\mathcal V]}"] && K_0\left(\WF(M) \right),
\end{tikzcd}
$$
where $q$ is the canonical map from the absolute homology to the relative one. 
\end{prop}
\begin{proof}
The proof is clear because $K_0(-)$ sends a twisted complex to the sum of its components.
\end{proof}

	The name "variation" will be further justified in the upcoming section.
	In particular, we show in Theorem \ref{thm: HF is finite} below that $HF^*(\mathcal V(L) , L')$ and $HF^*(L', \mathcal V(L))$ are cohomologically finite for all $L'$. 
	Therefore, it is natural to expect a compact representative $V_L$ for $\mathcal V(L)$ in a favorable situation.
	But first, let us write the following simple property of $\rho$.
	
	\begin{lemma}
	\label{lem:equalprimitive}
	Let $f_L$ be a primitive of $\lambda\vert_L$. Then $f_{\rho(L)} = f_L$ near $\partial M$. 
	\end{lemma}
	\begin{proof}
	Recall $f_{\rho(L)} = f_L\circ \rho^{-1}+F_{\rho^{-1}}$ where $F_{\rho^{-1}}$ as in \eqref{eq:F}. 
	$F_{\rho^{-1}}$ is constant near the boundary.
	We can further assume that $F_{\rho^{-1}}$ is constantly zero near the boundary.
	When $n>2$, observe that $M$ has only one connected boundary component so that we can add a suitable constant on $F_{\rho^{-1}}$ to make it zero.
	When $n=2$, one may use the fact that $\rho^{-1}$ is isotopic to the composition of Dehn twists through compactly supported isotopy and use an explicit form of $F_\tau$ for each of Dehn twists $\tau$, e.g., as in \cite{S08}. 
	\end{proof}
	With this property in mind, the construction of $V_L$ is presented in Section \ref{subsec:surgery}. 
	We provide an elementary proof of this fact for two-dimensional surfaces. 
	
	\begin{prop}
	\label{prop: variation dim 2}
	Suppose that $\dim M=2$ and $L\cap \rho(L)$ is empty except near the boundary. Then there is a compact representative $V_L$ for $\mathcal V(L)$. 
	\end{prop}
	\begin{proof}
	$L$ and $\rho(L)$ are topologically intervals with coordinate $x$. 
	One can draw $V_L$ in the following way; choose $L'$ as a small perturbation of $L$ inside its normal neighborhood given by a graph of $df$ such that $f'(x)$ is increasing from $-1$ to $1$.
	Orient $L'$ in an opposite direction to $L$. 
	Set $\rho(L)'=\rho(L')$ and orient it as $\rho(L)$ does. 
	One can check that  $L'$ meets $L$ only once in the interior and $L'$ and does not meet $\rho(L)$, and the same is true for $\rho(L)'$. In this situation, we choose the exact $V_L$ as a union of $L'$ and $\rho(L)'$ connected along small arcs connecting their respective boundaries.
	Then $V_L$ meets $L$ and $\rho(L)$ once and thrice. 
	We name them as 
	\[V_L \cap L= \{1_L\}, \hspace{10pt} V_L \cap \rho(L) =\{1_{\rho (L)} , p, q\}.\]
	
\begin{figure}[h]
\includegraphics[scale=1]{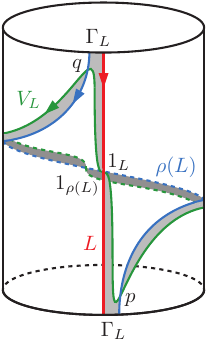}
\centering
\caption{Intersection points in cylinder case}
\label{fig:surg}
\end{figure}

	Here, $1_{\rho (L)}$ corresponds to the unique intersection of $\rho(L)'$ and $\rho(L)$.
	$p$ and $q$ are intersection points coming from the negative and positive boundaries of $\rho(L)'$.
	There are only four holomorphic polygons bounded by $L, \rho(L)$ and $V_L$ whose corners are either their intersections or $\Gamma_L$. See Figure \ref{fig:surg}.
	From this information, we can calculate the Floer complex as follows:
	\begin{align*}
	CW^*(V_L, \mathcal V(L)) & = \C \langle 1_L \rangle \oplus \C \langle 1_{\rho (L)},p,q\rangle,\\
	HW^*(V_L, \mathcal V(L)) & = \C \langle [1_{\rho (L)} + 1_L], [p+q]\rangle, \\
	\deg(1_L) = \deg(1_{\rho (L)}) =0, & \hspace{5pt} \deg(p)=\deg(q)=1,\\
	m_1(1_{\rho (L)}) = p-q, & \hspace{5pt} m_2(\Gamma_L, 1_L) = q-p.
	\end{align*}
	Here, the two strips including $1_{\rho(L)}$ in Figure \ref{fig:surg} contribute to $m_1$, and the other two triangles contribute to $m_2(\Gamma {,} -).$
	Similarly, by taking the dual of the above, we obtain
	\begin{align*}
	CW^*(\mathcal V(L), V_L) & = \C \langle 1_{\rho (L)}^\vee,p^\vee,q^\vee \rangle \oplus \C \langle 1_L^\vee \rangle,\\
	HW^*(\mathcal V(L), V_L) & = \C \langle [1_{\rho (L)}^\vee + 1_L^\vee], [p^\vee+q^\vee] \rangle, \\
	\deg(1_L^\vee) = \deg(1_{\rho (L)}^\vee) =1, & \hspace{5pt} \deg(p^\vee)=\deg(q^\vee)=0.
	\end{align*}
	Finally, the same polygons can be used to show that $[1_{\rho (L)}+1_L]: V_L \to \mathcal V(L)$ and $[p^\vee+q^\vee] : \mathcal V(L) \to V_L$ gives the isomorphism between them, i.e., their $m_2$-products are the identity maps. 
	\end{proof}
	
\subsection{Construction of a compact Lagrangian $V_L$}
\label{subsec:surgery}
Recall that an exact Lagrangian submanifold $L \subset M$ is said to have a cylindrical end if, near its boundary, it coincides with
$ \{\phi_Z^t(\partial L) \mid a \leq t \leq 0\}$ for some $-\infty < a < 0$ where $\partial L = L \pitchfork \partial M$. This implies that, for any primitive $f_L$ of $\lambda \mid_L$, $f_L$ is constant near $\partial L$.

Let $L$ be an exact Lagrangian submanifold of $M$ with a cylindrical end $\partial L  \subset \partial M$.
Consider $\phi_H^\epsilon(L)$, where $\epsilon >0$ is a small constant and $H$ is a Hamiltonian on $M$ satisfying $ H = r+ b$ for some $b\in \R$ near $\partial M$ so that $\phi_H^t$ coincides with the time $t$ flow of the Reeb vector field near $\partial M$.
For any $y \in \partial M$ and a Reeb flow $\gamma : \R \mapsto \partial M$ given by $\gamma(t) = \phi_R^{t\epsilon}(y)$ on $\partial M$, we have $\int_0^t \gamma^*\theta =t \epsilon$.
So it is natural to define a primitive $f_{\phi_H^\epsilon(L)}$ for $\lambda|_{\phi_H^\epsilon(L)}$ in such a way that
$$f_{\phi_H^\epsilon(L)}(r,y) = f_{L}(r, \phi_R^{-\epsilon}(y)) + \epsilon$$
for all $r$ near 1 and all $y \in \phi_H^\epsilon(\partial L)$.

\subsubsection{Flow surgery}
We review the notion of flow surgery \cite{MW18}, which is a Lagrangian surgery along a clean intersection.
Suppose that two Lagrangian $L_1$ and $L_2$ cleanly intersect along a submanifold $D$. 
Weinstein theorem tells us that there is a Weinstein neighborhood $L_1$ along $D$ such that 
\begin{itemize}
	\item $L_1$ is a zero-section of $T^*L_1$ and 
	\item $L_2$ is a conormal bundle $N^*_{D/L_1}$.
\end{itemize}
Next, choose a smooth curve $s\mapsto a(s)+ib(s) \subset \C$ whose image is the red curve in Figure \ref{fig:surgery} and define a function $\nu_\eta :[0,\infty) \to \R$ to be any solution to the equation
	\[\frac{d \nu_\eta}{dr} = a\circ b^{-1}(r).\]
Define the \textit{flow handle} of $D$ with respect to $\nu_\lambda$ by
	\[H_{\nu_\eta}^D = \left\{ \phi^1_{\nu_\eta (\Vert p \Vert)}(p) \in T^*L_1 : p\in N^*_{D/L_1}\setminus D, \Vert p \Vert <\epsilon \right\}.\]
\begin{prop}
\label{defn: flow surgery}
	If $\eta$ is smaller than the injective radius of $D$, then $\partial H_{\nu_\eta}^D$ divides $L_1$ into exactly two components. 
	In particular, one obtains a Lagrangian submanifold
	\[L_1 \sharp^{\nu_\eta}_DL_2 := \left(L_1\setminus (U\cap L_1) \right) \cup \left(L_2\setminus (U\cap L_2)\right) \cup H_{\nu_\eta}^D\]
	obtained by replacing a small Darboux neighborhood $U$ of $D$ with the flow handle. 
	We call $L_1 \sharp^{\nu_\eta}_DL_2$ a flow surgery of $L_1$ and $L_2$ along $D$.
\end{prop}

The isotopy class of the flow surgery does not depend on the particular choice of $\nu_\eta$ as long as $\epsilon_1$ is sufficiently small and $\eta$ is smaller than the injective radius of $D$. We will omit $\nu_\eta$ and simply write $L_1 \sharp_DL_2$. 

\subsubsection{Flow surgery along $\partial L$} 
Let $L_1$ and $L_2$ be exact Lagrangians of $M$ with a cylindrical end such that $\partial L_1 =\partial L_2  \subset \partial M$.
Choose coordinates $x_i$ for the Legendrian $\partial L \subset \partial M$, around which the contact form is presented as a standard form $\alpha = \sum_{i=1}^{n-1} y_i dx_i +dz.$
Then $\lambda = r\alpha$ takes the standard form of the Liouville form in a Weinstein neighborhood of $L_1$ near $\partial M$ with Darboux coordinates $x_1,\ldots x_{n-1}, r, ry_1, \ldots ry_{n-1}, z$.

Let $f: [0, 1] \to [0, 2\delta]$ be a strictly increasing smooth function satisfying 
\begin{itemize}
	\item $f(r) = \left \{ 
	\begin{array}{cc}
	0 & (0 \leq r \leq \epsilon)\\
	2\delta & (1-\epsilon \leq r \leq 1),
	\end{array}
	\right .$
	\item $f(1-r) = 2\delta-f(r)$ and $f'(1/2)=1/\sqrt 2$.
\end{itemize}

Take Hamiltonian perturbations $\tilde L_i$ of $L_i$ inside the neighborhood above in the following way.
\[ \widetilde L_1 = \left \{y_i =0, z=f(r)\right\}, \hspace {5pt} \widetilde L_2 = \left\{y_i=0, z= f(1-r)\right\}, \hspace{5pt} \widetilde L_1 \cap \widetilde L_2 = \left\{r=1/2, z= \delta \right\} \simeq \partial L.\] 

Since this takes the form of the local model for a clean intersection (after $\frac{\pi}{4}$-rotation of $r$-$z$ plane), we can apply the flow surgery along $\partial L$. 

We go back to the case we are interested in and prove our assertions under an additional assumption.
The additional algebraic condition below is expected to hold for a broad class of Lagrangians. See \cite{Gan21}.
\begin{thm}
\label{thm:compactrepresentative}
	Assume that 
	\begin{enumerate} 
		\item $\partial L$ is connected,
		\item $L\cap \rho(L)$ is empty away from the boundary, and
		\item $HW^*(L, L)$ or $HW^*\left(L, \rho(L)\right)$ is infinite dimensional. 
	\end{enumerate}
Then the compact component $\widetilde{L}\sharp_{\partial L}\widetilde{\rho(L)}$, denoted by $V_L$, is quasi-isomorphic to $\mathcal V(L)$ inside $\mathcal{WF}(M)$.
\end{thm}
\begin{proof}
This is proved by using a simple Lagrangian cobordism $K \subset M\times \C$ associated with $\widetilde L, \widetilde{\rho(L)}$ and $\widetilde L\sharp_{\partial L}\widetilde{\rho(L)}$. 
	We briefly recall its construction and refer to \cite{BC} for a more detailed discussion, whose extension to Liouville domains can be found in \cite{Bo22}. 
	First, Let $K'$ be a flow surgery of $\widetilde L \times \R$ and $\widetilde {\rho(L)} \times i\R$ along $\partial L \times \{0\}$. 
	Then $K'' := K' \cap \pi^{-1}\left( \{x\leq 0, y\geq 0\}\right )$ is a Lagrangian such that $K'' \vert_{x\ll 0, y=0} = \widetilde L$, $ K'' \vert_{x =0, y\gg0} = \widetilde{\rho(L)}$ and $K''\vert _{x=y=0} = \widetilde L\sharp_{\partial L}\widetilde{\rho(L)}$.
	We can complete $K''$ into a Lagrangian with cylindrical ends $K$ by extending $K''\vert _{x=y=0}$ in the radial direction. 
	See Figure \ref{fig:surgery}.
	
	By \cite{BC, Bo22}, the Lagrangian cobordism $K$ provides an exact triangle 
	\[
	\begin{tikzcd}
		\widetilde L \ar[r, "\mathcal F_K"] & \widetilde{\rho(L)} \ar[r] & \widetilde L \sharp_{\partial L} \widetilde{ \rho(L) }\ar[r] & L[1]
	\end{tikzcd}
	\]
	in $\mathcal{WF}(M)$. 
	The closed morphism $\mathcal F_K$ is determined in the following way: consider a product Lagrangian $\widetilde L\times l$ in $ M\times \C$, where $l$ is an exact Lagrangian curve in ($\C, \frac{1}{2}(xdy-ydx))$ depicted as the blue curve in Figure \ref{fig:surgery} below. 
	We obtain a closed morphism between Floer complexes and a distinguished element 
	\[\mathcal F(K) : CW^* (\widetilde L, \widetilde L) \to CW^*(\widetilde L, \widetilde{\rho(L)}), \hspace{5pt} \mathcal F_K :=\mathcal F(K)(1_{\widetilde L})\]
	by counting pseudo-holomorphic sections from $p$ to $q$ over the shaded gray region in Figure \ref{fig:surgery} below. Here, $1_{\widetilde L}$ is a cocycle for the unit of $HW^*(\widetilde L, \widetilde L)$.
	\begin{figure}[h]
	\def\svgwidth{15cm}
	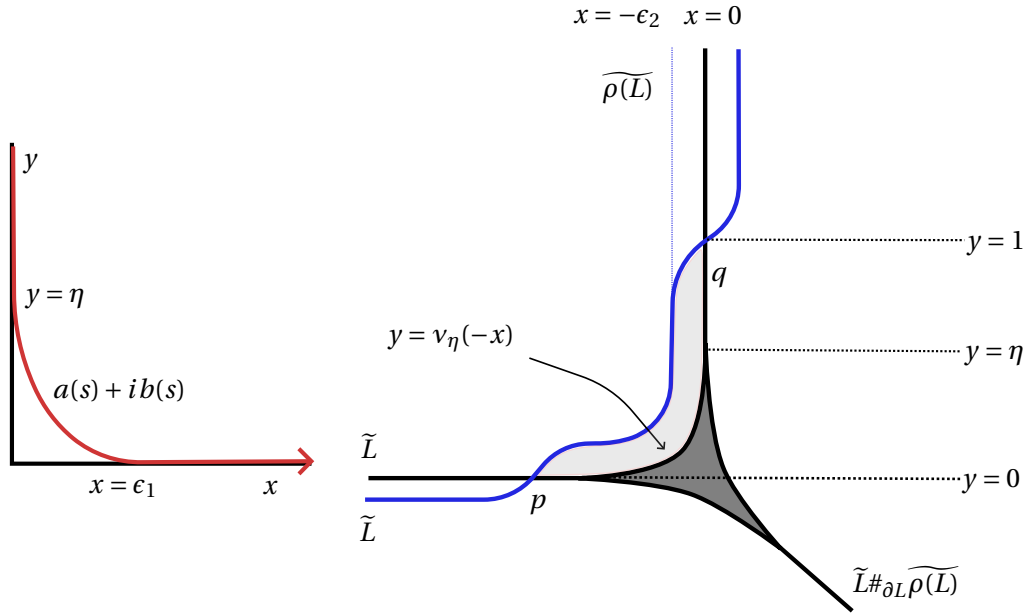
	\caption{(Left) Definition of $\nu_\eta$. (Right) Cobordism $K$ and a computation of $\mathcal F_K$.}
	\label{fig:surgery}
	\end{figure}
	
	The proposition follows from the following additional approximation of action values (see \cite{BCS18} for a related analysis).
	Suppose $\xi$ is a possible contribution of $\mathcal F_K$ which is given by an output of a perturbed pseudo-holomorphic section $u$.
	Denote the topological energy of $u$ by $E_{top}(u)$. It satisfies $0 \leq E_{top}(u) = \mathcal A(\xi) -\mathcal A(1_{\widetilde L})$. Therefore the following inequality holds. 
	\begin{align*}
		0\leq \int_0^1 \left(-\xi^*\lambda + H_{\widetilde L\times l, K}(\xi(t))dt \right) + \left (f_{\widetilde L \times l}(\xi(1))-f_K(\xi(0))\right) -\mathcal A(1_{\widetilde L}).
	\end{align*}
	Suppose that $\xi$ is a non-constant Hamiltonian chord. By the construction of $\widetilde L$ and $\widetilde{\rho(L)}$, it is contained in $\{r > r_0\}$ for some fixed $r_0>1$. We estimate each of the three terms in the inequality as follows: 
	
	\begin{enumerate}
		\item The first integral of the right-hand side is dominated by $-r_0^2$ since $H_{\widetilde L\times l, K} = H_{\widetilde L, \widetilde{\rho(L)}}$ is a $C^2$-small perturbation of $r^2$. 
		\item On $\{r > r_0\}$, the primitives of $\widetilde L$ and $\widetilde{\rho(L)}$ are constants given by $f_{\widetilde L} = f_L +\delta$ and $f_{\widetilde{\rho(L)}} = f_L$ by Lemma \ref{lem:equalprimitive}.
		We may assume that $f_{\widetilde L\times l} = f_{\widetilde L} + g_l$ where $g_l$ is a primitive of $l$ so that $f_K(p) = f_{\widetilde L}(p) = f_{\widetilde L \times l}(p)$. 
		The primitive of $l$ will start changing as $l$ becomes not radial. 
		As in Figure \ref{fig:surgery}, $l$ is drawn to be close to the boundary of the projection image of $K$ so that their $x$-coordinates differ by less than $\epsilon_2$.
		Notice that the boundary of the image of $K$ in the 2nd quadrant is precisely given by the graph of the function $\nu_\lambda(-r)$.
		By Stokes' theorem, we have 
			\[0 \geq g_l(q)-g_l(p) = \frac{1}{2}\int l^*(xdy-ydx) \geq -(\epsilon_1 \eta+\epsilon_2).\]
		Therefore, we have
			\[0\leq f_{\widetilde L \times l}(\xi(1))-f_{K}(\xi(0)) = f_{\widetilde L}(\xi(1)) - f_{\widetilde{\rho(L)}}(\xi(0)) - (g_l(q)-g_l(p)) \leq (\epsilon_1\eta +\epsilon_2 + \delta).\]
		\item The action value of $1_{\widetilde L}$ depends on the specific  Floer data $H_{\widetilde L, \widetilde L}$.
		We choose its $C^2$-norm to be bounded by $\epsilon_3$ inside a compact part of $M$ so that 
			\[0\leq \mathcal {A}(1_{\widetilde L}) \leq 2 \Vert H \Vert_{\infty, M^{cpt}}+ \Vert H \Vert_{1, M^{cpt}} \leq 3\epsilon_3.\]
	\end{enumerate}	
	From these approximations, we conclude that $0 \leq -r_0^2 + \epsilon_1\eta+\epsilon_2+3\epsilon_3+ \delta$ is positive for arbitrary $\epsilon_i$, $\delta$, and $\eta$, which is a contradiction. Therefore, we conclude that $\xi$ is an element from the intersection appearing in the compact part of $M$. 
	
	We can apply the Mose-Bott perturbation as before along the intersection $\widetilde L \cap \widetilde{\rho(L)} \simeq \partial L$ so that $\xi$ corresponds to the Morse critical point of $\partial L$. 
	Recall that $\widetilde L$, $\widetilde{\rho(L)}$ and $ \widetilde L \sharp_{\partial L} \widetilde{ \rho(L) }$ are all graded exact. 
	This condition forces the degree of $\mathcal F_K$ to be zero (\cite{MW18}). 
	Therefore, the only candidate for $\xi$ is the fundamental cycle of $\widetilde L \cap \widetilde{\rho(L)}$, which corresponds exactly to $\Gamma_L$. 
	As a result, we get $\mathcal F_K = c\cdot \Gamma_L$ for some $c$.
	
	The proof is completed if $c$ is not zero because the quasi-isomorphism class of $\mathrm{Cone} (c\cdot\Gamma_L)$ does not depend on $c$ while $c$ is nonzero. 
	For this, observe that $\widetilde L[1]\oplus \widetilde{\rho(L)}$ is not a proper object by assumption (2). 
	Meanwhile, $\widetilde L \sharp_{\partial L} \widetilde{ \rho(L) }$ has two components. 
	One is a compact connected $V_L$, and another is quasi-isomorphic to zero as it is contained in the collar neighborhood of $\partial M$. 
	Therefore, $\widetilde L \sharp_{\partial L} \widetilde{ \rho(L)}$ is a proper object and $c$ must be non-zero. 
\end{proof}

\section{Monodromy Lagrangian Floer cohomology $HF_\rho^*$}
\label{sec:HFrho}

In this section, we define the monodromy Lagrangian Floer cohomology $HF^*_\rho(L_1, L_2)$ between two Lagrangians $L_1$ and $L_2$ and describe its basic properties. 

\subsection{Seifert pairing}

	We recall another classical definition that motivates our constructions.
	\begin{defn}
	\label{defn:Seifert}
	Lef $f:(\C^n,0) \to (\C,0)$ be an isolated hypersurface singularity and let $M$ be the Milnor fiber associated with it. The {\em Seifert paring} of the isolated singularity $f$ is a bilinear pairing 
	\begin{align*}
		S:  H_{n-1}(M) \otimes H_{n-1}(M) \to \Z, \hspace{5pt} v_1\otimes v_2\to l\left(v_1,\rho_\epsilon(v_2)\right)
	\end{align*}
	where $\rho_\epsilon$ is a small parallel transport in the counter-clockwise direction and $l$ is a linking number of two disjoint $(n-1)$ cycles inside $S^{2n-1}$. 
	\end{defn}
	The Seifert pairing is non-degenerate and $\rho$-invariant but not symmetric in general due to $\rho_\epsilon$ in the formula. 
	We collect some basic properties of $S$ that are relevant to this work.
	\begin{thm} [Arnold, Gussein-Zade, Varchenko \cite{AGV2}]
	\label{thm:about S}
	For $v_1, v_2 \in H_{n-1}(M)$,
	\begin{enumerate}
		\item $S(v_1, v_2) = \mathrm{var}^{-1}(v_1) \bullet v_2$,
		\item $\mathrm{var}^{-1}(v_1)\bullet v_2  + v_1 \bullet \rho ( \mathrm{var}^{-1}(v_2))=0$, and
		\item $v_1 \bullet v_2 = - S(v_1, v_2) +(-1)^{n}S(v_2,v_1)$.
	\end{enumerate}
	\end{thm}

	The property (1) of Theorem \ref{thm:about S} provides a realization of $S$ using $H_{n-1}(M, \partial M)$ instead of $H_{n-1}(M)$. Let $l_i = \mathrm{var}^{-1}(v_i)$. Then we have
	\begin{align}
	\label{eqn: S by var}
	S(v_1,v_2) = l_1\bullet v_2 = l_1 \bullet \rho_*(l_2) -  l_1 \bullet l_2. 
	\end{align}
We aim to realize the equation \eqref{eqn: S by var} in the categorification.
	Next, we will find identities corresponding to the formula (2) and (3) of Theorem \ref{thm:about S}. 
	We upgrade each homology cycle $l_i$ to a Lagrangian $L_i$ and the intersection number $- \bullet - $ to the wrapped Floer cohomology $HW^*(-, -)$. 
	The most difficult part is realizing the subtraction in the equation, which we will explain in the next section.

\subsection{Definition of $HF_\rho^*$}
	We describe another action of $SH^*(\rho^{-1})$ on Lagrangian Floer cohomologies.
	
	Let $S$ be a decorated disc with
		\begin{itemize}
			\item one interior marked point $0 \in D^2$, equipped with positive cylindrical end $\eta^+$,
			\item one boundary marked point $1 \in D^2$ equipped with positive strip-like end $\epsilon^+$,
			\item one boundary marked point $-1\in \partial D^2$ equipped with negative strip-like end $\epsilon^-$, and
			\item a branch cut $(0,-i] \subset D^2$.
		\end{itemize}
	We also choose a domain-dependent Hamiltonian function denoted by $H_S$ which satisfies 
	\begin{align*}
		(\epsilon^+)^* H_S &= H_{\rho (L_1), L_2}, \\
		(\epsilon^-)^* H_S &= H_{L_1, L_2}, \\
		(\eta^+)^* H_S &= H.
	\end{align*}
	Consider the following moduli space of pseudo-holomorphic curves 
		\begin{align*}
		 \mathcal M_{\cap} (\gamma_+ , \xi_+, \xi_-; H_S, J_S):=
			\left\{u: S \to M  \;\middle|\;
				\begin{array}{ll}
				\rho\circ u(e^{s+\frac{3 \pi i}{2}}) = u(e^{s-\frac{\pi i}{2}}),& \left( du-X_{H_S}(u)\otimes \alpha \right)^{0,1}_{J_S}=0,\\
				\lim_{s\to \pm\infty}(\epsilon^\pm)^* u = \xi_\pm(t), &\lim_{s\to \infty} (\eta^+)^* u = \gamma_+(t), \\
				u(e^{\pi it}) \in L_2, \hspace{5pt} t \in (0,1), &u(e^{\pi i(1+t)}) \in\mathcal L_{1,t}, \hspace{5pt} t\in (0, 1)\\
				\end{array}
			\right\}.
		\end{align*}
Here, $\mathcal L_{1,t}$ is $L_1$ for $t\in (0, 1/2)$ and $\rho (L_1)$ for $t\in (1/2, 1)$ with branch cut at $t=1/2$; see \eqref{defn:moving bd}.
	
	\begin{defn}
		\label{defn: twisted cap action}
	Let $\gamma_+ \in CF^*(\rho^{-1}; H, J)$ be a cochain. Its twisted quantum cap action 
		\[\cap \gamma_+: CW^*\left(\rho(L_1), L_2; H_{\rho(L_1), L_2}, J_{\rho(L_1), L_2} \right) \to CW^* \left(L_1, L_2; H_{L_1, L_2}, J_{L_1, L_2} \right)\]
	is defined by
	\[\xi_+ \cap \gamma_+: = \sum_{\xi_-\in \mathcal P(L_1, L_2; H_{L_1, L_2})} \# \mathcal M_{\cap} (\gamma_+ , \xi_+, \xi_-; H_S, J_S) \cdot \xi_-.\]
	\end{defn}
	If $\gamma_+$ is a cocycle, then the standard argument shows that $\cap \gamma_+$ is a cochain map of $\deg \gamma_+$.
	We apply this fact to the cocycle $\Gamma$.  
	\begin{defn}
	Define $CF^*_\rho(L_1, L_2)$ as the total complex of
	\[CF^*_\rho(L_1, L_2) := \left( \begin{tikzcd} CW^*\left(\rho(L_1), L_2, J_{\rho(L_1), L_2}, H_{\rho(L_1), L_2}\right)[1] \arrow[r, "\cap \Gamma"] & CW^* \left(L_1, L_2, J_{L_1, L_2}, H_{L_1, L_2}\right) \end{tikzcd}\right).\]
	We denote its cohomology by $HF_\rho^*(L_1, L_2)$ and call it monodromy Lagrangian Floer cohomology. 
	\end{defn}
	
\subsection{Invariance}
	We verify that the cohomology group $HF_\rho^*$ has certain Hamiltonian invariance properties.
	The first one ensures that it is independent of Hamiltonian perturbation data. 
	\begin{prop}
	\label{prop: Hamiltonian invariance}
		 The cohomology group $HF_\rho^*$ is independent of the choice of Floer data.
	\end{prop}
	\begin{proof}
	The statement amounts to showing that the following diagram commutes up to homotopy
	\[
	\begin{tikzcd}[column sep=large]
	CW^*\left(\rho(L_1), L_2; H^1_{\rho(L_1), L_2}\right) \arrow[r, "\cap \Gamma"] \arrow[d, "c_{\rho(L_1),L_2}"] & CW^* \left(L_1, L_2; H^1_{L_1, L_2}\right) \arrow[d, "c_{L_1,L_2}"]\\
	CW^*\left(\rho(L_1), L_2; H^2_{\rho(L_1), L_2}\right) \arrow[r, "\cap c_{1,2}(\Gamma)"] & CW^* \left(L_1, L_2; H^2_{L_1, L_2}.\right)
	\end{tikzcd}
	\]
	Here, $c_{1,2}$ denotes the standard continuation map for Hamiltonian Floer complexes.
	$c_{L_1, L_2}$ and $c_{\rho(L_1), L_2}$ denote corresponding continuation maps of Lagrangian Floer complexes which are quasi-isomorphisms.
	
	We first build a family of complex structures on $\R \times [0,1]$ by constructing a diffeomorphism $\Psi_{R\in[1,\infty)}: \R \times [0,1] \to \R \times [0,1]$ as in Subsection \ref{subsubsec: First homotopy} so that its induced complex structure degenerates into the desired configuration depicted in the Figure \ref{fig:degen4} below.  

\begin{figure}[h]
\begin{subfigure}[t]{0.48\textwidth}
\raisebox{5ex}{\includegraphics[scale=0.95]{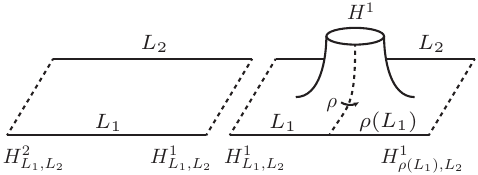}}
\centering
\caption{$R \to 1$ degeneration}
\end{subfigure}
\begin{subfigure}[t]{0.48\textwidth}
\includegraphics[scale=0.95]{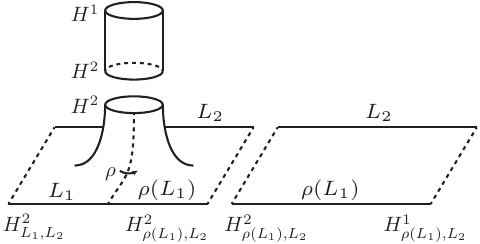}
\centering
\caption{$R \to \infty$ degeneration}
\end{subfigure}
\centering
\caption{Family of complex structures}
\label{fig:degen4}
\end{figure}
	
In detail, choose disjoint open subsets inside $(-\infty, -1)\times[0,1]$;
	\[U_1 = \left(-\infty, -\frac{2R}{R-1}-3\right) \times [0,1], \; V_1=\left(-\frac{2R}{R-1}-2, -\frac{R}{R-1}-2\right)\times[0,1], \; W_1=\left(-2, -1\right)\times [0,1].\]
	Then, $\Psi_R$ is given by 
	\begin{itemize}
		\item $\Psi_R: U_1 \to (-\infty, -5) \times [0,1]$ is a translation; $(s,t) \mapsto \left (s+ \frac{2}{R-1}, t \right)$,
		\item $\Psi_R: V_1 \to (-4, -3)\times[0,1]$ is a $\frac{R}{R-1}$ scale on the first factor $(s, t) \mapsto \left( \frac{R}{R-1}(s+2)-2 , t\right)$,
		\item $\Psi_R: W_1 \to (-2, -1)$ is the identity,
		\item $\Psi_R \vert _{(-\infty, -1)\times[0,1] \setminus (U_1\cup V_1 \cup W_1)}$ is a scale on the first factor interpolating $\Psi_R\vert_{U_1}$, $\Psi_R\vert_{V_1}$, and $\Psi_R\vert_{W_1}$, and
		\item $\Psi_R$ converges to the identity on $(-\infty, -1)\times [0,1]$ as $R \to \infty$. 
	\end{itemize}
	Similary, choose disjoint open subsets $U_2 = (2R+3, \infty) \times [0,1], V_2= (2+R, 2+2R)\times [0,1]$, and $W_2=(1 , 2)\times [0,1]$. On these open sets, 
	\begin{itemize}
		\item $\Psi_R: U_2 \to (5, \infty) \times [0,1]$ is a translation; $(s,t) \mapsto \left (s+2-2R, t \right)$,
		\item $\Psi_R: V_2 \to (3, 4)\times[0,1]$ is an $R^{-1}$-scale on the first factor $(s, t) \mapsto \left(R^{-1}(s-2)+2 , t\right)$,
		\item $\Psi_R: W_2 \to (1, 2) \times [0, 1]$ is the identity,
		\item $\Psi_R \vert _{(1, \infty)\times[0,1] \setminus (U_2\cup V_2 \cup W_2)}$ is a scale on the first factor interpolating $\Psi_R\vert_{U_2}$, $\Psi_R\vert_{V_2}$, and $\Psi_R\vert_{W_2}$, and
		\item $\Psi_R$ converges to the identity on $(1, \infty)\times [0,1]$ as $R \to 1$. 
	\end{itemize}
	Third, we introduce a radial coordinate $e^{-(r+i\theta)}$ for the disc neighborhood $D^*:=\left\{ \left\vert (s,t-\frac{1}{2})\right\vert <\frac{1}{3}\right\} \subset \R\times [0,1]$.
	Choose disjoint open subsets of $D^*$ (with respect to radial coordinate)
	\[U_3 = \left(\log(2R+5), \infty\right) \times [0,2\pi), \; V_3= \left(\log(R+4), \log(2R+4)\right) \times [0, 2\pi), \; W_3 = (\log3, \log4)\times[0,2\pi)\]
	and define $\Psi_R$ with respect to the radial coordinate by
	\begin{itemize}
		\item $\Psi_R: U_3 \to (\log 7, \infty) \times [0,2\pi)$ is a translation; $(r,\theta) \mapsto \left (r-\log(2R+5)+\log 7, \theta \right)$,
		\item $\Psi_R: V_3 \to (\log 5, \log 6)\times[0,2\pi)$ is given by $(r, \theta) \mapsto \left(\log (R^{-1}(e^r -4)+4) , \theta \right)$,
		\item $\Psi_R: W_3 \to (\log3, \log4)\times[0, 2\pi)$ is the identity,
		\item $\Psi_R \vert _{(\log3, \infty)\times[0,2\pi) \setminus (U_3\cup V_3 \cup W_3)}$ interpolates between $\Psi_R\vert_{U_3}$, $\Psi_R\vert_{V_3}$, and $\Psi_R\vert_{W_3}$, and
		\item $\Psi_R$ converges to the identity on $(\log 3, \infty)\times [0,2\pi)$ as $R \to 1$.
	\end{itemize}
	Finally, set $\Psi_R$ to be the identity for $\forall R$ on $[-1, 1]\times[0,1] \setminus D^*$.
	Let  $j_R$ be a complex structure on a punctured strip induced by $\Psi_R$.
	Observe that $j_R$ is designed to degenerate along $\Psi_R(V_1)$ into a union of punctured strip and the other strip when $R\to 1$, while it breaks simultaneously at $\Psi_R(V_2)$ and $\Psi_R(V_3)$ into the union of a punctured strip, a cylinder and the other strip as $R \to \infty$. 
	We put 
	\begin{itemize}
		\item a negative strip-like end $\epsilon^-$ around $s=-\infty$ so that $\epsilon^-( (-\infty, 0) \times [0,1]) \subset U_1$, 
		\item a positive strip-like end $\epsilon^+$ around $s=\infty$ so that $\epsilon^+( (0, \infty) \times [0,1]) \subset U_2$,
		\item a positive cylindrical end $\eta^+$ around $(0,1/2)$ so that $\eta^+( (0, \infty) \times S^1) \subset U_3$, 
	\end{itemize}
and a domain-dependent Hamiltonian $H_{S, R}$ so that 
	\begin{align*}
	H_{S,R} &= \left\{
		\begin{array}{ll}
		H^1_{\rho(L_1), L_2} & (7 < s),\\
		H^2_{\rho(L_1), L_2} &  (3 < s < 6, \hspace{5pt}1 \leq r \ll 2),\\
		H^1_{\rho(L_1), L_2} &  (3 < s < 6,  \hspace{5pt} 2 \ll r),\\
		H^2_{L_1, L_2} & (-6 < s <-3,  \hspace{5pt} 1\leq R \ll 2),\\
		H^1_{L_1, L_2} & (-6 <s< -4,  \hspace{5pt} 2 \ll R),\\
		H^2_{L_1, L_2} & (s < -7), \\
		H^1 & (\log 8 < r), \\
		H^1 & (\log 4< r < \log 7,  \hspace{5pt} 1\leq R \ll 2),\\
		H^2 & (\log 4 < r < \log 7,  \hspace{5pt} 2 \ll R).
		\end{array}
	\right.
	\end{align*}
	As $R \to 1$, $H_{S, R}$ becomes a Hamiltonian perturbation for Floer continuation map from $H^1_{L_1, L_2}$ to $H^2_{L_1, L_2}$ on the strip break off in the course of the degeneration. 
	On the other extreme, it is designed to become a perturbation for the Floer continuation map from $H^1_{\rho(L_1), L_2}$ to $H^2_{\rho(L_1), L_2}$ on the strip component of the $R \to \infty$ degeneration.  
	Also, it becomes a perturbation for the continuation map from $H^1$ to $H^2$ on a cylinder component of the same degeneration.
	
	In conclusion, the moduli space of perturbed holomorphic curves $u: (S, j_R) \to M$ provides a homotopy between  $c_{L_1, L_2} \circ\cap \Gamma$ and $ \left( \cap c (\Gamma) \right) \circ c_{\rho(L_1), L_2}$.
	\end{proof}

	The second invariance property ensures that it is well-defined for the Hamiltonian isotopy class of $\rho$. 
	\begin{prop}
	\label{prop: isotopy invariance}
	Let $\rho_t$ be a compactly supported Hamiltonian isotopy between $\rho=\rho_0$ and $\tau = \rho_1$. 
	Then the following diagram is homotopy commutative;
		\[
	\begin{tikzcd}[column sep=large]
	CW^*\left(\rho(L_1), L_2; H^1_{\rho(L_1), L_2}\right) \arrow[r, "\cap \Gamma"] \arrow[d, "c_{\rho_t(L_1), L_2}"] & CW^* \left(L_1, L_2;  H^1_{L_1, L_2}\right) \arrow[d, "c_{L_1, L_2}"]\\
	CW^*\left(\tau(L_1), L_2; H^2_{\tau L_1, L_2}\right) \arrow[r, "\cap c_{\rho_t}(\Gamma)"] & CW^* \left(L_1, L_2; H^2_{L_1, L_2} \right)
	\end{tikzcd}.
	\]
	Here, the vertical homomorphisms are continuation quasi-isomorphisms of Lagrangian Floer complexes, and $c_{\rho_t}$ is a continuation quasi-isomorphism of Hamiltonian Floer complexes in Proposition \ref{prop:invariance1}. 
	In particular, we have a canonical isomorphism of cohomologies
		\[HF^*_{\rho}(L_1, L_2) \simeq HF^*_{\tau}(L_1, L_2).\]
	\end{prop}
	\begin{proof}
	The construction is almost the same as in Proposition \ref{prop: Hamiltonian invariance}, except that we further use the identification $\rho_{l(t,R)}$ along the branch cut $\{0\}\times [0, 1/2] \in \R\times [0,1]$ for each $R$:
	\[ p(r):= \left\{ 
		\begin{array}{ll}
		0 & (1\leq r \leq 2),\\
		\textrm{increasing} & (2\leq r \leq 3),\\
		1 & (3\leq r),
		\end{array}
	\right. \hspace{5pt}
	l(t,R) := \left\{ 
		\begin{array}{ll}
		0 & (1/3\ll t \leq1/2),\\
		\textrm{increasing} & (1/6<t<1/3),\\
		p(R) & (0\leq t \ll1/6).
		\end{array}
	\right.
	\]
	In particular, $\rho_{l(t,R)}$ is constantly $\rho_0=\rho$ when $R<2$, while it interpolates $\rho$ and $\tau$ when $R>3$;  see Figure \ref{fig:degen5}. 
	We leave details to the reader. 
\begin{figure}[h]
\begin{subfigure}[t]{0.48\textwidth}
\raisebox{5ex}{\includegraphics[scale=0.95]{degen41}}
\centering
\caption{$R \to 1$ degeneration}
\end{subfigure}
\begin{subfigure}[t]{0.48\textwidth}
\includegraphics[scale=0.95]{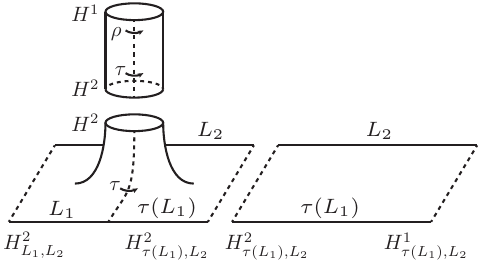}
\centering
\caption{$R \to \infty$ degeneration}
\end{subfigure}
\centering
\caption{ }
\label{fig:degen5}
\end{figure}
	\end{proof}
\begin{remark}
Using the same argument, one can further prove the same commutative diagram when $\rho_t$ is an isotopy given by a good perturbation. 
\end{remark}
\subsection{Proof of the main theorems} \label{subsec:main}
	We provide proofs of Theorem \ref{thm:1}, \ref{thm:2} and \ref{thm:var}.
	Let us start with an easy relation between cohomologies $HF^*_\rho$ and the variation operator $\mathcal V$.
	\begin{prop}\label{prop:HF and V}
 	We have isomorphisms
	\[HF^*_\rho(L_1, L_2) \cong HW^* \left( \mathcal{V} (L_1)[-1],  L_2 \right) \cong HW^* \left(\rho(L_1), \mathcal V(L_2)\right).\]
	\end{prop}
	\begin{proof}
	We construct a homotopy equivalence $K$ between $CF^*_\rho(L_1, L_2)$ and the following twisted complex of cochain complexes
	\[\begin{tikzcd}[column sep=huge] CW^* (\rho(L_1), L_2)[1]\arrow[r, "m_2(\mathcal{CO}_{L_{1}}^\rho(\Gamma){,}-)"] & CW^*(L_1, L_2).\end{tikzcd}\]
	Let $\{ S_r \}_{r\in [0,1)}$ be a 1-parameter family of decorated discs with
	\begin{itemize}
		\item one interior marked point $-ri \in D^2$, equipped with positive cylindrical end $\eta^+$,
		\item one boundary marked point $1 \in \partial D^2$ equipped with positive strip-like end $\epsilon^+$,
		\item one boundary marked point $-1\in \partial D^2$ equipped with negative strip-like end $\epsilon^-$, and
		\item a branch cut $(-ri,-i] \subset D^2$.
	\end{itemize}
	We also choose a domain-dependent almost complex structure $J_{S_r}$,  Hamiltonian function $H_{S_r}$, and sub-closed one form $\alpha$ as before.
	Consider 
	\begin{align*}
	\mathcal M_K (\Gamma , \xi_-, \xi_+; H_{S_r}, J_{S_r}):=
		\left\{u: S_r  \to M  \;\middle|\;		\begin{array}{ll}
			\left( du -X_{H_{S_r}}(u)\otimes \alpha \right)^{0,1}_{J_{S_r}}=0,\\
			\rho\circ u(e^{s+\frac{3\pi i}{2}}) = u(e^{s-\frac{\pi i}{2}}), & \log r <s <0,\\
			\lim_{s\to \pm\infty}(\epsilon^\pm)^* u = \xi_\pm, & \lim_{s\to \infty} (\eta^+)^* u = \Gamma, \\
			u(e^{\pi it}) \in\mathcal L_{1,t}, t\in (0, 1), &
			u(e^{\pi it}) \in L_2, t\in (1, 2)
			\end{array} 
		\right\}
		\end{align*}
	with a Lagrangian boundary condition as before. 
	As $r \to 1$, the domain $S_r$ degenerates into two discs which provide an $m_2\left(\mathcal {CO}_{L_1}^\rho(\Gamma), -\right)$.
	Therefore, the counting of regid elements of this moduli space defines an operator $K$, which satisfies 
	\[d\circ K+ K\circ d = (-\cap \Gamma) - m_2(\mathcal{CO}_{L_1}(\Gamma), -).\]
	This proves the first isomorphism. The second isomorphism can be proven similarly after we move the branch cut to the opposite side of the disc. 
	\end{proof}
	
	The next proposition establishes a relation between $HF^*_\rho$ and the Seifert paring.
	\begin{thm}
	\label{thm: HF is finite}
		$HF_\rho^*(L_1, L_2)$ is a finite-dimensional graded vector space whose Euler characteristic is 
		\[\chi\left( HF^*_\rho(L_1, L_2) \right) = (-1)^nS\left(v_2, v_1\right), \]
		where $S$ is the Seifert pairing (Definition \ref{defn:Seifert}) and $v_i $ are the homology classes of $\mathcal V(L_i)$.
	\end{thm}
	\begin{proof}
	We use $\check \rho$ in the proof so that the action of $\Gamma$ is clear. 
	Consider the following diagram of short exact sequences
	\[
	\begin{tikzcd}
	 0  \ar[r] & CW^*_{\geq 0}\left(\check \rho(L_1), L_2\right) \ar[d, "\cap \Gamma"] \ar[r] &CW^* \left(\check \rho(L_1), L_2\right) \ar[d, "\cap\Gamma"]\ar[r] & CW^*_{<0}\left(\check \rho(L_1), L_2\right) \ar[d, "\cap \Gamma"] \ar[r] & 0 \\
	 0  \ar[r] & CW^*_{\geq0}\left( L_1, L_2\right)  \ar[r] &CW^* \left( L_1, L_2\right) \ar[r] & CW^*_{<0}\left(L_1, L_2\right) \ar[r] & 0.
	\end{tikzcd}
	\]
	The horizontal short exact sequences are the decompositions of Floer complexes by the energy filtration; $CW^*_{\geq 0}$ denotes a subcomplex generated by generators whose action value is nonnegative, i.e., interior intersection points. 
	The vertical maps are the twisted quantum cap action $\cap \Gamma$ and its induced maps. 
	They make sense because the action value of $\Gamma$ can be made arbitrarily small. 

	{\em Claim: The right vertical map is an isomorphism.} 
	Note that the claim is true for an obvious reason when one of the $L_i$ is closed. 
	We may assume that both $L_i$ have cylindrical ends. 
	The claim follows if one shows that the right vertical map has the identity diagonal because the map is always lower-triangular with respect to the energy filtration.
	To be more specific about what it means, note that $CW^*_{<0}(\check \rho(L_1), L_2)$ is canonically identified with $CW^*_{<0}(\phi_{b}^1 (L_1), L_2)$, where $b$ is the Hamiltonian of slope $-\epsilon$ that appeared in Proposition \ref{prop:perturbation}.
	We see that for each $\xi_+ \in CW^*(\check \rho(L_1), L_2)$ that is not an interior intersection, there is a canonical element $\xi_-= \phi_{-b}^1 \circ \xi \in CW^*(L_1, L_2)$.
	We claim that the following moduli space of curves
	\[ \mathcal M_K(\Gamma, \xi_-, \xi_+; U) := \left \{ u\in \mathcal M_K(\Gamma, \xi_-, \xi_+) | u(S)\subset U \right\}\] 
	contains a single element for sufficiently small $\epsilon$. 
	It means that the induced map $\cap \Gamma$ on the associated graded module of energy filtration is the identity. 
		
	The rest of the claim is similar to \cite[Theorem 6.8]{Rit13} and Theorem \ref{thm: PSS CO diagram}. 
	By Theorem \ref{thm: PSS CO diagram},  the associated graded map of $\cap \Gamma$ is equal to $\cap \Phi([U])$. 
	There is a homotopy between $\mathcal M_K(\Gamma, \xi_-, \xi_+; U)$ and the moduli space of pairs consisting of $(u, l)$, where $u$ is a pseudo-holomorphic strip used to define the continuation map and $l$ is a Morse flow emanating from the fundamental class $[U]$ attached to the center of the strip; see Figure \ref{fig:QC}.
	The second condition is vacuous.
	There is one and only one such flow for generic points because the fundamental class is represented by the minimum of Morse function $f$ on $U$.
	Meanwhile, there is a unique Floer continuation strip from $\xi_+$ to $\xi_-$ for small $\epsilon$, which corresponds to the constant strip for $\epsilon=0$. 
	
\begin{figure}[h]
\includegraphics[scale=1]{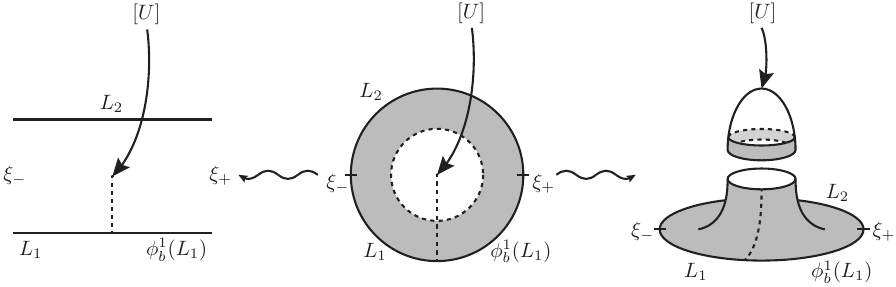}
\centering
\caption{ }
\label{fig:QC}
\end{figure}
	
	The claim implies that the right vertical map is homotopic to a lower-triangular matrix whose diagonal is the identity.
	Therefore, the right vertical map is null-homotopic and $HF^*_\rho(L_1, L_2)$ is isomorphic to the cohomology of the left vertical arrow in the diagram which is of finite dimension.

	Without loss of generality, one may assume $L_1$ and $L_2$ do not intersect along $\partial M$. Then the Euler number of $CW^*_{\geq0}(L, K)$ provides a signed intersection number $[L]\bullet[K]$ between geometric cycles. By \eqref{eqn: S by var}, 
	\begin{align*} 
	\chi\left( HF^*_\rho(L_1, L_2) \right) & = [L_1]\bullet [L_2]  - [\rho(L_1)] \bullet [L_2]  = - \left(v_1\bullet \mathrm{var}^{-1}(v_2) \right)=(-1)^nS\left(v_2, v_1\right).
	\end{align*}
	\end{proof}
	
	\begin{cor}
	\label{cor:VLisproper}
	The variation $\mathcal V(L)$ is a proper object for any Lagrangian $L$.
	\end{cor}
	\begin{proof}
	The assertion follows from Proposition \ref{prop:HF and V} combined with Theorem \ref{thm: HF is finite}. 
	\end{proof}

	\begin{prop}
	\label{prop:Serre dual}
	The following hold.
	\begin{enumerate}
		\item There is a canonical duality $HF^*_\rho(L_1, L_2) \simeq HF^{n-2-*}_\rho\left(L_2, \rho(L_1)\right)^\vee$.
		\item There is a long exact sequence of cohomologies
		\[
		\begin{tikzcd}[column sep=small]
		\cdots \arrow[r] & HF_\rho^{n-1-*}(L_2,L_1)^\vee \arrow[r] & 
		HW^*\left(\mathcal V(L_1), \mathcal V(L_2)\right) \arrow[r] & HF_\rho^*(L_1,L_2) \arrow[r] &
		HF_\rho^{n-2-*}(L_2,L_1)^\vee \arrow[r] & \cdots.
		\end{tikzcd}
		\]
	\end{enumerate}
	\end{prop}
	\begin{proof}
	$\WF(M)$ carries a smooth $(n-1)-$Calabi-Yau structure \cite{Ga}, which induces a proper Calabi-Yau structure on its subcategory generated by proper objects \cite[Theorem 3.1]{BD19}. Since $\mathcal V (L)$ is proper by Corollary \ref{cor:VLisproper}, we get a functorial trace $CF^{n-1}(\mathcal V(L_2), \mathcal V(L_2)) \xrightarrow{\mathrm{tr}} \C$ and a pairing
	\[CF^*(\rho(L_1), \mathcal V(L_2))\otimes CF^{n-1-*}(\mathcal V(L_2), \rho(L_1)) \xrightarrow{\mu_2} CF^{n-1}(\mathcal V(L_2), \mathcal V(L_2)) \xrightarrow{\mathrm{tr}} \C.\]
	The pairing is nondegenerate due to \cite[Equation (2.13)]{BD19} combined with the given smooth Calabi-Yau structure on $\WF(M)$.
	Then the first claim follows from the chain of isomorphisms:
	\begin{align*}
	HF^{*}_\rho(L_1, L_2) &\simeq HW^{*}(\rho(L_1), \mathcal V(L_2))\\
	&\simeq HW^{n-1-*}(\mathcal V(L_2), \rho(L_1) )^\vee\\
	&\simeq HW^{n-2-*}(\mathcal V(L_2)[-1], \rho(L_1) )^\vee\\
	&\simeq HF^{n-2-*}_\rho(L_2, \rho(L_1))^\vee.
	\end{align*}
	The second statement follows from the long exact sequence
	\[
	\begin{tikzcd}[column sep=small]
	\cdots \ar[r] & HW^{*-1}(L_1, \mathcal V(L_2)) \ar[r] & HW^*( \mathcal V(L_1), \mathcal V(L_2)) \ar[r] & HW^* (\rho(L_1), \mathcal V(L_2)) \ar[r] & HW^*(L_1, \mathcal V(L_2)) \ar[r] & \cdots
	\end{tikzcd}
	\]
	induced by the canonical short exact sequence:
	\[
	\begin{tikzcd}
	0 \ar[r] & CW^*(L_1[1], \mathcal V(L_2)) \ar[r, "\mathrm{inc.}"] & CW^*( \mathcal V(L_1), \mathcal V(L_2)) \ar[r, "\mathrm{proj.}"] & CW^* (\rho(L_1), \mathcal V(L_2)) \ar[r] &0.
	\end{tikzcd}
	\]
	Then the second statement follows from the first statement combined with Proposition \ref{prop:HF and V}. 
	\end{proof}

 \section{Adapted family and exceptional collections}
 \label{sec:distinguished}
 
In this section, we deal with the categorification of another property of the Seifert form $S$: it admits a lower-triangular presentation. 

Recall that given a Lefschetz fibration over $\C$, one can pick a base point on $\C$, choose vanishing paths to the critical values in $\C$, and find an ordered family of vanishing cycles, say $(V_1,\cdots, V_\mu)$ in the Milnor fiber.
Such a choice is called a {\em distinguished collection} of vanishing cycles which depends on the choice of vanishing paths.
Then the Seifert form $S$ has the following presentation:
	\begin{align}
	S(V_i, V_j) = \left\{ 
		\begin{array}{ll} 
		0 & (i<j), \\
		1& (i=j).
		\end{array} \right.
	\end{align}
A natural counterpart of this property is {\em directedness}.
Indeed, Fukaya-Seidel category of a Lefschetz fibration is built upon this property. 
It can be described as a directed $\AI$-category whose objects are $V_1, \ldots V_\mu$ and morphism spaces between them are defined by
	\begin{align}
	\hom^*_{FS}(V_i, V_j) = \left\{ 
		\begin{array}{ll} 
		0 & (i>j), \\ 
		\mathbb K \cdot e_{V_i}[0]& (i=j),\\
		CF^*(V_i, V_j) & (i<j).
		\end{array} \right.
	\end{align}
Observe that directedness is imposed as a definition. Namely, the Poincar\'e duality $HF^{*}(V_i,V_j) \cong HF^{n-1-*}(V_j,V_i)^\vee$ has been broken on purpose.
Seidel showed that its derived category does not depend on the choice of distinguished collection.

We will see that the story for $HF^*_\rho$ goes the other way around. 
It is a property of a collection of Lagrangians on which $HF_\rho^*$ vanishes in one direction.

\begin{defn}[Exceptional collection] \label{defn:dist}
An ordered family of Lagrangians $(L_1,\cdots, L_m)$ is said to be an {\em exceptional collection} with respect to $HF^*_\rho$ if
\begin{enumerate}
\item\label{item:dist1} $ HF^*_\rho(L_i, L_j) \cong 0$ for any $ i>j $,  and
\item\label{item:dist2} $ HF^*_\rho(L_i,L_i) \cong \mathbb{K} \cdot e_i [0]$ for any $i$.
\end{enumerate}
\end{defn}

The following calculations justify the name "exceptional".
\begin{prop}
\label{prop:distinguished vanishing}
Suppose that $(L_1,\cdots, L_k)$ is an exceptional collection with respect to $HF^*_\rho$.
For $i<j$, we have
\begin{enumerate}
\item  $HW^*\left(\mathcal V(L_j), L_i\right) = HW^*\left(L_i,\mathcal V(L_j)\right) = 0,$
\item  $HW^*\left(\rho(L_j), \mathcal V(L_i)\right) = HW^*\left(\mathcal V(L_i), \rho(L_j)\right) =0,$ and
\item  $HF_\rho^*(L_i, \rho(L_j)) =0$.
\end{enumerate}
\end{prop}
\begin{proof}
The first and second follow from Proposition \ref{prop:HF and V}.
The third one follows from Proposition \ref{prop:Serre dual}.
\end{proof}

\begin{cor}
\label{cor:HF and FS}
Suppose that $(L_1,\cdots, L_k)$ is an exceptional collection with respect to $HF^*_\rho$. Then
	\[HW^*(\mathcal{V}(L_i), \mathcal{V}(L_j)) \cong \left\{
	\begin{array}{ll}
	\mathbb K[0] \oplus \mathbb K [n-1] &(i=j),\\
	HF_\rho^* (L_i, L_j) & (i<j),\\
	HF_\rho^{n-1-*} (L_j, L_i)^\vee & (i>j) 
	\end{array}
	\right.
	, \hspace{5pt}
	HF_\rho^*(L_i, L_j) \cong \left\{
	\begin{array}{ll}
	0 & (i>j),\\
	\mathbb K[0] & (i=j),\\
	HW^*(\mathcal{V}(L_i), \mathcal{V}(L_j))& (i<j).\\
	\end{array}
	\right.
	 \]
\end{cor}

\begin{proof}
	The assertion follows from Proposition \ref{prop:distinguished vanishing} together with the long exact sequence in Proposition \ref{prop:Serre dual}.
\end{proof}

The Corollary \ref{cor:HF and FS} invokes the following question: 
suppose that we have a distinguished collection of vanishing cycles $(V_1,\cdots, V_\mu)$ for a given singularity $f$.
Is it possible to find an exceptional collection $(K_1,\cdots, K_\mu)$ which maps to the corresponding vanishing cycles under the variation operator? 
We solve this problem topologically first.
Let us denote the geometric signed intersection number $V_i \bullet V_j = a_{ij} \in \Z $ for $i\neq j$.

\begin{defn} \label{defn:adapted}
A collection of Lagrangians $(K_1,\cdots, K_\mu)$ is said to be {\em adapted} to the distinguished collection of vanishing cycles $(V_1,\cdots, V_\mu)$ if
it satisfies the following intersection conditions.

\begin{enumerate}
\item For any $j > i$, we require that $K_j \bullet V_i =- ( -1)^\frac{n(n-1)}{2} a_{ji}$.
\item For any $j < i$, we require that $K_j  \cap V_i = \emptyset$.
\item We require that $K_j$ intersects $V_j$ at a single point positively.
\end{enumerate}
The sign in (1) is due to the Picard Lefschetz formula.
\end{defn}
\begin{remark}
We may relax the conditions (2) and (3) by $K_j \bullet V_i = 0$ and $K_j \bullet V_j = 1$ respectively,
if we are only interested in the topological considerations.
But the current conditions (2) and (3) would imply that this gives an exceptional collection in Definition \ref{defn:dist}.
\end{remark}

\begin{example}\label{ex:A}

Let us discuss  the $A_{n}$ singularity given by $f(x,y) = x^{n+1} + y^{2}$.
For simplicity, we will use pictures of $A_{4}$  case but the idea generalizes easily to all $A_n$ cases. 
The $A_4$ Milnor fiber is a genus two Riemann surface with one boundary component. 
There is a distinguished collection of vanishing cycles
$$(V^{-}_{1}, V^{-}_{2}, V^{0}_{1},V^{0}_{2}).$$
(These labelings will be introduced in the next section.)
The Milnor fiber $M_{f}$ and vanishing cycles are given in Figure \ref{fig:A4} (A).
Here we orient the vanishing cycles so that $V^0_i \bullet V^-_{j} = +1$ when they intersect.

\begin{figure}[h]
\begin{subfigure}[t]{0.43\textwidth}
\includegraphics[scale=1]{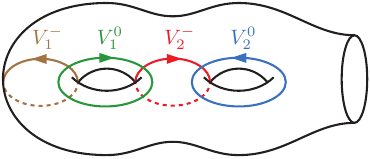}
\centering
\caption{Vanishing cycles}
\end{subfigure}
\begin{subfigure}[t]{0.43\textwidth}
\includegraphics[scale=1]{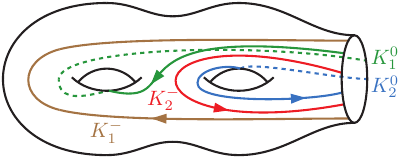}
\centering
\caption{Adapted family}
\end{subfigure}
\centering
\caption{$A_{4}$ Milnor fiber}
\label{fig:A4}
\end{figure}
One can find easily the adapted family as in Figure \ref{fig:A4} (B).
Note that $K^-_i$ only intersects $V^-_i$, but $K^0_i$ intersects  not only $V^0_i$ but also neighboring vanishing cycles
(due to condition (1) above).
\end{example}

Now, let us show that an adapted family solves the topological representability.
\begin{prop}
\label{lem:adpatedvar}
If $(K_1,\cdots, K_\mu)$ is adapted to $(V_1,\cdots, V_\mu)$, then  we have
$$\mathrm{var} ([K_i]) = ( -1)^\frac{n(n-1)}{2} [V_i], \;\;\; \forall i=1 \ldots \mu$$ 
\end{prop}

\begin{proof}
The monodromy map $\rho$ of $f$ is given by the composition of Dehn twists
$$\rho = \tau_1 \tau_2 \cdots  \tau_\mu.$$
We have the following from the intersection condition (2):
$$\tau_{i+1} \cdots \tau_\mu (K_i) = K_i.$$
Thus,
$$\rho(K_i) = \tau_1 \tau_2 \cdots \tau_i (K_i).$$
Here,  $\tau_i(K_i)$ has the homology class of $[K_i]+ ( -1)^\frac{n(n-1)}{2}[V_i]$.
Thus $\tau_{i-1}(\tau_i(K_i))$ is 
$$[K_i] +( -1)^\frac{n(n-1)}{2} [V_i]  + ( -1)^\frac{n(n-1)}{2}( K_i \bullet V_{i-1} +( -1)^\frac{n(n-1)}{2}V_i \bullet V_{i-1}  ) [V_{i-1}] =[K_i] + ( -1)^\frac{n(n-1)}{2}[V_i] $$
where the latter two terms cancel out from our intersection condition (1).
Proceeding similarly, observe that $\rho (K_i)$ represents the class $[K_i] + ( -1)^{n(n-1)/2}[V_i]$.
Therefore we have
$$[\rho (K_i)] - [K_i] = [K_i]+ ( -1)^\frac{n(n-1)}{2}[V_i] -[K_i] = ( -1)^\frac{n(n-1)}{2}[V_i].$$
\end{proof}

\section{Exceptional collections for some curve singularities}
\label{sec:exceptionalcollection}
In this section, we construct exceptional collections of Lagrangians for some curve singularities.
Recall that the elegant construction of A'Campo,  called {\em divide}  \cite{ACampo1975} gives an explicit description of
a distinguished collection of vanishing cycles and the Milnor fiber. 
For simple divides, called divides of depth 0, we will construct an exceptional collection of non-compact connected Lagrangians adapted to this distinguished collection of vanishing cycles. For divides of higher depth, we expect that the corresponding objects are twisted complexes built out of disjoint non-compact Lagrangians. We will illustrate it in the case of depth 1.

\subsection{A'Campo divide}
Let us recall briefly the notion of divide and $A\Gamma$ diagram. First, we introduce a topological definition of divide. See \cite{ACampo1975}, \cite{Areal} for more details.

Let $C$ be the disjoint union of a finite set of intervals $\{ I_{i} \cong [0,1] \}_{i \in I}$ and circles $\{ S_{j} \}_{j \in J}$, For a given immersion $\gamma : C \to B_{\epsilon}(0)\subset \mathbb{R}^{2}$ of $C$ to a disc of radius $\epsilon$,  a {\em region} is defined to be any connected component of $B_{\epsilon}(0) \setminus \gamma(C)$ which does not intersect the boundary $\partial B_{\epsilon}(0)$.
\begin{defn}
We call the image of an immersion $\gamma$ a {\em divide} if it satisfies the following conditions.
\begin{enumerate}
\item $\gamma(C)$ is connected and all intersection points of $\gamma(C)$ are only transversal double points.
\item Each circle $\gamma(S_{j})$ is disjoint from $\partial B_{\epsilon}(0)$ for all $j \in J$.
\item $\gamma(I_{i}) \cap \partial B_{\epsilon}(0) = \gamma(\partial I_{i})$ and $\gamma(I_{i})$ intersects $\partial B_{\epsilon}(0)$ transversely for all $i \in I$.
\item Each region is homeomorphic to an open disc.
\end{enumerate}
\end{defn}

Suppose that a plane curve singularity $f$ is totally real, i.e., admits a factorization into the product of $r$-many real irreducible factors where each factor should be irreducible as a complex function \cite{GZ74}, \cite{ACampo1975}. Let us consider a deformation of each factor whose product gives a real Morsification $\{ f_{t} \}_{0 \leq t \leq t_{0}}$ of $f$ for sufficiently small $t_{0} \in \mathbb{R}_{>0}$.
An {\em A'Campo divide} $\mathbb{D}_{f}$ of $f$ is given by the inclusion
$$\mathbb{D}_{f} := f_{t_{0}}^{-1}(0) \cap B_{\epsilon}(0) \hookrightarrow B_{\epsilon}(0)$$
for positive real number $\epsilon$ such that the Milnor number $\mu$ of $f$, the number $d$ of double points, and the number $r$ of irreducible factors satisfy the equation $\mu = 2d - r + 1$.

The definition of $\mathbb{D}_{f}$ depends on the choice of a Morsification, and we fix one.

\begin{defn}[$A\Gamma$ diagram \cite{GZ1974},  \cite{ACampo1975}] \label{defn:AG}
Let $\mathbb{D}_{f}$ be an A'Campo divide of $f$. Let us call any region of given A'Campo divide $\mathbb{D}_{f}$ a $+$ or $-$ region depending on the value of $f_{t_{0}}$. The A'Campo--Gusein--Zade ($A\Gamma$ for short) diagram $A\Gamma(\mathbb{D}_{f})$ of $\mathbb{D}_{f}$ is the following planar graph:
\begin{enumerate}
\item Vertices for Saddle: each double point in the divide gives a vertex of type 0, called a saddle.
\item Vertices for maximum and minimum: each bounded region of the divide gives a vertex of type $+$ or $-$ (depending on the value of $f$ in this region). There is no vertex for an unbounded region.
\item Edges:  if a $+$ region and a $-$ region meet along their boundaries, we assign an edge connecting the corresponding vertices. Connect a saddle vertex and a $\pm$ vertex if the neighborhood of the double point (for the saddle) intersects the region for the $\pm$ vertex. There is no edge between vertices of the same type.
\end{enumerate}
Let $n_{-}, n_{0}, n_{+} \geq 0$ be the number of $-$, $0$, and $+$ type vertices, respectively. 
We may choose an ordering among vertices to obtain the following ordered set
\begin{equation}\label{eq:vertices}
\{v^{-}_{1}, \dots, v^{-}_{n_{-}}, v^{0}_{1} ,\dots, v^{0}_{n_{0}}, v^{+}_{1}, \dots, v^{+}_{n_{+}} \}.
\end{equation}
The ordering among the vertices of the same type is chosen arbitrarily, and the ordering between different types is given by
their sign: $- \mbox{ vertices} < 0 \mbox{ vertices} < + \mbox{ vertices}$.
\end{defn}
See Figure \ref{fig:E6D} (A) for the $E_6$ example. 

The $A\Gamma$ diagram associated to $f$ is the Coxeter--Dynkin diagram of vanishing cycles of  the singularity $f$.
\begin{thm}[\cite{Areal}]\label{thm:ACampo}
Given a divide $\mathbb{D}_{f}$ of $f$, one can choose a set of vanishing paths so that 
the corresponding vanishing cycles in the Milnor fiber of $f$ satisfy the following properties:
\begin{enumerate}
\item The ordered set of vertices \eqref{eq:vertices} of $A\Gamma$ diagram $A\Gamma (\mathbb{D}_{f})$  corresponds to the 
distinguished collection of vanishing cycles of $f$ for this set of vanishing paths.
\item Two vanishing cycles intersect at a point if and only if the corresponding vertices are connected by an edge.
Moreover, vanishing cycles are oriented so that we have
$$V^+_{i} \bullet V^0_{j} =  V^0_{j} \bullet V^-_{k}= V^+_{k} \bullet V^-_{i}= +1$$
for any $1 \leq i \leq n_{+}$, $1 \leq j \leq n_{0}$, and $1 \leq k \leq n_{-}$ if they intersect.
\end{enumerate}
\end{thm}

We define the notion of depth.
\begin{defn} \label{defn:depth}
The {\em depth} of vertices of the $A\Gamma$ diagram $A\Gamma(\mathbb{D}_{f})$ is defined as follows.
\begin{enumerate}
\item  If a vertex $v$ meets a non-compact region in $\mathbb{R}^2 \setminus \mathbb{D}_{f}$, $v$ has depth 0.
\item Remove all vertices of depth 0 and all edges that are connected to them from $A\Gamma(\mathbb{D}_{f})$ and obtain a new diagram $A\Gamma_{1}(\mathbb{D}_{f})$. A vertex of $A\Gamma(\mathbb{D}_{f})$ has depth 1 if
$v$ is contained in $A\Gamma_{1}(\mathbb{D}_{f})$ as a depth $0$ vertex.
\item Inductively, remove all vertices of depth $<k$ and adjacent edges from $A\Gamma(\mathbb{D}_{f})$, and obtain a new diagram $A\Gamma_{k}(\mathbb{D}_{f})$. A vertex of $A\Gamma(\mathbb{D}_{f})$ has depth $k$ if
$v$ is contained in $A\Gamma_{k}(\mathbb{D}_{f})$ as a depth $0$ vertex.
\end{enumerate}
 The depth of the $A\Gamma$ diagram (and its A'Campo divide) is given by the maximum among depths of all its vertices.
\end{defn}
The depth of vanishing cycle $V^{\bullet}_{i}$ is defined to be the depth of corresponding vertex $v^{\bullet}_{i}$ in $A\Gamma(\mathbb{D}_{f})$.

\begin{figure}[h]
\includegraphics[scale=1]{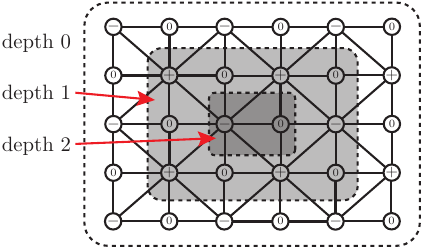}
\centering
\caption{Depth of vertices}
\end{figure}

\subsection{Exceptional collection of Lagrangians for depth 0: construction} \label{subsec:depth0}
In this section, we prove Theorem \ref{thm:ec}. First, note that Proposition \ref{prop: isotopy invariance} ensures that $HF^*_\rho$ does not depend on the choice of the representative $\rho$ in the compactly supported Hamiltonian isotopy class.
We would like to use a representative, denoted as $\tau$, described as a composition of explicit Dehn twists.

By the general construction of Fernandez de Bobabilla and  Pe{\l}ka (see Section 5 of  \cite{BP}) for the family of singularities with the constant Milnor number, there exists a symplectic isotopy between the Milnor fibers and the monodromies of $f$ and $f_\epsilon$. So we may identify them using this isotopy from now on, and regard the vanishing cycles from the divide as Lagrangians in the Milnor fiber of $f$.

Assume that  $A\Gamma$ diagram $A\Gamma(\mathbb{D}_{f})$ has depth $0$.
We denote the corresponding distinguished collection of vanishing cycles as
\begin{equation} \label{eqn:vanishing}
\overrightarrow{V}_{f} = (V^{-}_{1}, \dots, V^{-}_{n_{-}}, V^{0}_{1}, \dots, V^{0}_{n_{0}}, \dots, V^{+}_{1}, \dots, V^{+}_{n_{+}})
\end{equation}
At first, we may assume that the vanishing cycles constructed in Theorem \ref{thm:ACampo} are already exact Lagrangian. 
A priori, the construction is only topological and their exact Lagrangian representatives may violate the intersection condition.
However, one can recover minimal intersection conditions as in Theorem \ref{thm:ACampo} using compactly supported Hamiltonian isotopies (cf. \cite[Proposition 2.35]{Keating}). 

Now, the monodromy $\tau$ in the Milnor fiber of $f$ is defined by
$$\tau = \tau_{V^{-}_{1}} \circ \dots \circ \tau_{V^{-}_{n_{-}}} \circ \tau_{V^{0}_{1}} \circ \dots \circ \tau_{V^{0}_{n_{0}}} \circ \tau_{V^{+}_{1}} \circ \dots \circ \tau_{V^{+}_{n_{+}}}$$
where $\tau_{V^{\bullet}_{i}}$ is the right Dehn twist along the vanishing cycle $V^{\bullet}_{i}$.
Since two vanishing cycles $V^{\bullet}_{i}$, $V^{\bullet}_{j}$ with the same sign are disjoint, the corresponding Dehn twists $\tau_{V^{\bullet}_{i}}$, $\tau_{V^{\bullet}_{j}}$ commute. (This is related to the fact that the order in Definition \ref{defn:AG} among the vertices of the same type can be chosen arbitrarily.) Here, we assume that the supports of the Dehn twists $\tau_{V^{\bullet}_{i}}$, $\mathrm{Supp}(\tau_{V^{\bullet}_{i}})$, satisfy that, up to compactly supported Hamiltonian isotopies of vanishing cycles,
\begin{itemize}
\item $\overline{\mathrm{Supp}(\tau_{V^{\bullet}_{i}})} \cap \overline{\mathrm{Supp}(\tau_{V^{\bullet}_{j}})} = \emptyset$ if $V^{\bullet}_{i} \cap V^{\bullet}_{j} = \emptyset$ and
\item $M_{f} \setminus \cup \mathrm{Supp}(\tau_{V^{\bullet}_{i}})$ retracts to $\partial M_f$.
\end{itemize}

Next, we perturb $\tau$ to remove the codimension $0$ fixed point set near $\partial M_{f}$.
Intuitively, we will add a small rotation in the Reeb vector field direction.
A more precise procedure is as follows. 
For each Weinstein neighborhood $N_{V_i} \supset V_i$, $\tau_{V_i} \vert _{N_{V_i}\setminus V_i}$ is equal to the flow of Hamiltonian $H_{V_i}$ which only depends on the fiber direction and vanishes near $\partial N_{V_i}$. 
Therefore, by taking a small thickening of $M_{f} \setminus \cup \mathrm{Supp}(\tau_{V^{\bullet}_{i}})$,  one can find a $\tau$-invariant open set $N$ such that 
\begin{itemize}
	\item $\bigcup V_i^\bullet \subset N \subset \bigcup \mathrm{Supp}(\tau_{V^{\bullet}_{i}})$, 
	\item $N$ is an open submanifold of $M_f$ with  $\partial N$ a (union of) circle(s), 
	\item $U := M_f \setminus \overline N$ is a (union of) cylinder(s) and retracts to the $\partial M_f$,
	\item there is a $C^\infty$-trivialization of $U \cong S^1_t \times [0,1]$ so that the vector field $\partial_t $ is equal to the Reeb flow at $\partial M_f$ while it is equal to $v$, the $\pi/2$-rotation of the outward normal vector, at $\partial N$, and
	\item $\tau\vert_U$ equals the time-$1$ flow of a vector field $X$ which satisfies $\langle X, \partial_t\rangle \geq 0$.
\end{itemize}	
Then consider a Hamiltonian function $H$ such that $H\vert_{N}=0$ and $H\vert_U$ is an increasing function on the radial coordinate with respect to the $C^\infty$-trivialization. 
For a small enough $\epsilon$, the composition $\phi^{1}_{\epsilon H} \circ \tau$ rotates $U$ a little along $S^1$-direction. 
In particular, if $L\cap U$ becomes the graph of some function on the radial coordinate, then $\phi^{1}_{\epsilon H} \circ \tau$ will displace it from itself inside $U$. We continue to denote this perturbation by $\tau$ by abuse of notation.

Note that A'Campo \cite{Areal} found a way to recover the Milnor fiber of $f$ from the divide, using a building block (see Figure \ref{fig:K-}) for a neighborhood of a double point, 
and connecting them with half twists following the diagram (see Figure \ref{fig:E6}). Also, vanishing cycles in Theorem \ref{thm:ACampo} can be drawn in the building block and Milnor fiber. We will construct an adapted Lagrangian family consisting of connected Lagrangians which will form a desired exceptional collection.

\begin{lemma} \label{lem:depth0}
Let $A\Gamma (\mathbb{D}_{f})$ be the depth $0$ $A\Gamma$ diagram and $\overrightarrow{V}_{f}$ as in \eqref{eqn:vanishing} be a distinguished collection of vanishing cycles associated with $A\Gamma (\mathbb{D}_{f})$. Then there exists an adapted family
$$\overrightarrow{K}_{f} = (K^{-}_{1}, \dots, K^{-}_{n_{-}}, K^{0}_{1}, \dots, K^{0}_{n_{0}}, \dots, K^{+}_{1}, \dots, K^{+}_{n_{+}})$$
for $\overrightarrow{V}_{f}$ such that any $K^{\bullet}_{i}$ is represented by a non-compact connected Lagrangian.
\end{lemma}

Before giving a proof, we briefly explain the construction. First, we find non-compact Lagrangians in the union of closed supports, $\cup \overline{\mathrm{Supp}(\tau_{V^{\bullet}_{i}})}$. Then, at the boundary of non-compact Lagrangian, we extend it along the gradient flow of a chosen Hamiltonian $H$ (and smooth it if needed). These extensions do not change any intersection patterns between Lagrangians. The reason why we divide the construction into two steps is that later, it is convenient to see intersections between $K^{\bullet}_{i}$ and its monodromy image $\tau(K^{\bullet}_{i})$. 

\begin{proof}
We will find an adapted family satisfying the intersection conditions in Definition \ref{defn:adapted}. 
 For a double point of the divide, we can choose an open set of the Milnor fiber, diffeomorphic to the 
 building block (see Figure \ref{fig:K-}). The double point meets 4 regions, and the building block has 4 corresponding arcs
 which may become parts of the vanishing cycles. Also, there exists a circle in the cylindrical part of the building block which
 corresponds to the vanishing cycle of the double point.

The main idea of the construction is the following.
From the depth zero condition, each double point meets at least one unbounded region
and thus the arc corresponding to the unbounded region does not become a part of the vanishing cycle. Let us call this an unnecessary arc. If a non-compact Lagrangian intersects an unnecessary arc in the building block, it does not create a new intersection number with the vanishing cycles.

Let us explain each case in more detail.
For  the vanishing cycles $V^{-}_{i}$ corresponding to the bounded negative region, we need to find a non-compact Lagrangian $K^{-}_{i}$ satisfying 
$$K^{-}_{i} \bullet V^{\bullet}_{j} = \begin{cases} 1 & (\bullet = -, \; i=j), \\ 0 &(\mbox{otherwise}). \end{cases}$$
If all the positive neighboring regions of the given negative region are bounded, then the depth of the negative region cannot be zero.
Thus, there exists at least one neighboring unbounded positive region. 
From the explanation above, we have an unnecessary arc in the building block of one of the neighboring double points.
In Figure \ref{fig:K-}, the arc on the left should have a counterpart on the right, but we did not draw it as it is an unnecessary arc.
Now, we choose $K_i^-$ as in Figure \ref{fig:K-}. As explained above, it is obtained by an extension of a small piece in $\mathrm{Supp}(\tau_{V^{-}_{i}})$. Note that $K_i^-$ is contained in this building block and intersects $V_i^-$ only 
(it intersects the unnecessary arc, but this is not relevant to the adaptedness condition). Moreover, we can say that $K^{-}_{i}$ does not intersect the support of any Dehn twists except $\mathrm{Supp}(\tau_{V^{-}_{i}})$.
The same construction works for any $1 \leq i \leq n_{-}$.

\begin{figure}[h]
\includegraphics[scale=0.7]{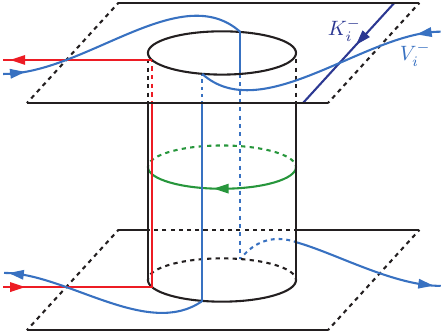}
\centering
\caption{$K^{-}_{i}$ in the building block}
\label{fig:K-}
\end{figure}

There are two cases for each vanishing cycle $V^{0}_{i}$ associated with a double point.  A non-compact Lagrangian $K^{0}_{i}$ should satisfy
$$K^{0}_{i} \bullet V^{\bullet}_{j} = \begin{cases} 1 & (\bullet = 0, \; i=j), \\  1 & (\bullet=-, V^{-}_{j}\mbox{ and }V^{0}_{i}\mbox{ are connected in }A\Gamma \mbox{ diagram}),  \\ 0 &(\mbox{otherwise}). \end{cases}$$
A priori, every double point has two positive neighboring regions and two negative neighboring regions. Then, there is at least one unbounded positive or negative neighboring region. If there is an unbounded negative region, in the building block of chosen double point, the arc on the front becomes an unnecessary arc. In this case, a non-compact Lagrangian $K^{0}_{i}$ is given as in Figure \ref{fig:K0} (A). Here, we should choose $K^{0}_{i}$ to satisfy
\begin{equation} \label{eqn:supp}
\tau_{V^{0}_{i}}(K^{0}_{i}) \cap \mathrm{Supp}(\tau_{V^{-}_{j}}) = \emptyset.
\end{equation}
It is always possible by shrinking the support of $\tau_{V^{-}_{j}}$. $K^{0}_{i}$ intersects $V^{0}_{i}$ and $V^{-}_{j}$ which is not the unnecessary arc (it does not intersect the unnecessary arc in this case). If there is an unbounded positive region, the arc on the right is an unnecessary arc. Hence, $K^{0}_{i}$ intersects $\mathrm{Supp}(\tau_{V^{0}_{i}})$ and $\mathrm{Supp}(\tau_{V^{-}_{j}})$ only. Also, one can choose a non-compact Lagrangian $K^{0}_{i}$ similarly as in Figure \ref{fig:K0} (B).

\begin{figure}[h]
\begin{subfigure}[t]{0.43\textwidth}
\includegraphics[scale=0.7]{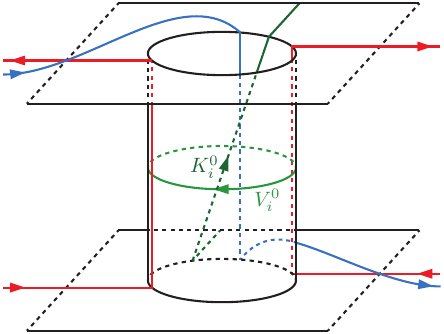}
\centering
\caption{ }
\label{fig:K01}
\end{subfigure}
\begin{subfigure}[t]{0.43\textwidth}
\includegraphics[scale=0.7]{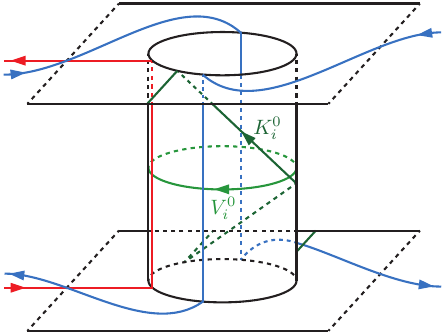}
\centering
\caption{ }
\label{fig:K02}
\end{subfigure}
\centering
\caption{$K^{0}_{i}$ in the building block}
\label{fig:K0}
\end{figure}

Finally, for the vanishing cycles $V^{+}_{i}$ corresponding to the bounded positive region, a non-compact Lagrangian $K^{+}_{i}$ satisfies
$$K^{+}_{i} \bullet V^{\bullet}_{j} = \begin{cases} 1 & (\bullet = +, \; i=j), \\ 
1 & (\bullet=0, V^{0}_{j}\mbox{ and }V^{+}_{i}\mbox{ are connected in }A\Gamma \mbox{ diagram}), \\
1 & (\bullet=-, V^{-}_{j}\mbox{ and }V^{+}_{i}\mbox{ are connected in }A\Gamma \mbox{ diagram}),  \\
0 &(\mbox{otherwise}). \end{cases}$$
Note that this intersection condition is the same as $V^{+}_{i} \bullet V^{\bullet}_{j}$ except when $\bullet = +$, $i=j$. Since the depth is $0$, we can find a part where $V^{+}_{i}$ lives alone as in the $-$ case. Let $\widetilde{K}^{+}_{i}$ be the non-compact Lagrangian in Figure \ref{fig:K+} which intersects $\mathrm{Supp}(\tau_{V^{+}_{i}})$ only. Then, we define $K^{+}_{i} := \tau^{-1}_{V^{+}_{i}}(\widetilde{K}^{+}_{i})$. Topologically, $K^{+}_{i}$ is described as follows.
We cut $V_i^+$ near the intersection of $V_i^+$ and the unnecessary arc, and attach two new ends to the boundary of the Milnor fiber 
as drawn in Figure \ref{fig:K+}. This $K^{+}_{i}$ and $V^{+}_{i}$ have almost the same intersection conditions with other vanishing cycles except $V^{+}_{i}$ itself, i.e., we have $K^{+}_{i} \bullet V^{+}_{i} = 1$ and
$$K^{+}_{i} \bullet V^{\bullet}_{i} = V^{+}_{i} \bullet V^{\bullet}_{i}$$
for any $V^{\bullet}_{i}$ except $V^{+}_{i}$. 
Thus, this $K^{+}_{i}$ satisfies the desired conditions.

\begin{figure}[h]
\includegraphics[scale=0.7]{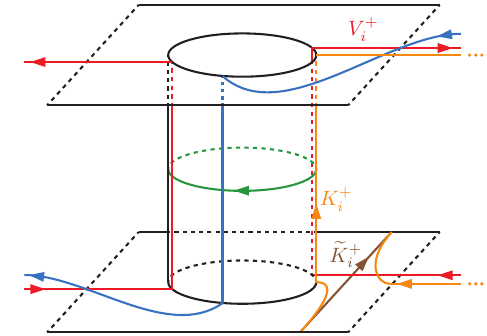}
\centering
\caption{$K^{+}_{i}$ in the building block}
\label{fig:K+}
\end{figure}

These constructions give the adapted family for $\overrightarrow{V}_{f}$ in any depth $0$ case. Obviously, each $K^{\bullet}_{i}$ is connected by construction.
\end{proof}

\subsection{Exceptional collection of Lagrangians for depth 0: computations} \label{subsec:depth0}
We first make the following observations on intersections between constructed Lagrangians.
\begin{lemma} \label{lem:disjoint}
The following hold for the adapted family that is constructed in Lemma \ref{lem:depth0}.
\begin{enumerate}
\item Any pair of Lagrangians in the family are disjoint from each other. 
\item $K^{\bullet}_{i} \cap \tau(K^{\bullet}_{j}) = \emptyset$ for all $\bullet = -,0,+$ and $1 \leq i, j \leq n_{\bullet}$.
\item $K^{-}_{i} \cap \tau(K^{0}_{j}) = \emptyset$, $K^{-}_{i} \cap \tau(K^{+}_{j}) = \emptyset$, and $K^{0}_{i} \cap \tau(K^{+}_{j}) = \emptyset$ for any $i,j$.
\end{enumerate}
\begin{proof}
The first statement follows directly from the construction.

Recall that the monodromy $\tau$ is defined by
$$\tau = \phi^{1}_{\epsilon H} \circ \tau_{V^{-}_{1}} \circ \dots \circ \tau_{V^{-}_{n_{-}}} \circ \tau_{V^{0}_{1}} \circ \dots \circ \tau_{V^{0}_{n_{0}}} \circ \tau_{V^{+}_{1}} \circ \dots \circ \tau_{V^{+}_{n_{+}}}.$$
We claim that it is enough to consider only one Dehn twist for each $K^{\bullet}_{j}$ because there is only the Dehn twist $\tau_{V^{\bullet}_{j}}$ acts nontrivially on $K^{\bullet}_{j}$.
For the non-compact Lagrangian $K^{-}_{j}$ given in Lemma \ref{lem:depth0}, $K^{-}_{j}$ intersects $\mathrm{Supp}(\tau_{V^{-}_{j}})$ only. Hence,
$$\tau_{V^{-}_{j+1}} \circ \dots \circ \tau_{V^{-}_{n_{-}}} \circ \tau_{V^{0}_{1}} \circ \dots \circ \tau_{V^{0}_{n_{0}}} \circ \tau_{V^{+}_{1}} \circ \dots \circ \tau_{V^{+}_{n_{+}}}(K^{-}_{j}) = K^{-}_{j}.$$
Let us consider the case $i=j$ first. After taking $\tau_{V^{-}_{j}}$, we have
$$K^{-}_{j} \cap \tau_{V^{-}_{j}} \circ \tau_{V^{-}_{j+1}} \circ \dots \circ \tau_{V^{-}_{n_{-}}} \circ \tau_{V^{0}_{1}} \circ \dots \circ \tau_{V^{0}_{n_{0}}} \circ \tau_{V^{+}_{1}} \circ \dots \circ \tau_{V^{+}_{n_{+}}}(K^{-}_{j}) \cong \partial K^{-}_{j} \times I$$
for some interval $I$. A non-compact Lagrangian
$$\tau_{V^{-}_{j}} \circ \tau_{V^{-}_{j+1}} \circ \dots \circ \tau_{V^{-}_{n_{-}}} \circ \tau_{V^{0}_{1}} \circ \dots \circ \tau_{V^{0}_{n_{0}}} \circ \tau_{V^{+}_{1}} \circ \dots \circ \tau_{V^{+}_{n_{+}}}(K^{-}_{j})$$
does not intersect any $\mathrm{Supp}(\tau_{V^{-}_{k}})$ for $1 \leq k \leq n_{-}$ since $V^{-}_{k} \cap V^{-}_{j} = \emptyset$ for all $1 \leq k \leq n_{-}$ by the construction. Thus, 
$$K^{-}_{j} \cap \tau_{V^{-}_{1}} \circ \dots \circ \tau_{V^{+}_{n_{+}}}(K^{-}_{j}) \cong \partial K^{-}_{j} \times I.$$
Then using the perturbation $\phi^{1}_{\epsilon H}$, we can conclude for any $1 \leq j \leq n_{-}$,
$$K^{-}_{j} \cap \tau(K^{-}_{j}) = \emptyset.$$
If $i \neq j$, the difference is that $K^{-}_{i} \cap K^{-}_{j} = \emptyset$ by construction. This implies they are already disjoint before the Hamiltonian deformation:
$$K^{-}_{i} \cap \tau_{V^{-}_{1}} \circ \dots \circ \tau_{V^{+}_{n_{+}}}(K^{-}_{j}) \cong \emptyset.$$
Hence, we have $K^{-}_{i} \cap \tau(K^{-}_{j}) = \emptyset$ and this proves (2) for $\bullet = -$ case.
 
 Next, let $K^{0}_{j}$ be the non-compact Lagrangian given in Lemma \ref{lem:depth0}. Similarly, we have
\begin{enumerate}[(i)]
\item $\tau_{V^{0}_{j+1}} \circ \dots \circ \tau_{V^{0}_{n_{0}}} \circ \tau_{V^{+}_{1}} \circ \dots \circ \tau_{V^{+}_{n_{+}}}(K^{0}_{j}) = K^{0}_{j}$,
\item $K^{0}_{i} \cap \tau_{V^{0}_{j}} \circ \tau_{V^{0}_{j+1}} \circ \dots \circ \tau_{V^{0}_{n_{0}}} \circ \tau_{V^{+}_{1}} \circ \dots \circ \tau_{V^{+}_{n_{+}}}(K^{0}_{j}) \cong \begin{cases} \partial K^{0}_{i} \times I &(i=j), \\ \emptyset &(i \neq j) \end{cases}$ for some interval $I$, and
\item $\tau_{V^{0}_{j}} \circ \tau_{V^{0}_{j+1}} \circ \dots \circ \tau_{V^{0}_{n_{0}}} \circ \tau_{V^{+}_{1}} \circ \dots \circ \tau_{V^{+}_{n_{+}}}(K^{0}_{j})$ does not intersect any $\mathrm{Supp}(\tau_{V^{-}_{k}})$ for $1 \leq k \leq n_{-}$ and $\mathrm{Supp}(\tau_{V^{0}_{l}})$ for $1 \leq l \leq n_{0}$.
\end{enumerate}
The third one follows from the property \eqref{eqn:supp} of $K^{0}_{i}$. Hence, we get $K^{0}_{i} \cap \tau(K^{0}_{j}) = \emptyset$ for all $1 \leq i,j \leq n_{0}$.

Lastly, recall that a non-compact Lagrangian $K^{+}_{j}$ is defined by $\tau^{-1}_{V^{+}_{j}}(\widetilde{K}^{+}_{j})$ for a non-compact Lagrangian $\widetilde{K}^{+}_{j}$ in Lemma \ref{lem:depth0}. Since $\mathrm{Supp}(\tau_{V^{+}_{j}}) \cap \mathrm{Supp}(\tau_{V^{+}_{k}}) = \emptyset$ for any $k \neq j$,
$$\tau_{V^{+}_{j}} \circ \dots \circ \tau_{V^{+}_{n_{+}}}(K^{+}_{j}) = \widetilde{K}^{+}_{j}.$$
Then, by the intersection conditions of $\widetilde{K}^{+}_{j}$ with the supports of Dehn twists as above, it is easy to see that $K^{+}_{i} \cap \tau(K^{+}_{j}) = \emptyset$ holds for all $1 \leq i,j \leq n_{+}$.

The third statement (3) of the lemma easily follows from the facts that any distinct two Lagrangians in the adapted family $\overrightarrow{K}_{f}$ are disjoint and the intersection conditions
$$K^{-}_{i} \cap \mathrm{Supp}(\tau_{V^{0}_{j}}) = \emptyset, K^{-}_{i} \cap \mathrm{Supp}(\tau_{V^{+}_{j}}) = \emptyset, \mbox{ and } K^{0}_{i} \cap \mathrm{Supp}(\tau_{V^{+}_{j}}) = \emptyset.$$
\end{proof}
\end{lemma}

Recall from Subsection \ref{subsec:main} the quasi-isomorphism of monodromy Floer cohomology:
\begin{equation} \label{eqn:quasi}
HF_{\tau}^{*}(K^{\bullet}_{i},K^{\bullet}_{j}) \cong HW^{*}\big((K^{\bullet}_{i}[1] \xrightarrow{\mathcal{CO}^{\tau}_{K^{\bullet}_{i}}(\Gamma)} \tau(K^{\bullet}_{i})),K^{\bullet}_{j} \big).
\end{equation}
By Lemma \ref{lem:disjoint}, $\mathcal{CO}^{\tau}_{K^{\bullet}_{i}}(\Gamma) \in HW^*(K^{\bullet}_{i}, \tau(K^{\bullet}_{i}))$ is represented by the sum of two shortest Reeb chord from $\partial K^{\bullet}_{i}$ to $\partial \tau(K^{\bullet}_{i})$ since there is no interior intersection.  The following will be helpful for computations.
\begin{lemma} \label{lem:i}
Let $K$ and $K^{\prime}$ be non-compact connected Lagrangians in the Milnor fiber $M_{f}$ such that
\begin{enumerate}[(i)]
\item $K \cap K^{\prime} = \emptyset$, $K \cap \tau(K) = \emptyset$, and $K^{\prime} \cap \tau(K) = \emptyset$.
\item $K$ and $\tau(K)$ are not isomorphic in the wrapped Fukaya category of $M_{f}$.
\end{enumerate}
Then, the monodromy Floer cohomology can be computed as follows.
\begin{enumerate}
\item $HF^{*}_{\tau}(K,K) \cong \mathbb{K} \cdot e[0]$ where $e$ is the identity morphism and
\item $HF^{*}_{\tau}(K,K^{\prime}) = 0$.
\end{enumerate}
\begin{proof}
We can assume that the differentials vanish on $CW^{*}(\tau(K),K)$ and $CW^{*}(K,K)$ by choosing an appropriate Hamiltonian since $K$ and $\tau(K)$ are disjoint (e.g., see the proof of Proposition \ref{prop: variation dim 2}). Then, there is one interior intersection point $e$, which represents the identity morphism in $CW^{*}(K,K)$. Any other Hamiltonian chord representing a morphism in $CW^{*}(K,K)$ has a negative action value and hence it defines a morphism belonging to $CW^{*}_{<0}(K,K)$. 

By the assumption of the differentials, we have
\begin{equation} \label{eqn:cone}
HF^{*}_{\tau}(K, K) \cong HW^{*}(\tau(K)),K)[1] \xrightarrow{m_2(\mathcal{CO}_{K}^\tau(c(\Gamma)){,}-)} HW^{*}(K,K).
\end{equation}
where $c$ is the continuation map.
Remark that Theorem \ref{thm: PSS CO diagram} is still valid for $\tau$. 
Then by Corollary \ref{cor: Gamma and GammaL}, we can conclude that the term of $\mathcal{CO}_{K}^\tau(c(\Gamma))$ with the lowest action value must be $\Gamma_K$, the cochain of shortest Reeb chord from $K$ to $\tau(K)$ for any $K$.
Therefore, $\mathcal{CO}_{K}^\tau(c(\Gamma))$ is $\Gamma_K$ under the assumption (i).

Given a nontrivial morphism $q \in HW^{*}_{<0}(K,K)$, one can find a triangle which contributes to the $m_{2}$-product $m_2(\Gamma_K, p) = q$ for some $p \in HW^*(\tau(K), K)$. This triangle can be described as follows. Let $Q$ be a Hamiltonian which is $C^{2}$-small in some compact subset $K \subset M_{f}$ and quadratic in the complement of $K$, and $\phi = \phi^{1}_{Q}$ be a time-$1$ Hamiltonian flow of $Q$. With the labeling of morphisms in Figure \ref{fig:tri}, there is precisely one triangle for the equation $m_2(\Gamma_K, p_{k}) = q_{k}$ for each $k \geq 1$.
\begin{figure}[h]
\includegraphics[scale=0.9]{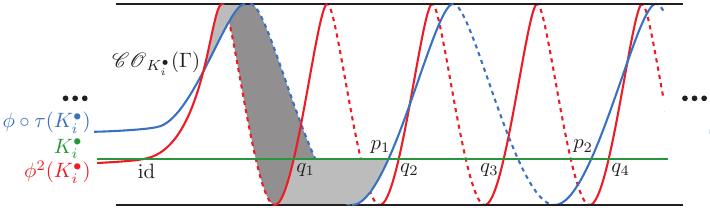}
\centering
\caption{Holomorphic triangle for the quantum cap action}
\label{fig:tri}
\end{figure}

If the identity morphism becomes an output of some product, then this implies that $K$ and $\tau(K)$ are isomorphic objects in the wrapped Fukaya category. It is a contradiction. Thus, we have the isomorphism
$$HF^{*}_{\tau}(K, K) \cong \mathbb{K} \cdot e[0].$$
This proves the first result (1).

For the second statement (2), note that there is no morphism corresponding to the interior intersection in $HW^{*}(\tau(K),K^{\prime})$ and $HW^{*}(K,K^{\prime})$. Then, a modification of Figure \ref{fig:tri} explains all the triangles which give the isomorphism between two cohomologies:
$$HW^{*}(\tau(K),K^{\prime}) \xrightarrow[\simeq]{m_2(\Gamma_K{,}-)} HW^{*}(K,K^{\prime}).$$
Hence, the mapping cone of the twisted quantum cap action vanishes and this gives the second result.
\end{proof}
\end{lemma}

\begin{prop} \label{prop:depth0dist}
The adapted family in Lemma \ref{lem:depth0} forms an exceptional collection.
\end{prop}
\begin{proof}
When there is only one Lagrangian in the adapted family $\overrightarrow{K}_{f}$ (Morse case), the condition (ii) in Lemma \ref{lem:i} does not hold. But, the Milnor fiber is just a cylinder in this case and it is easy to see that the identity morphism cannot be an output of quantum cap action. This proves 
$$HF^{*}_{\tau}(K, K) \cong \mathbb{K} \cdot e[0],$$
i.e., $\overrightarrow{K}_{f}$ is an exceptional collection.

Now, we assume that $\overrightarrow{K}_{f}$ consists of more than one Lagrangian. It is enough to show that the conditions in Lemma \ref{lem:i} hold for the adapted family $\overrightarrow{K}_{f}$ according to order.

Let us choose any pair of non-compact Lagrangians $K^{\bullet}_{i}$ and $K^{\circ}_{j}$ from $\overrightarrow{K}_{f}$ with $K^{\bullet}_{i} > K^{\circ}_{j}$. Then, the condition (1) for this pair directly follows from Lemma \ref{lem:disjoint}.

For the condition (2), consider a Lagrangian $K^{-}_{i}$ first. It intersects $V^{-}_{i}$ only among vanishing cycles. However, its monodromy image $\tau(K^{-}_{i})$ meets once all the other vanishing cycles which have an intersection with $V^{-}_{i}$. Since there is always a vertex with a different type by the assumption, this implies that there is a vanishing cycle $V^{\bullet}_{j}$ such that $K^{-}_{i} \cap V^{\bullet}_{j} = \emptyset$ and $\tau(K^{-}_{i}) \cap V^{\bullet}_{j} = p$, i.e.,
$$HF^{*}(K^{-}_{i},V^{\bullet}_{j}) = 0 \mbox{ and } HF^{*}(\tau(K^{-}_{i}),V^{\bullet}_{j}) \neq 0.$$
It guarantees that $K^{-}_{i}$ and $\tau(K^{-}_{i})$ are not isomorphic in $\mathcal{WF}(M_{f})$.

One can show the condition (2) for other cases similarly. Note that there are a vanishing cycle $V^{-}_{j}$ that intersects with $K^{0}_{i}$, but not with $\tau(K^{0}_{i})$ (see Figure \ref{fig:K0}) for $K^{0}_{i}$ and a vanishing cycle $V^{0}_{k}$ that intersects with $K^{+}_{i}$, but not with $\tau(K^{+}_{i})$ (see Figure \ref{fig:K+}) for $K^{+}_{i}$. Then, the statement follows from the Lemma \ref{lem:i}.
\end{proof} 

In the next lemma, we prove that $\mathcal V (K_i)$ is quasi-isomorphic to $V_i$ up to an odd shift.

\begin{lemma}
\label{lem:varandVforadaptedK}
In $\mathcal{WF}(M)$, $\mathcal V(K_i^\bullet)$ is quasi-isomorphic to $V_i^\bullet$ up to an odd shift. 
Moreover, we can choose gradings on $\{K_i^\bullet\}$ which makes $\mathcal V(K_i^\bullet) \simeq V_i^\bullet[1]$.
 
\end{lemma}
\begin{proof}
Recall that our $K_i^\bullet$ intersects $V_i^\bullet$ only once.
In the course of the proof of Lemma \ref{lem:disjoint}, we also observed that the only Dehn twists act non-trivially on $K_i^\bullet$ is $\tau_{V_i^\bullet}$. 
Therefore $\tau(K^\bullet_i)$ also intersects $V^\bullet_i$ at single point.  
Then we can conclude that $V_i^\bullet$ is quasi-isomorphic to $\mathcal V(K_i^\bullet)$ up to a shift using Proposition \ref{prop: variation dim 2}, 
The shift is odd because the choice of an orientation of $K_i^\bullet$ is opposite to that of Proposition \ref{prop: variation dim 2}. 

Note that we have only decided the orientations of $\{K_i^\bullet\}$ and said nothing about gradings until now.
We have the freedom of shifting the grading of $K_i^\bullet$ by an arbitrary even number. 
Since the calculation involves $V_i$ only, any even shift does not affect the gradings of other $K_j^\circ$ or previous computations of $HF^*_\tau$. 
\end{proof}
By combining Proposition \ref{prop:depth0dist} and Lemma \ref{lem:varandVforadaptedK}, we obtain; 

\begin{thm}
\label{thm:ecadapted}
Let $f$ be an isolated plane curve singularity having an A'Campo divide of depth 0.
Denote by $(V_1, \ldots, V_\mu)$ the distinguished collection of Lagrangian vanishing cycles in $M$ from the given divide.
Then, there exists an exceptional collection (with respect to $HF^*_\rho$) of non-compact connected Lagrangians $(K_1,\cdots, K_\mu)$  adapted to $(V_1,\cdots, V_\mu)$ satisfying
$$\mathcal{V}(K_i)\simeq V_i[1] \in \mathcal{WF}(M).$$ 
\end{thm}

\subsection{Examples} 
In this subsection, we consider $E_{6}$ and $A_{n}$ singularities as examples. 
We present explicit computations that hopefully give the readers some flavor of $HF^*_\tau$.
We also describe an interesting observation regarding the relationship between $A_{n}$ singularities and a partially wrapped Fukaya category.
\subsubsection{$E_6$-singularity} \label{subsubsec:E6}

We investigate the $E_{6}$ example given in the introduction. There is the adapted family $(K^{-}_{1}, K^{-}_{2}, K^{0}_{1}, K^{0}_{2}, K^{0}_{3}, K^{+}_{1})$ in Figure \ref{fig:E6} (B) associated with the vanishing cycles in Figure \ref{fig:E6} (A). By Proposition \ref{prop:depth0dist}, we know that the family $(K^{-}_{1}, K^{-}_{2}, K^{0}_{1}, K^{0}_{2}, K^{0}_{3}, K^{+}_{1})$ is exceptional. 
Any two distinct Lagrangians in the family are disjoint from each other. 
Moreover,  by Lemma \ref{lem:i},  their monodromy Floer cohomology $HF^{*}_{\tau}(K^{\bullet}_{i}$, $K^{\circ}_{j})$ vanishes when $\tau(K^{\bullet}_{i}) \cap K^{\circ}_{j} = \emptyset$ holds. Hence, it is enough to find pairs $(K^{\bullet}_{i}, K^{\circ}_{j})$ such that $K^{\bullet}_{i} < K^{\circ}_{j}$ and $\tau(K^{\bullet}_{i}) \cap K^{\circ}_{j} \neq \emptyset$ in Figure \ref{fig:E6ex} (A).

\begin{figure}[h]
\begin{subfigure}[t]{0.43\textwidth}
\includegraphics[scale=0.8]{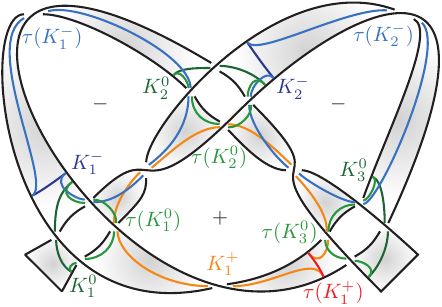}
\centering
\caption{Adapted family and monodromy images}
\end{subfigure}
\begin{subfigure}[t]{0.43\textwidth}
\includegraphics[scale=0.8]{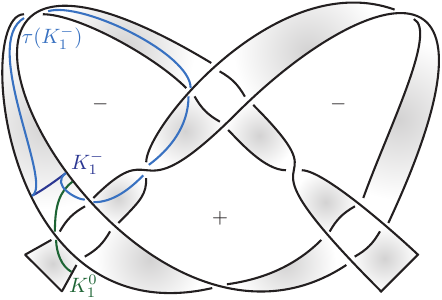}
\centering
\caption{Three Lagrangians $K^{-}_{1}$, $\tau(K^{-}_{1})$, and $K^{0}_{1}$}
\end{subfigure}
\centering
\caption{ }
\label{fig:E6ex}
\end{figure}

One of such pairs is $(K^{-}_{1}, K^{0}_{1})$ and we compute the monodromy Floer cohomology $HF^{*}_{\tau}(K^{-}_{1}, K^{0}_{1})$. Recall that we have the following isomorphism from the proof of Proposition \ref{prop:HF and V};
$$HF^{*}_{\tau}(K^{-}_{1},K^{0}_{1}) \cong H^*\left(CW^{*}(\tau(K^{-}_{1}),K^{0}_{1})[1] \xrightarrow{m_2(\mathcal{CO}^\tau_{K^{-}_{1}}(\Gamma){,}-)} CW^{*}(K^{-}_{1},K^{0}_{1})\right).$$
We claim that $m_2(\mathcal{CO}^\tau_{K^{-}_{1}}(\Gamma){,}-)$ is surjective with $1$-dimensional kernel.
The first component $CW^{*}(\tau(K^{-}_{1}),K^{0}_{1})[1]$ consists of wrapped generators (which appear near boundary, or equivalently, have negative action values) and one interior intersection point $\tau(K^{-}_{1}) \cap K^{0}_{1} = p_{1}$. The second component $CW^{*}(K^{-}_{1},K^{0}_{1})$ contains only wrapped generators. 
Then, $m_2(\mathcal{CO}^\tau_{K^{-}_{1}}(\Gamma){,}-)$ gives a bijection between all of the wrapped generators in the first component and the second component (see Lemma \ref{lem:i}).
Also, one can directly check that there are no bigons or triangles corresponding to $m_1(p_1)$ and $m_2(\mathcal{CO}^\tau_{K^{-}_{1}}(\Gamma),p_{1})$ in Figure \ref{fig:E6ex} (B). 
This implies that $$[p_{1}] \in HF^{*}_{\tau}(K^{-}_{1},K^{0}_{1})$$ 
and it is the unique generator of this cohomology.

There are $9$ pairs of Lagrangians that have nontrivial monodromy Floer cohomology. In each case, they have one-dimensional monodromy Floer cohomology generated by the following intersection point:
\begin{align*}
\tau(K^{-}_{1}) \cap K^{0}_{1}& = p_{1},\; \tau(K^{-}_{1}) \cap K^{0}_{2} = p_{2},\; \tau(K^{-}_{1}) \cap K^{+}_{1} = p_{3}, \;\tau(K^{-}_{2}) \cap K^{0}_{2} = p_{4},\; \tau(K^{-}_{2}) \cap K^{0}_{3} = p_{5}, \\
&\tau(K^{-}_{2}) \cap K^{+}_{1} = p_{6},\; \tau(K^{0}_{1}) \cap K^{+}_{1} = p_{7}, \;\tau(K^{0}_{2}) \cap K^{+}_{1} = p_{8}, \; \tau(K^{0}_{3}) \cap K^{+}_{1} = p_{9}.
\end{align*}
For any pair, this can be shown as the same as the above pair $(K^{-}_{1}, K^{0}_{1})$.

This computation can be represented by the following quiver (as a vector space, not an algebra):
\begin{equation}\label{diagram:E6}
\begin{tikzcd}[column sep = {5em,between origins}, row sep = {5em,between origins}]
K^{-}_{1} \arrow[r] \arrow[dr] \arrow[d] & K^{0}_{2} \arrow[d] & K^{-}_{2} \arrow[l] \arrow[dl] \arrow[d] \\
K^{0}_{1} \arrow[r] & K^{+}_{1} & K^{0}_{3} \arrow[l]
\end{tikzcd}.
\end{equation}
This quiver is the same as the $A\Gamma$ diagram with orientation (see Figure \ref{fig:E6D} (B)). Note that we need to consider the degree of generators. However, we know about vanishing cycles that their morphism spaces are supported in degree $0$. Then, by Proposition \ref{prop:Serre dual} (2), the monodromy Floer cohomologies also live in degree $0$ part.

\subsubsection{$A_n$-singularity}
\label{subsec:An}

We revisit $A_n$-singularity $f(x,y) = x^{n+1} + y^{2}$.
It has a well-known divide of depth zero whose distinguished collection and adapted family coincide with the one discussed in the previous Example \ref{ex:A}.

We continue from the previous $A_4$ example (see Figure \ref{fig:A4} (A)), but again the whole story generalizes to arbitrary $A_n$. 
There, we obtained distinguished collection consisting of four vanishing cycles $\overrightarrow{V}_{f} = (V^{-}_{1}, V^{-}_{2}, V^{0}_{1}, V^{0}_{2})$ and its adapted family $\overrightarrow{K}_{f} = (K^{-}_{1}, K^{-}_{2}, K^{0}_{1}, K^{0}_{2})$. For notational simplicity, let us denote $(V^{-}_{1}, V^{-}_{2}, V^{0}_{1}, V^{0}_{2})$ by $(V_1, V_3, V_2, V_4)$, and the adapted family $(K^{-}_{1}, K^{-}_{2}, K^{0}_{1}, K^{0}_{2})$ by $(L_{1}, L_{3}, L_{2}, L_{4})$.
Then endomorphism algebra of $\overrightarrow{V}_{f}=(V_1, V_3, V_2, V_4)$ is described using the following quiver $Q$:
$$\begin{tikzcd}
	\bullet_{1} \arrow[r] & \bullet_{2} & \bullet_{3} \arrow[l] \arrow[r] & \bullet_{4}.
\end{tikzcd}$$
We can choose gradings on vanishing cycles so that the morphisms between vanishing cycles in the Fukaya--Seidel category are concentrated at degree zero. 
After that, we have
\begin{equation}
\label{eqn:Ancomputation}
	HF^{*}_{\tau} \left( \bigoplus_{i=1}^{4} L_{i}, \bigoplus_{i=1}^{4} L_{i} \right) \cong
	\mathrm{Hom}^{*}_{FS} \left( \bigoplus_{i=1}^4 V_{i}, \bigoplus_{i=1}^4 V_{i} \right) \cong 
	\mathbb{K}Q.
\end{equation}
where $\mathbb K Q$ denote the path algebra associated with $Q$. 

Let us compute it more explicitly.
From \eqref{eqn:Ancomputation} or by Proposition \ref{prop:depth0dist}, we know that $HF^*_\tau$ between two distinct Lagrangians vanishes except:
$$HF^{*}_{\tau}(L_{1}, L_{2}), HF^{*}_{\tau}(L_{3},L_{2}), \mbox{ and }HF^{*}_{\tau}(L_{3},L_{4}).$$
These are all one-dimensional and we would like to specify their generators. 
To do this, we give a more convenient description of the Milnor fiber and $\overrightarrow{K}_{f}$.
A fundamental domain of $M_{f}$ given in Figure \ref{fig:A4} (A) is a decagon with a gluing rule as in Figure \ref{fig:A4fund} (A). 
Each circle on the chosen five vertices represents a part of the unique boundary component of $M_{f}$. 
In this fundamental domain, the adapted family $\overrightarrow{K}_{f}$ \ref{fig:A4} (B) can be drawn as in Figure \ref{fig:A4fund} (B).

\begin{figure}[h]
\begin{subfigure}[t]{0.33\textwidth}
\includegraphics[scale=0.85]{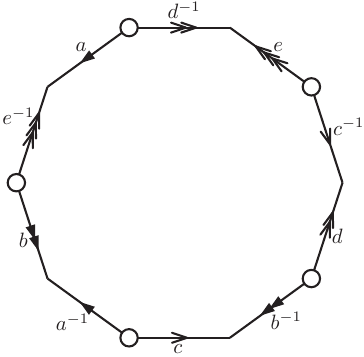}
\centering
\caption{Glueing rule}
\end{subfigure}
\begin{subfigure}[t]{0.33\textwidth}
\includegraphics[scale=0.85]{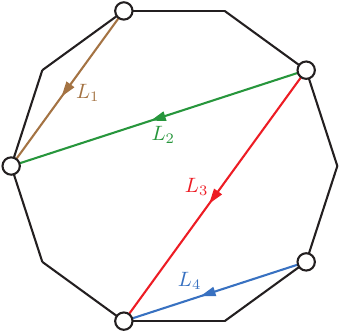}
\centering
\caption{$\overrightarrow{K}_{f} = (L_{1}, L_{3}, L_{2}, L_{4})$}
\end{subfigure}
\begin{subfigure}[t]{0.33\textwidth}
\includegraphics[scale=0.85]{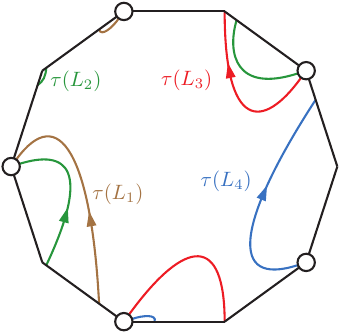}
\centering
\caption{Image of $\overrightarrow{K}_{f}$ under $\tau$}
\end{subfigure}
\centering
\caption{Fundamental domain of $A_{4}$ Milnor fiber}
\label{fig:A4fund}
\end{figure}
From Figures \ref{fig:A4} (B) and (C), one can check the following intersection conditions directly:
\begin{itemize}
\item Two distinct $L_{i}$ and $L_{j}$ are disjoint.
\item There are three intersecting pairs: $\tau(L_{1}) \cap L_{2} = \{ p_{1} \}$, $\tau(L_{2}) \cap L_{3} = \{ p_{2} \}$, and $\tau(L_{2}) \cap L_{4} = \{ p_{3} \}$.
\end{itemize}
One can perform a similar computation as $E_{6}$ case in \ref{subsubsec:E6} to show that 
	\[ [p_{1}] \in HF^{*}_{\tau}(L_{1},L_{2}), [p_2]\in HF^{*}_{\tau}(L_{3},L_{2}), \mbox{ and } [p_3]\in HF^{*}_{\tau}(L_{3},L_{4})\] 
are the unique generators for each group. 

We describe other natural representatives for a later purpose. 
Let $\eta \in CW^*(L_{1}, L_{2})$ be the shortest Reeb chord, depicted in the Figure \ref{fig:partially} (A). 
We claim that $\eta$ and $p_1$ represent the same class. 
Indeed, there is an element $\gamma \in CW^*(\tau(L_{1}), L_{2})$, a short Reeb chord as in Figure \ref{fig:partially} (A). 
One can check directly that $m_1(\gamma) = p_1$ and $m_2(\mathcal{CO}^\tau_{L_1}(\Gamma), \gamma) = \eta$, which implies that $\eta$ and $p_1$ differ by a coboundary as desired.

\begin{figure}[h]
\begin{subfigure}[t]{0.43\textwidth}
\raisebox{5ex}{\includegraphics[scale=1]{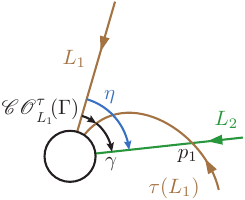}}
\centering
\caption{Another representative $\eta$}
\label{fig:partially1}
\end{subfigure}
\begin{subfigure}[t]{0.43\textwidth}
\includegraphics[scale=1]{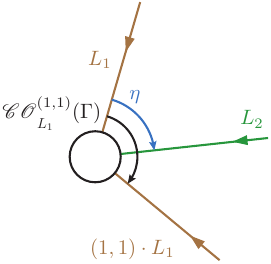}
\centering
\caption{$(1,1)$ monodromy case}
\label{fig:partially2}
\end{subfigure}
\centering
\caption{ }
\label{fig:partially}
\end{figure}

\subsection{Relation to partially wrapped Fukaya category}
We continue the computation in Example \ref{subsec:An} and exhibit an interesting connection with certain partially wrapped Fukaya categories.

Let $(D^2, \Lambda_{n+1})$ be a disc with $(n+1)$ stops on $\partial D^2$.
Its partially wrapped Fukaya category $\mathcal{WF}(D^2, \Lambda_{n+1})$ is derived equivalent to $D^b(\mathrm{Rep}Q_n)$, the derived category of representations of $A_n$ quiver $Q_n$. 
Moreover, a choice of isotopy class of full arc system on $(D^2,\Lambda_{n+1})$ determines a set of generators of $\mathcal{WF}(D^2,\Lambda_{n+1})$. 
We refer to \cite{HKK17} for more detailed computation. 

First, observe that $M_f$ has a two-fold covering given by $x: M_f \to \mathbb C$ ramified at $x = e^\frac{2k \pi i}{n+1}, \hspace{5pt} k=0, \ldots ,n$. 
We realize it as a Riemann surface of $\sqrt{1-x^{n+1}}$, where the branch cut at $e^\frac{2k \pi i}{n+1}$ is chosen as $\left\{e^{r+\frac{2k\pi i}{n+1}}\;\middle|\; r\geq 0\right\}$.
By choosing two different sheets of the cover, we can split $M_f$ into two pieces 
	$$ M_f = M_1\cup M_2, \hspace{5pt} M_1\cap M_2 = \partial M_1 = \partial M_2.$$
Moreover, $(D^2, \Lambda_{n+1})$ sits inside $M_1$ as the unit disc inside the $x$-plane. 
Then $M_1$ is equal to the closure of $x$-plane minus the union of branch cuts, which is again equal to the Liouville sector given by a completion of $D^2\setminus\Lambda_{n+1}$. 
(In fact the line field $\eta$ on $(D^2, \Lambda_{n+1})$, an extra data needed to put a grading on $\mathcal{WF}(D^2, \Lambda_{n+1})$, is determined by this covering.)

Second, observe that $M_f$ has an algebraic $\Z/(n+1) \times \Z/2$ action.  
An element $(k,l) \in \Z/(n+1) \times \Z/2$ acts as $(x, y) \mapsto (e^\frac{2k \pi i}{n+1} x, (-1)^ly)$. 
Since we choose $M_i$ as different sheets of the covering,  $(k,1) \in \Z/(n+1) \times \Z/2$ switches $M_1$ and $M_2$. 
We remark that the element $(1,1)\in \Z/(n+1) \times \Z/2$ is a representative of the monodromy up to positive boundary twist. 
Therefore, we can say that $M_2$ is the image of $M_1$ under the action of the monodromy. 

Using this observation, any full arc system $\{\alpha_i\}$ for $(D^2, \Lambda_{n+1})$ can be viewed as a collection of objects $\{L_i\}$ of $\mathcal{WF}(M_f)$ by conical extension. 
It is clearer from the pictorial description of $M_f$. 
It is described as a boundary gluing of  $2(n+1)$-gon with $(n+1)$ puncture at each vertex for every two.  
$M_1$ is the sub $(n+1)$-gon whose vertices consist of punctures. 
The Figure \ref{fig:partially3} below describes what is happening in $A_4$-case.

\begin{figure}[h]
\includegraphics[scale=0.8]{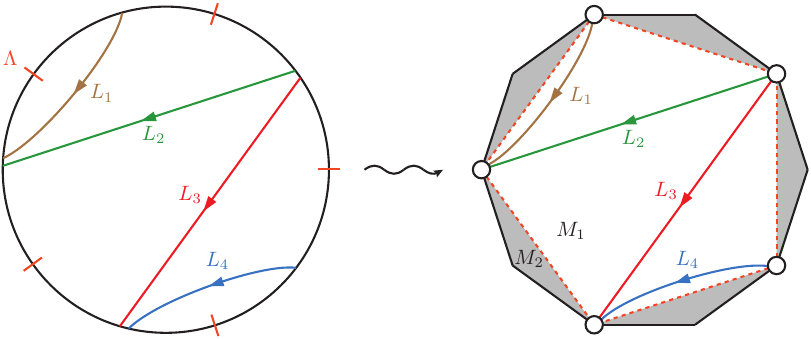}
\centering
\caption{The partially wrapped Fukaya category of $(D^2, \Lambda_{5})$ and the $A_4$ Milnor fiber}
\label{fig:partially3}
\end{figure}

\begin{prop} \label{prop:partially}
The canonical map $\hom^*_{\mathcal{WF}(D^2,\Lambda_{n+1})}(\alpha_i, \alpha_j) \hookrightarrow CW^*(L_i, L_j) \hookrightarrow CF_\rho^*(L_i, L_j)$ induces an isomorphism on cohomologies
\[ \Hom^*_{\mathcal{WF}(D^2,\Lambda_{n+1})}(\alpha_i, \alpha_j) \cong HF_\rho^*(L_i,L_j).\]

\end{prop}
\begin{proof}
We use the algebraic action of $(1,1)$ as a representative of the monodromy. 
Since $\alpha_i$ does not intersect to each other and $(1,1)$ shifts $M_1$ and $M_2$, we conclude that $\{L_i\}\cup\{(1,1)(L_i)\}$  are mutually disjoint inside $M_f$. 
Also, remark that $(1,1)$ has no fixed points by definition. 
Therefore, if $L$ and $(1,1)(L)$ does not intersect each other, then $\mathcal{CO}^{(1,1)}_L(\Gamma) \in CW^*(L, (1,1)(L))$ is equal to the shortest Reeb chord $\Gamma_L$ from $L$ to $(1,1)(L)$. 

Suppose the left-hand side is zero, which means that the ends of $\alpha_i$ and $\alpha_j$ do not share the same puncture. 
Then any Reeb chord from $L_i$ to $L_j$ is decomposed as a concatenation of $\mathcal{CO}^{(1,1)}_{L_1}(\Gamma)$ with other shorter Reeb chord from $(1,1)(L_i)$ to $L_j$. 
Using the argument similar to (2) of Lemma \ref{lem:i}, we conclude
\[HF^{*}_{(1,1)}(L_i, L_j) \cong H^*\left(CW^{*}((1,1)(L_i), L_j)[1] \xrightarrow{m_2(\mathcal{CO}^{(1,1)}_{L_1}(\Gamma){,}-)} CW^{*}(L_i, L_j)\right)\cong 0.\]

Suppose the left-hand side is non-zero. 
It means that the ends of $\alpha_i$ and $\alpha_j$ share a puncture, which is at most one because $\{\alpha_i\}$ forms a full arc system on a disc. 
When it is the case, $\hom^*_{\mathcal{W}(D^2,\Lambda_{n+1})}(\alpha_i, \alpha_j)$ is generated by a single Reeb chord $\eta_{i, j}$ from $\alpha_i$ to $\alpha_j$. 
This chord is the only Reeb chord from $L_i$ to $L_j$ trapped inside $M_1$.
If not, the chord $\eta'$ must pass through $M_2$ more than once, which implies such chords are written as a concatenation $\mathcal{CO}^{(1,1)}_{L_1}(\Gamma)$ with other shorter Reeb chord from $(1,1)(L_i)$ to $L_j$ as before: see Figure \ref{fig:partially} (B).
Using the argument similar to (1) of Lemma \ref{lem:i} and the discussion in Example \ref{subsec:An}, we conclude $HF^{*}_{(1,1)}(K_i, K_j)$ is also generated by a single class $[\eta_{ij}]$. 
\end{proof}

\subsection{Divides of higher depth}
\label{subsec:depth 1}
So far, we have discussed the cases of depth 0. We expect our construction  of Subsection \ref{subsec:depth0}
to be generalized to higher depth cases.  We will illustrate an idea of representing vanishing cycles of higher depth using twisted complexes built out of non-compact Lagrangians. But unfortunately, we would need the conjectural
$A_{\infty}$-category $\mathcal{C}_{\tau}$ and the twisted complexes there of. Hence, we will
consider a specific depth 1 example in detail, and leave the general construction for the future.

First, observe that for the classical variation operator \eqref{eq:var}, the inverse image of a vanishing cycle could be
given by a disjoint union of several non-compact Lagrangians. But if we take the direct sum of these Lagrangians, then its endomorphism algebra would not be local, but that of vanishing cycles in Fukaya-Seidel category is local.

Instead, we expect that we can form a twisted complex built from1214 these non-compact Lagrangians.
Let us illustrate this in the following example.

\begin{example}
Let $f:\C^2 \to \C$ be a polynomial
$$f(x,y) = x^{5} - x^{3}y^{2} + x^{2}y^{2} + y^{4} = (x^{3}+y^{2})(x+y)(x-y).$$ Its divide and $A\Gamma$ diagram are given  in Figure \ref{fig:depth1div}. There are 10 vanishing cycles which are ordered as in  Figure \ref{fig:depth1div} (B).
Its Milnor fiber is a Riemann surface with 4 genus and 3 boundary components and a monodromy $\tau$ is given by the composition of 10 Dehn twists.

\begin{figure}[h]
\begin{subfigure}[t]{0.43\textwidth}
\includegraphics[scale=1]{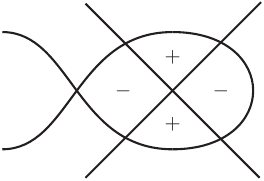}
\centering
\caption{Divide of $f$}
\end{subfigure}
\begin{subfigure}[t]{0.43\textwidth}
\includegraphics[scale=1]{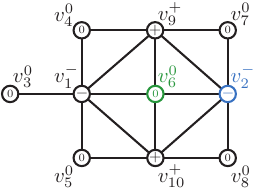}
\centering
\caption{$A\Gamma$ diagram of $(A)$}
\end{subfigure}
\centering
\caption{ }
\label{fig:depth1div}
\end{figure}
The middle vertex $v^{0}_{6}$
  is the unique depth $1$ vertex. Let $V^{0}_{6}$ be the corresponding vanishing cycle. Topologically, we want to find $[K_6^0]$ with $\mathrm{var}([K_6^0]) = - [V^{0}_6]$.  We start by choosing any depth $0$ vertex $w$ adjacent to $v^0_6$.
  In our case, we choose $w =  v^{-}_{2}$ in Figure \ref{fig:depth1div} (B).
  
   A neighborhood of $V^{0}_{6}$ can be identified with  that of vanishing cycle $V^0_i$ in Figure \ref{fig:K0} (A).
 In the case of depth 0, we had a  neighboring unbounded region, and the corresponding unnecessary arc (which was omitted in Figure \ref{fig:K0} (A)). We suppose that this unnecessary arc is now a part of the vanishing cycle for $v^{-}_{2}$ in the depth 1 case. We define $K^{\prime}_{6}$ to be $K^0_i$ in Figure \ref{fig:K0} (A)) under this identification. Then, we have $$\mathrm{var}([K^{\prime}_{6}]) = -[V^{0}_{6}] + [V^{-}_{2}].$$

Since the vertex $v^{-}_{2}$ has depth $0$, we can choose 
$K^{-}_{2}$ as in Lemma \ref{lem:depth0} and we still have $\mathrm{var}([K^-_2]) = -[V^{-}_{2}]$.
Therefore, we found two disjoint non-compact Lagrangians $K^{\prime}_{6}$ and $K^-_2$ such that
   $$\mathrm{var}([K^{\prime}_{6}]+[K^-_2] ) = -[V^{0}_{6}].$$
This together with depth 0 construction defines an adapted family.

Now, we explain how to define a twisted complex in this case.
By our construction, we  have an intersection point $p$ between $\tau(K^{\prime}_{6})$ and $K^{-}_{2}$:
this is because $\tau_{V^\prime_6}(K^{\prime}_{6})$  intersects the vanishing cycle $V^{-}_2$, and
thus $\tau_{V^-_2}(\tau_{V^\prime_6}(K^{\prime}_{6}))$ now intersects $K^-_2$. No other vanishing cycles are involved in this computation. The intersection point $p = \tau(K^{\prime}_{6})  \cap K^{-}_{2}$ defines a closed morphism in
$HF_\rho^*(K_6^\prime, K_2^-).$

 We define $K^{0}_{6}$ as a twisted complex of this morphism:
$$K^{\prime}_{6} \overset{p \epsilon}{\longrightarrow} K^{-}_{2}.$$

\begin{figure}[h]
\includegraphics[scale=1]{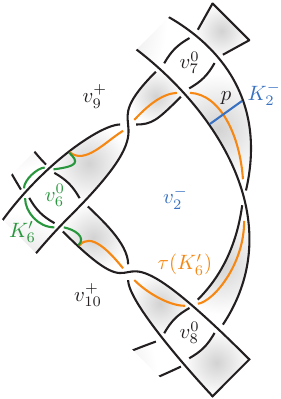}
\centering
\caption{Non-compact Lagrangians and the intersection point in the Milnor fiber}
\label{fig:depth1ex}
\end{figure}

We expect the following to hold in $\mathcal C_\rho$.
\begin{itemize}
\item $\{ K^{\bullet}_{i} \}_{i=1}^{10}$ forms an exceptional collection, and
\item the endomorphism algebra of them is isomorphic to that of corresponding vanishing cycles.
\end{itemize}

\end{example}

At least the first topological part of the construction in this example can be carried out for any depth,
which will appear elsewhere. 
\begin{thm}
For any A'Campo divide $\mathbb{D}_{f}$ and a distinguished collection of vanishing cycles  $\overrightarrow{V}_{f}$ associated with it, there exists an adapted family of non-compact Lagrangians
$$\overrightarrow{K}_{f} = (K^{-}_{1}, \dots, K^{-}_{n_{-}}, K^{0}_{1}, \dots, K^{0}_{n_{0}}, \dots, K^{+}_{1}, \dots, K^{+}_{n_{+}}).$$
Here for a depth $l$ vertex, the corresponding non-compact Lagrangian consists of $l+1$ disjoint connected
components.
\end{thm}

\bibliographystyle{amsalpha}
\bibliography{FukayaSing}

\end{document}